\documentclass[11pt,twoside, leqno]{article}

\usepackage{amssymb}
\usepackage{amsmath}
\usepackage{mathrsfs}
\usepackage{amsthm}
\usepackage{amsfonts}
\usepackage{enumerate}
\usepackage{esint}
\usepackage{latexsym}

\allowdisplaybreaks

\def\ls{\lesssim}
\def\gs{\gtrsim}
\def\fz{\infty}

\def\r{\right}
\def\lf{\left}

\def\supp{{\mathop\mathrm{\,supp\,}}}

\def\aa{{\mathbb A}}
\def\rr{{\mathbb R}}
\def\rh{{\mathbb R}{\mathbb H}}
\def\rn{{{\rr}^n}}
\def\zz{{\mathbb Z}}
\def\nn{{\mathbb N}}
\def\cc{{\mathbb C}}

\newcommand{\wz}{\widetilde}
\newcommand{\oz}{\overline}

\newcommand{\ca}{{\mathcal A}}

\newcommand{\cm}{{\mathcal M}}
\newcommand{\cn}{{\mathcal N}}

\newcommand{\ccr}{{\mathcal R}}
\newcommand{\cs}{{\mathcal S}}
\newcommand{\cx}{{\mathcal X}}

\def\cl{{\mathcal L}}
\def\az{\alpha}
\def\lz{\lambda}
\def\blz{\Lambda}
\def\bdz{\Delta}
\def\bfai{\Phi}

\def\epz{\epsilon}

\def\bz{\beta}

\def\fai{\varphi}
\def\gz{{\gamma}}
\def\bgz{{\Gamma}}

\def\tz{\theta}
\def\sz{\sigma}

\def\wz{\widetilde}

\def\ls{\lesssim}
\def\gs{\gtrsim}

\def\boz{\Omega}
\def\oz{\omega}

\def\uc{{\varepsilon}}

\def\pat{\partial}

\def\esup{\mathop\mathrm{\,esssup\,}}

\def\at{{{\mathop\mathrm{at}}}}

\def\hs{\hspace{0.3cm}}

\def\dsum{\displaystyle\sum}
\def\dint{\displaystyle\int}

\def\dsup{\displaystyle\sup}
\def\dinf{\displaystyle\inf}

\newtheorem{theorem}{Theorem}[section]
\newtheorem{lemma}[theorem]{Lemma}
\newtheorem{corollary}[theorem]{Corollary}
\newtheorem{proposition}[theorem]{Proposition}
\theoremstyle{definition}
\newtheorem{remark}[theorem]{Remark}
\newtheorem{definition}[theorem]{Definition}
\numberwithin{equation}{section}
\def\supp{{\mathop\mathrm{\,supp\,}}}

\def\loc{{\mathop\mathrm{loc\,}}}
\def\lfz{\lfloor}
\def\rfz{\rfloor}
\numberwithin{equation}{section}

\begin{document}

\arraycolsep=1pt

\title{\bf\Large Musielak-Orlicz-Hardy Spaces Associated with Operators Satisfying Reinforced
Off-Diagonal Estimates\footnotetext {\hspace{-0.35cm}
2010 {\it Mathematics Subject Classification}. Primary: 42B35;
Secondary: 42B30, 42B25, 42B20, 35J10, 46E30, 47B38, 47B06, 30L99.
\endgraf {\it Key words and phrases}.
Musielak-Orlicz-Hardy space, molecule, atom,
maximal function, Lusin area function, Schr\"odinger operator, elliptic operator,
Riesz transform.
\endgraf Dachun Yang is supported by the National
Natural Science Foundation (Grant No. 11171027) of China and
Program for Changjiang Scholars and Innovative Research Team in
University of China.}}
\author{The Anh Bui, Jun Cao, Luong Dang Ky, Dachun Yang\footnote{Corresponding author}\,
\ and Sibei Yang}
\date{ }
\maketitle

\vspace{-0.5cm}

\begin{center}
\begin{minipage}{13.9cm}
{\small {\bf Abstract}\quad Let $\mathcal{X}$ be a metric space
with doubling measure and $L$ a one-to-one operator of
type $\omega$ having a bounded $H_\infty$-functional calculus
in $L^2(\mathcal{X})$ satisfying the reinforced $(p_L, q_L)$
off-diagonal estimates on balls,
where $p_L\in[1,2)$ and $q_L\in(2,\infty]$.
Let $\varphi:\,\mathcal{X}\times[0,\infty)\to[0,\infty)$ be a
function such that $\varphi(x,\cdot)$ is an Orlicz function,
$\varphi(\cdot,t)\in {\mathbb A}_{\infty}(\mathcal{X})$ (the class of
uniformly Muckenhoupt weights), its uniformly critical upper type index
$I(\varphi)\in(0,1]$ and $\varphi(\cdot,t)$ satisfies the uniformly
reverse H\"older inequality of order $(q_L/I(\varphi))'$, where $(q_L/I(\varphi))'$ denotes
the conjugate exponent of $q_L/I(\varphi)$. In this paper, the authors
introduce a Musielak-Orlicz-Hardy space $H_{\varphi,\,L}(\mathcal{X})$,
via the Lusin-area function associated with $L$,
and establish its molecular characterization. In particular, when $L$ is nonnegative
self-adjoint and satisfies the Davies-Gaffney estimates, the atomic characterization
of $H_{\varphi,\,L}(\mathcal{X})$ is also obtained. Furthermore, a sufficient condition
for the equivalence between $H_{\varphi,\,L}(\mathbb{R}^n)$ and the
classical Musielak-Orlicz-Hardy
space $H_{\varphi}(\mathbb{R}^n)$ is given. Moreover, for the
Musielak-Orlicz-Hardy space $H_{\varphi,\,L}(\mathbb{R}^n)$
associated with the second order elliptic operator
in divergence form on $\rn$ or
the Schr\"odinger operator $L:=-\Delta+V$ with $0\le V\in
L^1_{\mathrm{loc}}(\mathbb{R}^n)$,
the authors further obtain its several equivalent
characterizations in terms of various non-tangential and radial maximal
functions; finally, the authors show that the Riesz
transform $\nabla L^{-1/2}$ is bounded from
$H_{\varphi,\,L}(\mathbb{R}^n)$ to the Musielak-Orlicz space $L^\varphi(\mathbb{R}^n)$
when $i(\varphi)\in(0,1]$, from $H_{\varphi,\,L}(\mathbb{R}^n)$ to
$H_{\varphi}(\mathbb{R}^n)$ when $i(\varphi)\in(\frac{n}{n+1},1]$, and from
$H_{\varphi,\,L}(\mathbb{R}^n)$ to the weak Musielak-Orlicz-Hardy space
$WH_{\varphi}(\mathbb{R}^n)$ when $i(\fai)=\frac{n}{n+1}$ is attainable and
$\varphi(\cdot,t)\in {\mathbb A}_1(\mathcal{X})$,
 where $i(\varphi)$ denotes the uniformly critical lower type index of $\varphi$.}
\end{minipage}
\end{center}


\section{Introduction\label{s1}}

\hskip\parindent The real-variable theory of Hardy spaces $H^p(\mathbb{R}^n)$ with
$p\in(0,1]$, introduced by Stein and Weiss \cite{sw60} and systematically
developed in the seminal paper of Fefferman and Stein
\cite{fs72}, plays a center role in various fields of harmonic analysis and
partial differential equations (see, for example, \cite{clms,s94} and the references
therein). One of the main features of the Hardy space
$H^p(\mathbb{R}^n)$ with $p\in(0,1]$ is their atomic decomposition characterizations
(see \cite{c74} for $n=1$ and \cite{l78} for $n>1$). Later, the theory of weighted
Hardy spaces $H^p_w(\mathbb{R}^n)$ with Muckenhoupt weights $w$ has been
studied by Garc\'ia-Cuerva \cite{g79}, and Str\"omberg and Torchinsky \cite{st}. Furthermore,
Str\"omberg \cite{s79} and  Janson \cite{ja80} introduced the Orlicz-Hardy space
which play an important role in studying the theory of nonlinear PDEs (see, for example,
\cite{GI,IS,bfg10,bg10,bijz07}). Recently, in \cite{k},
the last two authors of the present paper studied
Hardy spaces of Musielak-Orlicz type which generalize the Orlicz-Hardy space
in \cite{s79,ja80} and the weighted Hardy spaces in \cite{g79,st}. Furthermore,
several real-variable characterizations of the Hardy spaces of
Musielak-Orlicz type were established in \cite{lhy,hyy}. Moreover, the local Hardy space
of Musielak-Orlicz type was studied in \cite{yys3}.
It is worth pointing out that Musielak-Orlicz functions are
the natural generalization of Orlicz functions
(see, for example, \cite{d05,dhr09,k,m83})
and the motivation to study function spaces of
Musielak-Orlicz type is attributed to their extensive applications to
many fields of mathematics (see, for example,
\cite{bfg10,bg10,bgk,bijz07,d05,dhr09,k,k1,l05} for more details).
However, it is now understood that
there are many settings in which the theory of the spaces of Hardy type can not be
applicable; for example, the Riesz transform $\nabla L^{-1/2}$ may not be bounded from
$H^1(\mathbb R^n)$ to $L^1(\mathbb R^n)$ when $L:=-{\rm div} (A\nabla)$ is a second order
divergence elliptic operator with complex bounded measurable coefficients (see, for example,
\cite{hm09}).

Recently, there has been a lot of studies which pay attention to the theory of function
spaces associated with operators. In many applications, the very dependence on the
function spaces associated with the operators provides many advantages in studying the boundedness
of singular integrals which may not fall within the scope of the classical Calder\'on-Zygmund theory.
Here, we would like to give a brief overview of this research direction.
Let $L$ be an infinitesimal generator of an analytic semigroup $\{e^{-tL}\}_{t>0}$ on
$L^2(\mathbb{R}^n)$ whose kernels satisfy the Gaussian upper bound estimates.
The theory on Hardy spaces associated with such operators $L$ was investigated
in \cite{adm,dy05}. Later, Hardy spaces associated with operators which satisfy the weaker
conditions, so called Davies-Gaffney estimate conditions, were treated in the works of
Auscher et al. \cite{amr08}, Hofmann and Mayboroda \cite{hm09} and
Hofmann et al. \cite{hlmmy,hmm}.
In \cite{jy10, jy11,jyz09,lyy11,yys1, yys2, yys3,yys4}, the authors studied the
Orlicz-Hardy spaces associated with
operators and, in some sense, these results are extensions to  Hardy spaces associated
with operators. Then, the weighted Hardy spaces associated with operators were
also considered in \cite{sy10} and \cite{bd11}. Recently, in \cite{yys4}, the
last two authors of this paper studied the Musielak-Orlicz-Hardy spaces associated with
nonnegative self-adjoint operators satisfying Davies-Gaffney estimates.
Furthermore, some special Musielak-Orlicz-Hardy spaces  associated with the Schr\"odinger
operator $L:=-\Delta+V$ on $\rn$, where the nonnegative
potential $V$ satisfies the reverse H\"older inequality of order $n/2$,
were studied by the third author of this paper \cite{k3,k4,k5} and further applied to
the study of commutators of singular integral operators associated with the operator $L$.
Very recently, the authors of this paper \cite{bckyy} studied the weighted Hardy space
associated with nonnegative self-adjoint operators satisfying the reinforced off-diagonal
estimates on $\rn$ (see Assumption (B) for their definitions in the present setting),
which improves these results in \cite{sy10,bd11,yys4} in some sense by
essentially extending the range of the considered weights.

We would like to describe partly the results in \cite{yys4} which may be closely related
to this paper. Let $L$ be a nonnegative self-adjoint operator on $L^2(\cx)$ satisfying
Davies-Gaffney estimates, where $\cx$ denotes a metric space with doubling measure.
Let $\varphi : \cx\times [0,\infty)\to [0,\infty)$ be a function such that
$\varphi(x, \cdot)$ is an Orlicz function, $\varphi(\cdot, t)\in\mathbb{A}_\infty
(\mathcal{X})$ (the class of uniformly Muckenhoupt weights), its uniformly critical
upper type index $I(\varphi) \in (0, 1]$ and $\varphi(\cdot, t)$ satisfies the uniformly reverse
H\"older inequality of order $2/[2-I(\varphi)]$ (see Subsection \ref{s2.2} below
for these definitions). A typical example of such a $\fai$ is
\begin{equation}\label{1.1}
\fai(x,t):=w(x)\Phi(t)
\end{equation}
for all $x\in\cx$ and $t\in[0,\fz)$, where $w\in A_{\fz}(\cx)$ (the \emph{class of
Muckenhoupt weights}) and $\Phi$ is an Orlicz function on $[0,\fz)$ of
upper type 1 and lower type $p\in(0,1]$ (see Section \ref{s2.2} below
for the definition of types). Let $x_0\in\cx$. Another typical
and useful example of such a $\fai$ is
\begin{equation}\label{1.2}
\fai(x,t):=\frac{t^{\az}}{[\ln(e+d(x,x_0))]^{\bz}+[\ln(e+t)]^{\gz}}
\end{equation}
for all $x\in\cx$ and $t\in[0,\fz)$, with some $\az\in(0,1]$,
$\bz\in[0,n)$ and $\gz\in [0,2\az(1+\ln2)]$ (see
Section \ref{s2.2} for more details). Then, the last two authors of the present paper
\cite{yys4} introduced a Musielak-Orlicz-Hardy space $H_{\varphi,\,L}(\mathcal{X})$,
via the Lusin-area function associated with $L$, and obtained two
equivalent characterizations of $H_{\varphi,\,L}
(\mathcal{X})$ in terms of the atom and the molecule. Hence, it is natural
to raise the question \emph{when the condition
$\varphi(x, \cdot) \in \mathbb{RH}_{2/[2 - I(\varphi)]}
(\mathcal{X})$ can be relaxed}. One of  the main aims of this paper is to give
an affirmative answer to this question.

Moreover, motivated by \cite{bckyy,yys4,am07ii,cy}, in this paper,
we consider more general operators by assuming that the considered operator satisfies
\textbf{Assumptions (A)} and \textbf{(B)} in Subsection \ref{s2.3} of this paper.
Indeed, Assumption {\rm (A)} is weaker than ``the nonnegative and self-adjoint''
condition imposing on the operator $L$ in \cite{yys4}. Meanwhile, in Assumption (B), we first introduce
the notion of the \emph{reinforced $(p_L, q_L,m)$ off-diagonal estimates on balls}
in the spaces of homogeneous type (see Definition \ref{d2.3} below), which is quite wide
so that it can provide a framework  to treat  almost the results in previous works
(see, for example, \cite{adm, dy05, hm09, hlmmy, jy10, jy11, yys4, bckyy}).
Under Assumptions (A) and (B),
we first introduce the Musielak-Orlicz-Hardy spaces $H_{\varphi,\,L}(\mathcal{X})$
(see Definition \ref{d4.1} below), via
the Lusin-area function associated with L, and then
characterize the spaces of $H_{\varphi,\,L}(\mathcal{X})$ in terms of the molecular
with $\varphi(x, \cdot) \in \mathbb{RH}_{(q_L/I(\varphi))'}(\mathcal{X})$
(see Theorem \ref{t4.1} below), where
$(q_L/I(\varphi))'$ denotes the conjugate exponent of $q_L/I(\varphi)$.
In particular, when $L$ is nonnegative self-adjoint and satisfies the Davies-Gaffney
estimates, the atomic characterization of $H_{\varphi,\,L}(\mathcal{X})$ is
also obtained (see Theorem \ref{ct4.3} below).
It is important to notice that $\mathbb{RH}_{2/[2 - I(\varphi)]}(\mathcal{X})
\subset \mathbb{RH}_{(q_L/I(\varphi))'}(\mathcal{X})$ whenever $q_L>2$ and hence the
results in this paper improve significantly those in \cite{yys4}, by enlarging the range of the
weights. In particular case, if
the heat kernels associated with $\{e^{-tL}\}_{t>0}$ satisfy the Gaussian upper
bound estimate, then $p_L=1$ and $q_L=\infty$ and hence the class of $\varphi$ can be extended
to $\varphi(x, \cdot) \in \mathbb{A}_{\infty}(\mathcal{X})$. Moreover, we
also give a sufficient conditions on $L$ so that our Musielak-Orlicz-Hardy space
$H_{\varphi,\, L}(\mathcal{X})$ coincides with the Musielak-Orlicz-Hardy space
$H_{\varphi}(\mathcal{X})$ introduced by the third author of this paper in \cite{k}
when $\mathcal{X}:=\mathbb{R}^n$ (see Theorem \ref{t5.1} below).
As applications, we consider Musielak-Orlicz-Hardy spaces $H_{\varphi,\,L}(\mathcal{X})$
in some particular cases, for example, $L$ being the second order elliptic operator in
divergence form or the Schr\"odinger operator. More precisely, for the
Musielak-Orlicz-Hardy space $H_{\varphi,\,L}(\mathbb{R}^n)$
associated with the second order elliptic operator
in divergence form on $\rn$ with bounded measurable complex coefficients or
the Schr\"odinger operator $L:=-\Delta+V$, where $0\le V\in L^1_{\mathrm{loc}}(\mathbb{R}^n)$,
we further obtain its several equivalent
characterizations in terms of the non-tangential and the radial maximal
functions (see Theorems \ref{ct5.5} and \ref{t6.1} below); finally, we show that the Riesz
transform $\nabla L^{-1/2}$ is bounded from
$H_{\varphi,\,L}(\mathbb{R}^n)$ to the Musielak-Orlicz space $L^\varphi(\mathbb{R}^n)$
when $i(\varphi)\in(0,1]$, from $H_{\varphi,\,L}(\mathbb{R}^n)$ to
$H_{\varphi}(\mathbb{R}^n)$ when $i(\varphi)\in(\frac{n}{n+1},1]$, and from
$H_{\varphi,\,L}(\mathbb{R}^n)$ to the weak Musielak-Orlicz-Hardy space
$WH_{\varphi}(\mathbb{R}^n)$ when $i(\fai)=\frac{n}{n+1}$ is attainable
and $\varphi(\cdot,t)\in {\mathbb A}_{1}(\mathcal{X})$ (see Theorems
\ref{ct5.7}, \ref{ct5.12}, \ref{t6.2} and \ref{t6.3} below),
where $i(\varphi)$ denotes the uniformly critical lower type index of $\varphi$.

One of the new ingredients appeared in this paper is the introduction of the notion of
the reinforced $(p_L, q_L,m)$ off-diagonal estimates on balls
in spaces of homogeneous type with $m\in\nn:=\{1,\,2,\,\ldots\}$.
We remark that, to study the weighted Hardy space
$H_{\omega}^p(\rn)$ on the Euclidean space $\rn$ and to relax the range of the weight
$\omega$ as wide as possible,  the authors introduced a notion of
the reinforced $(p_L, q_L)$ off-diagonal estimates in \cite{bckyy}, which
is particular useful for studying the weighted Hardy space associated with
various differential operators of second order in the
setting of Euclidean spaces. However, if we consider the differential operators
on some more general spaces (for example, the Laplace-Beltrami operator on the Riemannian
manifold with doubling property), the reinforced $(p_L, q_L)$ off-diagonal estimates
in \cite{bckyy} seem no longer suitable (see Remark \ref{r2.4}(a)). To overcome this difficulty,
we introduce the reinforced $(p_L, q_L,m)$ off-diagonal estimates on balls by combining
the ideas of  the reinforced $(p_L, q_L)$ off-diagonal estimates from \cite{bckyy} and
the off-diagonal estimates on balls from \cite{am07ii}. Also, the order $m\in\nn$
makes many differential operators of higher order fall into our scope
(see Remark \ref{r2.4}(c)). We also point out that, in \cite{hlmmy,jy11}, the authors
introduced a Hardy space associated with operators $L$ in the space of homogenous
type by assuming that $L$ satisfies the so called Davies-Gaffney estimates. However,
due to the fact that Davies-Gaffney estimates are equivalent to $L^2$-$L^2$ off-diagonal
estimates on balls (see Remark \ref{equivalence between off-diagonal estimates
on balls and off-diagonal estimates}(ii)), their Hardy spaces can be viewed as a
special case of ours.

Another interesting ingredient appeared in this paper is the discussion of
the role of the $L^2(\cx)$ norm in the definitions of the Hardy space
$H_{\fai,\,L}(\cx)$ and  the atomic or the molecular Hardy space.
This discussion has two aspects, the first one is from \cite{bckyy}, where
the authors asked the question that what happen if we replace  $L^2(\mathbb R^n)$
by $L^q(\mathbb R^n)$ with $q\ne 2$ in the definition of the Hardy space. For
this question, in the present setting, we prove that the space $H_{\fai,\,L}(\cx)$
is invariant when we do this replacement for all $q\in (p_L,\,q_L)$
(see Theorem \ref{t4.2}), which coincides
with the result obtained in \cite{bckyy} when $\fai(x,\,t):=t^p w(x)$, for $p\in(0,\,1]$.
The second aspect of this discussion can be reduced to the following
question: ``\emph{what happen if we replace the  $L^2(\mathbb R^n)$-convergence of the
atomic (resp. molecular) representation by the $L^s(\mathbb R^n)$-convergence
with $s\ne 2$ in the definition of the the atomic and molecular Hardy space?}"
This question arises naturally when we study the boundedness of the fractional integral
between two different Hardy spaces. For this question, we prove that the atomic
and molecular Hardy spaces are invariant when we do this replacement
for all $s\in  (p_L,\,q_L)$ (see Theorems \ref{ct4.9} and \ref{t4.1}).

The organization of this paper is as follows. In Section \ref{s2}, we discuss the
settings which are considered in this papers. This includes the assumptions
for the function $\varphi$ and the operator $L$. Then, we establish the results
on the $L^p(\cx)$-boundedness of two square functions which is useful in what follows.

Section \ref{s3} is dedicated to studying the Musielak-Orlicz tent spaces. Like the classical
result for the tent spaces, we also give out the atomic decomposition for the Musielak-Orlicz
tent spaces.

In Section \ref{ss4}, we first introduce the Musielak-Orlicz-Hardy
space $H_{\varphi,\,L}(\mathcal{X})$ via the Lusin-area function and prove
that the operator $\pi_{L,\,M}$ (see \eqref{4.2} below for its definition) maps
the Musielak-Orlicz tent space $T_\fai(\cx_+)$ continuously into the
space $H_{\fai,\,L}(\cx)$ (see Proposition \ref{p4.1} below), here and in
what follows, $\cx_+:=\cx\times (0,\fz)$. By this and the
atomic decomposition of the space $T_\fai(\cx_+)$, we establish the molecular
characterization of $H_{\fai,\,L}(\cx)$ (see Theorem \ref{t4.1} below).
Moreover, similar to \cite[Theorem 3.4]{bckyy}, we show that $H_{\fai,\,L}(\cx)$
is invariant if we replace $L^2(\cx)$ by $L^q(\cx)$ with $q\in (p_L,q_L)$
in the definition of $H_{\fai,\,L}(\cx)$ (see Theorem \ref{t4.2} below).
As a consequence, we see that $L^s(\cx)\cap H_{\fai,\,L}(\cx)$ is dense
in $H_{\fai,\,L}(\cx)$ whenever $s\in(p_L,q_L)$ (see Corollary \ref{cs4.1} below).

If $L$ is a nonnegative self-adjoint operator in $L^2(\cx)$ satisfying
the reinforced $(p_L,p_L',1)$ off-diagonal estimates on balls with $p_L\in[1,2)$,
in Section \ref{cs4}, we establish the atomic characterization of the space
$H_{\fai,\,L}(\cx)$ (see Theorem \ref{ct4.3} below) by using the finite
propagation speed for the wave equation and a similar method used in Section \ref{ss4}.

The aim of Section \ref{s5} is to give an affirmative answer
to the question ``\emph{when do the Musielak-Orlicz-Hardy
spaces $H_{\varphi,\,L}(\mathbb{R}^n)$ and $H_{\varphi}(\mathbb{R}^n)$ coincide?}''.
More precisely, if the distribution kernel of the heat semigroup $\{e^{-tL}\}_{t>0}$
satisfies the Gaussian upper bound estimate, some H\"older regularity and the conservation
(see Assumption (C) below for details), then the spaces $H_{\fai,\,L}(\rn)$ and
$H_{\fai}(\rn)$ coincide with equivalent quasi-norms (see Theorem \ref{t5.1} below).

In Section \ref{cs5}, as a special case, we further study the Musielak-Orlicz-Hardy
space $H_{\fai,\,L}(\rn)$ associated with the second order elliptic operator
in divergence on $\rn$ with complex bounded measurable coefficients. By making full use of
the special structure of the divergence form elliptic operator and
establishing a good-$\lz$ inequality concerning the non-tangential maximal function
and the truncated Lusin-area function, we obtain the radial and the non-tangential
maximal function characterizations of $H_{\fai,\,L}(\rn)$ (see Theorem \ref{ct5.5} below).
We remark that the proof of Theorem \ref{ct5.5} is similar to that of
\cite[Theorem 7.4]{yys4} (see also the proof of \cite[Proposition 3.2]{yys2}).
Theorem \ref{ct5.5} completely covers
\cite[Theorem 5.2 and Corollary 5.1]{jy10} by taking $\fai$ as in \eqref{1.1}
with $w\equiv1$ and $\Phi$ concave.
Moreover, we prove that the Riesz transform $\nabla L^{-1/2}$,
associated with $L$, is bounded from $H_{\varphi,\,L}(\mathbb{R}^n)$
to $L^\varphi(\mathbb{R}^n)$ when $i(\varphi)\in(0,1]$, from
$H_{\varphi,\,L}(\mathbb{R}^n)$ to $H_{\varphi}(\mathbb{R}^n)$ when
$i(\varphi)\in(\frac{n}{n+1},1]$, and from
$H_{\varphi,\,L}(\mathbb{R}^n)$ to the weak Musielak-Orlicz-Hardy space
$WH_{\varphi}(\mathbb{R}^n)$ when $i(\fai)=\frac{n}{n+1}$ and is attainable (see Theorems
\ref{ct5.7} and \ref{ct5.12} below). We point out that
Theorem \ref{ct5.7} completely covers \cite[Theorems 7.1 and 7.4]{jy10} by taking
completely cover \cite[Theorems 5.1 and 5.2]{jy11} by taking $\fai$ as in \eqref{1.1}
with $w\equiv1$ and $\Phi$ concave.
Theorem \ref{ct5.12} completely covers \cite[Theorem 1.2]{cyy} by
taking $\fai$ as in \eqref{1.1} with $w\equiv1$ and $\Phi(t):=t^{n/(n+1)}$
for all $t\in[0,\fz)$.

In Section \ref{s6}, we consider the Musielak-Orlicz-Hardy spaces
$H_{\varphi,\,L}(\mathbb{R}^n)$ associated with
the Schr\"odinger operator $L:=-\Delta+V$, where $0\le V\in L^1_\loc(\rn)$. Similar to
Section \ref{cs5}, we establish several equivalent characterizations of $H_{\fai,\,L}(\rn)$
in terms of the radial and the non-tangential maximal functions associated with the heat
and the Poisson semigroups of $L$ (see Theorem \ref{t6.1} below). Moreover, we also study
the boundedness of $\nabla L^{-1/2}$ on the space $H_{\fai,\,L}(\rn)$ (see Theorems \ref{t6.2}
and \ref{t6.3} below). It is worth pointing out that Theorems \ref{t6.1} and \ref{t6.2},
respectively, improve \cite[Theorem 7.4]{yys4} and \cite[Theorems 7.11 and 7.15]{yys4}
by extending the range of weights (see Remarks \ref{r6.1} and \ref{r6.2} below
for details).

Finally we make some conventions on notation. Throughout the whole
paper, we denote by $C$ a \emph{positive constant} which is
independent of the main parameters, but it may vary from line to
line. We also use $C_{(\gz,\bz,\ldots)}$ to denote a \emph{positive
constant depending on the indicated parameters $\gz$, $\bz$,
$\ldots$}. The \emph{symbol} $A\ls B$ means that $A\le CB$. If
$A\ls B$ and $B\ls A$, then we write $A\sim B$. The  \emph{symbol}
$\lfz s\rfz$ for $s\in\rr$ denotes the maximal integer not more
than $s$. For any given normed spaces $\mathcal A$ and $\mathcal
B$ with the corresponding norms $\|\cdot\|_{\mathcal A}$ and
$\|\cdot\|_{\mathcal B}$, the \emph{symbol} ${\mathcal
A}\subset{\mathcal B}$ means that for all $f\in \mathcal A$, then
$f\in\mathcal B$ and $\|f\|_{\mathcal B}\ls \|f\|_{\mathcal A}$.
Also given $\lambda > 0$, we write $\lambda B$ for the $\lambda$-dilated ball,
which is the ball with the same center as $B$ and with radius
$r_{\lambda B} = \lambda r_B$. We also set
$\nn:=\{1,\,2,\, \ldots\}$ and $\zz_+:=\{0\}\cup\nn$. For each
ball $B\subset \cx$, we set
$$S_0(B)=B\ \text{and}\ S_j(B) = 2^jB\setminus 2^{j-1}B$$
for $j\in \mathbb{N}.$
For any measurable subset $E$ of $\cx$, we denote by $E^\complement$
the \emph{set} $\cx\setminus E$ and by $\chi_{E}$ its
\emph{characteristic function}. For any
$\tz:=(\tz_{1},\ldots,\tz_{n})\in\zz_+^{n}$, let
$|\tz|:=\tz_{1}+\dots+\tz_{n}$. For any subsets
$E$, $F\subset\cx$ and $z\in\cx$, let
$$d(E, F):=\inf_{x\in E,\,y\in F}d(x,y)\ \text{and}\ d(z, E):=\inf_{x\in E}d(z,x).$$
For $1 \le q \le\infty$, we denote by $q'$ the \emph{conjugate exponent} of
 $q$, namely, $1/q + 1/q'= 1$. Finally, we use the notation
$$\fint_B h(x)d\mu(x):=\frac{1}{\mu(B)}\int_Bh(x)d\mu(x).$$

\section{Preliminaries\label{s2}}

\hskip\parindent In Subsection \ref{s2.1}, we first recall some notions on metric measure
spaces and then, in Subsection \ref{s2.2}, we state some notions and assumptions
concerning growth functions considered in this paper and give some examples
which satisfy these assumptions; finally, we recall some properties of growth
functions established in \cite{k}. In Subsection \ref{s2.3}, we describe some basic
assumptions on the operator $L$ studied in this paper and then study the
$L^p(\cx)$-boundedness of two square functions associated with $L$.

\subsection{Metric measure spaces\label{s2.1}}

\hskip\parindent Throughout the whole paper, we let $\cx$ be a
\emph{set}, $d$ a \emph{metric} on $\cx$ and $\mu$ a \emph{nonnegative Borel regular
measure} on $\cx$. For all $x\in\cx$ and $r\in(0,\fz)$, let
$$B(x,r):=\{y\in\cx:\ d(x,y)<r\}$$
and $V(x,r):=\mu(B(x,r))$.
Moreover, we assume that there exists a constant
$C\in[1,\fz)$ such that, for all $x\in\cx$ and $r\in(0,\fz)$,
\begin{equation}\label{2.1}
V(x,2r)\le CV(x,r)<\fz.
\end{equation}

Observe that $(\cx,\,d,\,\mu)$ is a \emph{space of homogeneous type} in
the sense of Coifman and Weiss \cite{cw71}. Recall that in the
definition of spaces of homogeneous type in \cite[Chapter 3]{cw71}, $d$ is assumed to be a
quasi-metric. However, for simplicity, we always assume that
$d$ is a metric. Notice that the doubling property
\eqref{2.1} implies that the following strong homogeneity property
that, for some positive constants $C$ and $n$,
\begin{equation}\label{2.2}
V(x,\lz r)\le C\lz^n V(x,r)
\end{equation}
uniformly for all
$\lz\in[1,\fz)$, $x\in\cx$ and $r\in(0,\fz)$. There also
exist constants $C\in(0,\fz)$ and $N\in[0,n]$ such that, for all
$x,\,y\in\cx$ and $r\in(0,\fz)$,
\begin{equation}\label{2.3}
V(x,r)\le C\lf[1+\frac{d(x,y)}{r}\r]^NV(y,r).
\end{equation}
Indeed, the property \eqref{2.3} with $N=n$ is a simple corollary
of the triangle inequality for the metric $d$ and the strong
homogeneity property \eqref{2.2}. In the cases of Euclidean spaces
and Lie groups of polynomial growth, $N$ can be chosen to be $0$.

Furthermore, for $p\in(0,\fz]$, the \emph{space of $p$-integrable
functions on $\cx$} is denoted by $L^p(\cx)$ and the \emph{(quasi-)norm} of
$f\in L^p(\cx)$ by $\|f\|_{L^p(\cx)}$.

\subsection{Growth functions\label{s2.2}}

\hskip\parindent Recall that a function
$\Phi:[0,\fz)\to[0,\fz)$ is called an \emph{Orlicz function} if it
is nondecreasing, $\Phi(0)=0$, $\Phi(t)>0$ for $t\in(0,\fz)$ and
$\lim_{t\to\fz}\Phi(t)=\fz$ (see, for example,
\cite{m83,rr91,rr00}). The function $\Phi$ is said to be of
\emph{upper type $p$} (resp. \emph{lower type $p$}) for some $p\in[0,\fz)$, if
there exists a positive constant $C$ such that for all
$s\in[1,\fz)$ (resp. $s\in[0,1]$) and $t\in[0,\fz)$,
$\Phi(st)\le Cs^p \Phi(t).$

For a given function $\fai:\,\cx\times[0,\fz)\to[0,\fz)$ such that for
any $x\in\cx$, $\fai(x,\cdot)$ is an Orlicz function,
$\fai$ is said to be of \emph{uniformly upper type $p$} (resp.
\emph{uniformly lower type $p$}) for some $p\in[0,\fz)$ if there
exists a positive constant $C$ such that for all $x\in\cx$,
$t\in[0,\fz)$ and $s\in[1,\fz)$ (resp. $s\in[0,1]$), $\fai(x,st)\le Cs^p\fai(x,t)$.
We say that $\fai$ is of \emph{positive uniformly upper type}
(resp. \emph{uniformly lower type}) if it is of uniformly upper
type (resp. uniformly lower type) $p$ for some $p\in(0,\fz)$. Moreover,
let
\begin{equation}\label{2.4}
I(\fai):=\inf\{p\in(0,\fz):\ \fai\ \text{is of uniformly upper
type}\ p\}
\end{equation}
and
\begin{equation}\label{2.5}
i(\fai):=\sup\{p\in(0,\fz):\ \fai\ \text{is of uniformly lower
type}\ p\}.
\end{equation}
In what follows, $I(\fai)$ and
$i(\fai)$ are, respectively, called the \emph{uniformly critical
upper type index} and the \emph{uniformly critical lower type index} of $\fai$.
Observe that $I(\fai)$ and $i(\fai)$ may not be attainable, namely, $\fai$ may not
be of uniformly upper type $I(\fai)$ and uniformly lower type $i(\fai)$
(see below for some examples).

Let $\fai:\cx\times[0,\fz)\to[0,\fz)$ satisfy that
$x\mapsto\fai(x,t)$ is measurable for all $t\in[0,\fz)$. Following
\cite{k}, $\fai(\cdot,t)$ is called \emph{uniformly locally
integrable} if, for all bounded sets $K$ in $\cx$,
$$\int_{K}\sup_{t\in(0,\fz)}\lf\{\fai(x,t)
\lf[\int_{K}\fai(y,t)\,d\mu(y)\r]^{-1}\r\}\,d\mu(x)<\fz.$$

\begin{definition}\label{d2.1}
Let $\fai:\cx\times[0,\fz)\to[0,\fz)$ be uniformly locally
integrable. The function $\fai(\cdot,t)$ is said to satisfy the
\emph{uniformly Muckenhoupt condition for some $q\in[1,\fz)$},
denoted by $\fai\in\aa_q(\cx)$, if, when $q\in (1,\fz)$,
\begin{equation*}
\aa_q (\fai):=\sup_{t\in
(0,\fz)}\sup_{B\subset\cx}\fint_B
\fai(x,t)\,d\mu(x) \lf\{\fint_B
[\fai(y,t)]^{-q'/q}\,d\mu(y)\r\}^{q/q'}<\fz,
\end{equation*}
where $1/q+1/q'=1$, or
\begin{equation*}
\aa_1 (\fai):=\sup_{t\in (0,\fz)}
\sup_{B\subset\cx}\fint_B \fai(x,t)\,d\mu(x)
\lf(\esup_{y\in B}[\fai(y,t)]^{-1}\r)<\fz.
\end{equation*}
Here the first supremums are taken over all $t\in(0,\fz)$ and the
second ones over all balls $B\subset\cx$.

The function $\fai(\cdot,t)$ is said to satisfy the
\emph{uniformly reverse H\"older condition for some
$q\in(1,\fz]$}, denoted by $\fai\in \rh_q(\cx)$, if, when $q\in
(1,\fz)$,
\begin{eqnarray*}
\rh_q (\fai):&&=\sup_{t\in (0,\fz)}\sup_{B\subset\cx}\lf\{
\fint_B [\fai(x,t)]^q\,d\mu(x)\r\}^{1/q}\lf\{\fint_B
\fai(x,t)\,d\mu(x)\r\}^{-1}<\fz,
\end{eqnarray*}
or
\begin{equation*}
\rh_{\fz} (\fai):=\sup_{t\in
(0,\fz)}\sup_{B\subset\cx}\lf\{\esup_{y\in
B}\fai(y,t)\r\}\lf\{\fint_B
\fai(x,t)\,d\mu(x)\r\}^{-1} <\fz.
\end{equation*}
Here the first supremums are taken over all $t\in(0,\fz)$ and the
second ones over all balls $B\subset\cx$.
\end{definition}

Let $\aa_{\fz}(\cx):=\cup_{q\in[1,\fz)}\aa_{q}(\cx)$ and define
the \emph{critical indices} of $\fai\in\aa_{\fz}(\cx)$ as follows:
\begin{equation}\label{2.7}
q(\fai):=\inf\lf\{q\in[1,\fz):\ \fai\in\aa_{q}(\cx)\r\}
\end{equation}
and
\begin{equation}\label{2.8}
r(\fai):=\sup\lf\{q\in(1,\fz]:\ \fai\in\rh_{q}(\cx)\r\}.
\end{equation}

Now we introduce the notion of growth functions.

\begin{definition}\label{d2.2}
A function $\fai:\ \cx\times[0,\fz)\to[0,\fz)$ is called
 a \emph{growth function} if the following hold true:
 \vspace{-0.25cm}
\begin{enumerate}
\item[(i)] $\fai$ is a \emph{Musielak-Orlicz function}, namely,
\vspace{-0.2cm}
\begin{enumerate}
    \item[(i)$_1$] the function $\fai(x,\cdot):\ [0,\fz)\to[0,\fz)$ is an
    Orlicz function for all $x\in\cx$;
    \vspace{-0.2cm}
    \item [(i)$_2$] the function $\fai(\cdot,t)$ is a measurable
    function for all $t\in[0,\fz)$.
\end{enumerate}
\vspace{-0.25cm} \item[(ii)] $\fai\in \aa_{\fz}(\cx)$.
\vspace{-0.25cm} \item[(iii)] The function $\fai$ is of positive
uniformly upper type $1$ and of uniformly
lower type $p_2$ for some $p_2\in(0,1]$.
\end{enumerate}
\end{definition}

\begin{remark}\label{r2.1}
From the definitions of the uniformly upper type and the uniformly lower type, we deduce
that, if the growth function $\fai$ is of positive uniformly upper
type $p_1$ with $p_1\in(0,1]$, and of positive uniformly lower type $p_2$
with $p_2\in(0,1]$, then $p_1\ge p_2$.
\end{remark}

Clearly, $\fai(x,t):=\oz(x)\Phi(t)$ is a growth function if
$\oz\in A_{\fz}(\cx)$ and $\Phi$ is an Orlicz function of lower
type $p$ for some $p\in(0,1]$ and of upper type 1. It is known
that, for $p\in(0,1]$, if $\Phi(t):=t^p$ for all $t\in [0,\fz)$,
then $\Phi$ is an Orlicz function of lower type $p$ and of upper type $p$;
for $p\in[\frac{1}{2},1]$, if
$\Phi(t):= t^p/\ln(e+t)$ for all $t\in [0,\fz)$, then $\Phi$ is an
Orlicz function of lower type $q$ for $q\in(0, p)$ and of upper type $p$; for
$p\in(0,\frac{1}{2}]$, if $\Phi(t):=t^p\ln(e+t)$ for all $t\in
[0,\fz)$, then $\Phi$ is an Orlicz function of lower type $p$ and of upper type $q$
for $q\in(p,1]$. Recall that if an Orlicz function is of upper
type $p\in(0,1)$, then it is also of upper type 1.

Another typical
and useful example of the growth function $\fai$ is as in \eqref{1.2}.
It is easy to show that
$\fai\in \aa_1(\cx)$, $\fai$ is of uniformly
upper type $\az$, $I(\fai)=i(\fai)=\az$, $i(\fai)$ is not attainable,
but $I(\fai)$ is attainable.
Moreover, it worths to point out that such function $\fai$
naturally appears in the study of the pointwise multiplier
characterization for the BMO-type space on the metric space with
doubling measure (see \cite{na97,ny85}); see also \cite{k1,k3,k4,k5} for some
other applications of such functions.

Throughout the whole paper, we \emph{always
assume that $\fai$ is a growth function} as in Definition
\ref{d2.2}. Let us now introduce the Musielak-Orlicz space.

The \emph{Musielak-Orlicz space $L^{\fai}(\cx)$} is defined to be the set
of all measurable functions $f$ such that
$\int_{\cx}\fai(x,|f(x)|)\,d\mu(x)<\fz$ with \emph{Luxembourg
norm}
$$\|f\|_{L^{\fai}(\cx)}:=\inf\lf\{\lz\in(0,\fz):\ \int_{\cx}
\fai\lf(x,\frac{|f(x)|}{\lz}\r)\,d\mu(x)\le1\r\}.
$$
In what follows, for any measurable subset $E$ of $\cx$ and $t\in[0,\fz)$, we let
$$\fai(E,t):=\int_E\fai(x,t)\,d\mu(x).$$

The following Lemma \ref{l2.1} on the
properties of growth functions is just \cite[Lemmas
4.1 and 4.2]{k}.

\begin{lemma}\label{l2.1}
{\rm(i)} Let $\fai$ be a growth function as in Definition \ref{d2.2}.
Then $\fai$ is uniformly
$\sigma$-quasi-subadditive on $\cx\times[0,\fz)$, namely, there
exists a positive constant $C$ such that, for all
$(x,t_j)\in\cx\times[0,\fz)$ with $j\in\nn$,
$\fai(x,\sum_{j=1}^{\fz}t_j)\le C\sum_{j=1}^{\fz}\fai(x,t_j).$

{\rm(ii)} Let $\fai$ be a growth function as in Definition \ref{d2.2}. For all
$(x,t)\in\cx\times[0,\fz)$, assume that
$\wz{\fai}(x,t):=\int_0^t\frac{\fai(x,s)}{s}\,ds$. Then $\wz{\fai}$
is a growth function,
which is equivalent to $\fai$; moreover, $\wz{\fai}(x,\cdot)$ is
continuous and strictly increasing.

{\rm (iii)} Let $\fai$ be a growth function as in Definition \ref{d2.2}.
Then $\int_{\cx}\varphi(x,\frac{|f(x)|}{\|f\|_{L^\varphi(\cx)}})\,d\mu(x)=1$
for all $f\in L^\varphi(\cx)\setminus\{0\}$.

\end{lemma}

We have the following properties for $\aa_\fz(\cx)$, whose proofs
are similar to those in \cite{gr1,gra1}.

\begin{lemma}\label{l2.4}
$\mathrm{(i)}$ $\aa_1(\cx)\subset\aa_p(\cx)\subset\aa_q(\cx)$ for
$1\le p\le q<\fz$.

$\mathrm{(ii)}$ $\rh_{\fz}(\cx)\subset\rh_p(\cx)\subset\rh_q(\cx)$
for $1<q\le p\le\fz$.

$\mathrm{(iii)}$ If $\fai\in\aa_p(\cx)$ with $p\in(1,\fz)$, then
there exists $q\in(1,p)$ such that $\fai\in\aa_q(\cx)$.

$\mathrm{(iv)}$ If $\fai \in \rh_q(\cx)$ with $q\in(1,\fz)$, then there exists
$p\in(q,\fz)$ such that $\fai\in \rh_p(\cx)$.

$\mathrm{(v)}$ $\aa_{\fz}(\cx)=\cup_{p\in[1,\fz)}\aa_p(\cx)
\subset\cup_{q\in(1,\fz]}\rh_q(\cx)$.

$\mathrm{(vi)}$ If $p\in(1,\fz)$ and $\fai\in \aa_{p}(\cx)$, then
there exists a positive constant $C$ such that, for all measurable
functions $f$ on $\cx$ and $t\in[0,\fz)$,
$$\int_{\cx}\lf[\cm(f)(x)\r]^p\fai(x,t)\,d\mu(x)\le
C\int_{\cx}|f(x)|^p\fai(x,t)\,d\mu(x),$$
where $\cm$ denotes the
Hardy-Littlewood maximal function on $\cx$, defined by setting, for all $x\in\cx$,
$$\cm(f)(x):=\sup_{x\in B}\frac{1}{\mu(B)}\int_B|f(y)|\,d\mu(y),$$
where the supremum is taken over all balls $B\ni x$.

$\mathrm{(vii)}$ If $\fai\in \aa_{p}(\cx)$ with $p\in[1,\fz)$, then
there exists a positive constant $C$ such that, for all balls
$B_1,\,B_2\subset\cx$ with $B_1\subset B_2$ and $t\in[0,\fz)$,
$\frac{\fai(B_2,t)}{\fai(B_1,t)}\le
C[\frac{\mu(B_2)}{\mu(B_1)}]^p.$

$\mathrm{(viii)}$ If $\fai\in \rh_{q}(\cx)$ with $q\in(1,\fz]$,
then there exists a positive constant $C$ such that, for all balls
$B_1,\,B_2\subset\cx$ with $B_1\subset B_2$ and $t\in[0,\fz)$,
$\frac{\fai(B_2,t)}{\fai(B_1,t)}\ge
C[\frac{\mu(B_2)}{\mu(B_1)}]^{(q-1)/q}.$
\end{lemma}

\begin{remark}\label{r2.2}
By Lemma \ref{l2.4}(iii), we see that if $q(\fai)\in(1,\fz)$, then
$\fai\not\in\aa_{q(\fai)}(\cx)$. Moreover, there exists $\fai\not\in\aa_1(\cx)$
such that $q(\fai)=1$ (see, for example, \cite{jn87}). Similarly,
if $r(\fai)\in(1,\fz)$, then $\fai\not\in\rh_{r(\fai)}(\cx)$, and there exists
$\fai\not\in\rh_\fz(\cx)$ such that $r(\fai)=\fz$ (see, for example, \cite{cn95}).
\end{remark}

\subsection{Two assumptions on the operator $L$\label{s2.3}}

\hskip\parindent Before giving the assumptions on operators $L$, we first recall
some notions of bounded holomorphic functional calculus introduced by
McIntosh \cite{mc86}.

For $\tz\in[0,\,\pi)$, the \emph{open} and \emph{closed
sectors}, $S_\tz^0$ and $S_\tz$, of angle $\tz$ in the complex plane
$\cc$ are defined, respectively, by setting
$S_\tz^0:=\lf\{z\in\cc\setminus\{0\}:\ |\arg z|<\tz\r\}$ and
$S_\tz:=\lf\{z\in\cc:\ |\arg z|\le\tz\r\}$. Let
$\omega\in[0,\,\pi)$. A closed operator $T$ in $L^2(\cx)$ is said to
be of {\it type} $\omega$, if

(i) the \emph{spectrum} of $T$, $\sz(T)$, is contained in
$S_\omega$;

(ii) for each $\tz\in (\omega,\,\pi)$, there exists a nonnegative
constant $C$ such that, for all $z\in\cc\setminus S_\tz$,
$$\lf\|(T-zI)^{-1}\r\|_{\cl(L^2(\cx))}\le
C|z|^{-1},$$
where above and in what follows, for any normed linear space
${\mathcal H}$, $\|S\|_{\cl(\mathcal H)}$ denotes the {\it operator
norm} of the linear operator $S:\ {\mathcal H}\to {\mathcal H}$.

For $\mu\in\![0,\,\pi)$ and $\sz,\tau\in(0,\,\fz)$, let
$H(S_\mu^0)\!:=\!\lf\{f:\ f\ \text{is a holomorphic function on}\
S_\mu^0\r\},$
\begin{eqnarray*}
H_{\fz}(S_\mu^0):= \lf\{f\in H(S_\mu^0):\
\|f\|_{L^{\fz}(S^0_{\mu})}<\fz\r\}
\end{eqnarray*}
and
\begin{eqnarray*}
\Psi_{\sz,\,\tau}(S_\mu^0):=&&\lf\{f\in H(S_\mu^0):\ \text{there
exists a positive
constant}\ C \ \text{such that}\r.\\
&&\hspace{3.5cm}\text{for all}\ \xi\in S_\mu^0,\lf. |f(\xi)|\le C
\inf\{|\xi|^\sz,\,|\xi|^{-\tau}\}\r\}.
\end{eqnarray*}

It is known that every one-to-one operator $T$ of type $\omega$ in
$L^2(\cx)$ has a unique holomorphic functional calculus (see, for
example, \cite{mc86}). More precisely, let $T$ be a one-to-one
operator of type $\omega$, with $\omega\in[0,\,\pi)$,
$\mu\in(\omega,\,\pi)$, $\sz,\,\tau\in(0,\,\fz)$ and
$f\in\Psi_{\sz,\tau}(S_\mu^0)$. The \emph{function of the operator
$T$}, $f(T)$, can be defined by the $H_\fz$-functional calculus in
the following way,
\begin{eqnarray}\label{2.9}
f(T):=\frac{1}{2\pi i}\int_\Gamma(\xi I-T)^{-1}f(\xi)\,d\xi,
\end{eqnarray}
where $\Gamma:=\{re^{i\nu}:\ \fz>r>0\}\cup\{re^{-i\nu}:\ 0<r<\fz\}$,
$\nu\in(\omega,\,\mu)$, is a curve consisting of two rays
parameterized anti-clockwise. It is known that $f(T)$ in
\eqref{2.9} is independent of the choice of $\nu\in(\omega,\,\mu)$
and the integral in \eqref{2.9} is absolutely convergent in
$\|\cdot\|_{\cl(L^2(\cx))}$ (see \cite{mc86,ha06}).

In what follows, we {\it always assume $\omega\in[0,\,\pi/2)$}.
Then, it follows, from \cite[Proposition 7.1.1]{ha06}, that for every
operator $T$ of type $\omega$ in $L^2(\cx)$, $-T$ generates a
holomorphic $C_0$-semigroup $\{e^{-zT}\}_{z\in S^0_{\pi/2-\omega}}$
on the open sector $S^0_{\pi/2-\omega}$ such that
$\|e^{-zT}\|_{\cl(L^2(\cx))}\le1$ for all $z\in S^0_{\pi/2-\omega}$
and, moreover, every \emph{nonnegative self-adjoint operator} is of
type $0$.

Let $\Psi(S_\mu^0):=\cup_{\sz,\tau>0}\Psi_{\sz,\,\tau}(S_\mu^0)$. It
is well known that the above holomorphic functional calculus defined
on $\Psi(S_\mu^0)$ can be extended to $H_\fz(S_\mu^0)$ via a limit
process (see \cite{mc86}).  Recall that, for $\mu\in(0,\,\pi)$, the
operator $T$ is said to have a \emph{bounded $H_\fz(S_\mu^0)$
functional calculus} in the Hilbert space $\mathcal{H}$, if there
exists a positive constant $C$ such that, for all $\psi\in
H_\fz(S_\mu^0)$, $\|\psi(T)\|_{\cl(\mathcal{H})}\le
C\|\psi\|_{L^\fz(S_\mu^0)}$ and $T$ is said to have a \emph{bounded
$H_\fz$ functional calculus} in the Hilbert space $\mathcal{H}$ if
there exists $\mu\in (0,\,\pi)$ such that $T$ has a bounded
$H_\fz(S_\mu^0)$ functional calculus.

For any given $f\in L^1_\loc(\cx)$, each ball $B\subset\cx$ and $j\in\zz_+$, let
$$\fint_{S_j(B)}|f(x)|d\mu(x):=\frac{1}{\mu(2^jB)}\int_{S_j(B)}|f(x)|d\mu(x).$$

Now we recall the notion of $L^p-L^q$ off-diagonal estimates
on balls, which was first introduced
in \cite{am07ii}.

\begin{definition}\label{d2.3}
Let $k\in\nn$, $p,\,q\in[1,\fz]$ with $p\le q$, and $\{A_t\}_{t>0}$ be a family of sublinear
operators. The family $\{A_t\}_{t>0}$ is said to satisfy \emph{$L^p-L^q$  off-diagonal
estimates on balls of order} $m$, denoted by $A_t\in \mathcal{O}_{m}(L^p-L^q)$, if
there exist constants $\theta_1,\,\theta_2 \in[0,\fz)$ and $C,\,c\in(0,\fz)$ such that,
 for all $t\in(0,\fz)$ and all balls $B\subset\cx$ and $f\in L^p_\loc(\cx)$,
\begin{equation}\label{2.10}
\lf\{\fint_B|A_t\lf(\chi_Bf\r)(x)|^qd\mu(x)\r\}^{1/q}\le C
\lf[\Upsilon\lf(\frac{r_B}{t^{1/2m}}\r)\r]^{\theta_2}\lf\{\fint_B
|f(x)|^p\,d\mu(x)\r\}^{1/p},
\end{equation}
and, for all $j\in\nn$ with $j\ge3$,
\begin{eqnarray}\label{2.11}
&&\lf\{\fint_{S_j(B)}|A_t\lf(\chi_{B}f\r)(x)|^q\,d\mu(x)\r\}^{1/q}\\ \nonumber
&&\hs\le C
2^{j\theta_1}\lf[\Upsilon\lf(\frac{2^jr_B}{t^{1/2m}}\r)\r]^{\theta_2}
e^{-c\frac{(2^jr_B)^{2m/(2m-1)}}{t^{1/(2m-1)}}}\lf\{\fint_{B}
|f(x)|^p\,d\mu(x)\r\}^{1/p}
\end{eqnarray}
and
\begin{eqnarray*}
&&\lf\{\fint_{B}|A_t(\chi_{S_j(B)}f)(x)|^q\,d\mu(x)\r\}^{1/q}\\ \nonumber
&&\hs\le C
2^{j\theta_1}\lf[\Upsilon\lf(\frac{2^jr_B}{t^{1/2m}}\r)\r]^{\theta_2}
e^{-c\frac{(2^jr_B)^{2m/(2m-1)}}{t^{1/(2m-1)}}}\lf\{\fint_{S_j(B)}
|f(x)|^p\,d\mu(x)\r\}^{1/p},
\end{eqnarray*}
where $\Upsilon(s):=\max\{s,\frac{1}{s}\}$ for all $s\in(0,\fz)$.
\end{definition}
Similar to the comments below \cite[Definition 2.1]{am07ii},
we have the following properties on $\mathcal{O}_m(L^p-L^q)$.
\begin{remark}\label{r2.3}
\begin{enumerate}[(i)]
\item It is easy to see that, for $p\leq p_1\leq q_1\leq q$,
$$\mathcal{O}_m(L^{p}-L^{q})
\subset\mathcal{O}_m(L^{p_1}-L^{q_1}).$$

\item Similar to \cite[Proposition 2.2]{am07ii}, we see that
$A_t\in \mathcal{O}_m(L^{1}-L^\fz)$ if and only if the
associated kernel $p_t$ of $A_t$ satisfies the Gaussian upper
bound, namely, there exist positive constants $c$ and $C$ such that, for all
$x,\,y\in\cx$ and $t\in(0,\fz)$,
$$
|p_t(x,y)|\le \frac{C}{V(x,t^{1/2m})}\exp\lf\{-c\frac{[d(x,y)]^{2m/(2m-1)}}
{t^{1/(2m-1)}}\r\}.
$$

\item $A_t\in \mathcal{O}_m(L^{p}-L^q)$ if and only if its dual,
$A^*_t$, belongs to $\mathcal{O}_m(L^{q'}-L^{p'})$.
\end{enumerate}
\end{remark}

Now, we make the following two assumptions on operators $L$, which are used through
the whole paper.

\medskip

\noindent {\bf Assumption (A).} Assume that the operator $L$ is a one-to-one operator
of type $\omega$ in $L^2(\cx)$ with $\omega\in[0,\,\pi/2)$, has dense
range in $L^2(\cx)$ and a bounded $H_\fz$-calculus in $L^2(\cx)$.

\medskip

\noindent {\bf Assumption (B).} Let $m\in\nn$.
Assume that there exist $p_L\in[1,\,2)$ and
$q_L\in(2,\fz]$, depending on $L$, such that the family
$\{(tL)^k e^{-tL}\}_{t>0}$, with $k\in\zz_+$, satisfies the
\emph{reinforced $(p_L,\,q_L,\,m)$ off-diagonal estimates on balls}, namely,
for all $t\in(0,\fz)$ and $p,\,q\in(p_L,q_L)$ with $p\le q$, $(tL)^k e^{-tL}\in
\mathcal{O}_m(L^{p}-L^q)$.

\medskip

\begin{remark}\label{r2.4}
\begin{enumerate}[(a)]

\item We first point that in Assumptions {\rm (A)} and {\rm (B)},
if $L$ is non-negative self-adjoint,
$\cx$ is the Euclidean space $\rn$ and $m=1$, from \cite[Proposition 3.2]{am07ii},
 it follows that the notion of the reinforced $(p_L,\,q_L,\,m)$ off-diagonal estimates
 on balls is the same as the reinforced $(p_L,\,q_L)$ off-diagonal estimates introduced
in \cite{bckyy} (see \cite{bk05-2,da95,ga59}
and their references for the history of the off-diagonal estimates).
Here, we use the off-diagonal estimates
on balls, because they coincide with the off-diagonal
estimates when $\cx=\rn$, and the off-diagonal estimates on balls seem more
suitable in a general space of homogeneous type. For example, the heat semigroup
$e^{-t\bdz}$ on functions for general Riemannian manifolds with doubling
property is not $L^p-L^q$ bounded when $p<q$ unless the measure of
any ball is bounded below by a power of its radius. However, if we assume
the $L^p-L^q$ off-diagonal estimates, it then implies the $L^p-L^q$ boundedness
(see also the discussions above \cite[Proposition 3.2]{am07ii}).

\item Denote by $L^\ast$ the \emph{adjoint operator} of $L$ in $L^2(\cx)$.
Let $p_L,\,q_L$ be as in Assumption {\rm (B)}, $m\in\nn$ and $k\in\nn$. If $(tL)^k e^{-tL}$
satisfies the reinforced $(p_L,\,q_L,\,m)$ off-diagonal estimates on balls, then
$(tL^*)^k e^{-tL^*}$ also satisfies the reinforced $(q'_L,\,p'_L,\,m)$ off-diagonal
estimates on balls. Recall that, for any $p\in[1,\fz]$, $1/p+1/p'=1$.

\item Examples of operators which satisfy Assumptions {\rm (A)} and {\rm (B)}
include:

\begin{enumerate}[(i)]

\item the second order divergence form elliptic operators with complex
bounded coefficients as in \cite{hm09} (see also \eqref{c5.1x} below
for its precise definition);

\item the $2m$-\emph{order homogeneous divergence form elliptic operators}
$$(-1)^m\sum_{|\az|=m=|\bz|}\pat^\bz\lf(a_{\az,\bz}\pat^\az\r)$$
interpreted in the usual weak sense via a
sesquilinear form, with complex bounded measurable coefficients
$a_{\az,\,\bz}$ for all multi-indices $\az$ and $\bz$
(see, for example, \cite{bk1,cy});

\item the \emph{Schr\"odinger operator} $-\Delta +V$ on $\rn$ with the nonnegative potential $V
\in L_{\mathrm{loc}}^1(\rn)$ (see, for example, \cite{hlmmy,jy11} and related references);

\item the \emph{Schr\"odinger operator} $-\Delta +V$ on $\rn$ with the suitable real
potential $V$ as in \cite{ao};

\item the nonnegative self-adjoint operators satisfying \emph{Gaussian upper bounds},
namely, there exist positive constants $C$ and $c$ such that, for all $x,\,y\in\cx$
and $t\in(0,\fz)$,
$$|p_t(x,y)|\le \frac{C}{V(x,t^{1/2m})}\exp\lf\{-c\frac{[d(x,y)]^{2m/(2m-1)}}{t^{1/(2m-1)}}
\r\},$$
where $p_t$ is the associated kernel of $e^{-tL}$ and $m\in\nn$;
\end{enumerate}

\item We point out that the condition that $L$ is one-to-one is necessary
for the bounded $H_\fz$ functional calculus on $L^2(\cx)$
(see \cite{mc86,cdmy96}). Moreover, from \cite[Theorem 2.3]{cdmy96},
it follows that if $T$ is a one-to-one operator of type $\omega$
in $L^2(\cx)$, then $T$ has dense domain and dense range;

\item
If $L$ is nonnegative self-adjoint on $L^2(\cx)$
satisfying the reinforced $(p_L,\,p_L',\,m)$ off-diagonal estimates on balls, then
the condition that $L$ is one-to-one can be removed and we can introduce
another kind of functional calculus by using the spectral theorem.
More precisely, in this case, for every bounded Borel function $F:[0,\,\fz)\to
\cc$, we define the operator $F(L): L^2(\cx)\to L^2(\cx)$ by the formula
$$F(L):=\dint_0^\fz F(\lz)\,d E_L(\lz),$$
where $E_L(\lz)$ is the \emph{spectral resolution} of $L$ (see \cite{hlmmy}
for more details). Observe also that a one-to-one nonnegative self-adjoint operator is of
type $0$.
\end{enumerate}
\end{remark}

Assume that the operator $L$ satisfies Assumptions {\rm (A)} and {\rm (B)}.
For all $k\in \mathbb{N}$, the \emph{vertical square function} $G_{L,\,k}$
is defined by setting, for all
$f\in L^2(\cx)$ and $x\in\cx$,
$$
G_{L,\,k}(f)(x):=\lf\{\int_0^\infty \lf|(t^{2m}L)^ke^{-t^{2m}L}f(x)\r|^2
\frac{dt}{t}\r\}^{1/2},
$$
which is bounded on $L^2(\cx)$ (see, for example, \cite{mc86}).
When $k=1$, we write $G_L$ instead of $G_{L,\,1}$.

\begin{theorem}\label{thm on G{L,m}}
Let $L$ satisfy Assumptions {\rm (A)} and {\rm (B)}, $k\in\nn$ and, $p_L$ and $q_L$ be as
in Assumption {\rm (B)}. Then $G_{L,\,k}$
is bounded on $L^p(\cx)$ for all $p\in(p_L,q_L)$.
\end{theorem}

To prove Theorem \ref{thm on G{L,m}}, we need the following two criteria, which
are due to \cite{au07} (see also \cite{am07i}).

\begin{lemma}\label{Auscher's thm 1}
Let $p_0\in [1,2)$ and $\{A_t\}_{t>0}$ be a family of linear operators acting on $L^2(\cx)$.
Suppose that $T$ is a sublinear operator of strong type $(2,2)$. Assume that there exists
a sequence $\{\az(j)\}_{j=2}^\fz$ of positive numbers such that, for all balls
$B:=B(x_B,r_B)$ and $f\in L^{p_0}(\cx)$ supported in $B$,
\begin{equation}\label{2.13}
\lf\{\fint_{S_j(B)}\lf|T(I-A_{r_B})f(x)\r|^2\,d\mu(x)\r\}^{1/2}\le
\alpha(j)\lf\{\fint_B |f(x)|^{p_0}\,d\mu(x)\r\}^{1/p_0}
\end{equation}
when $j\ge3$, and
\begin{equation}\label{2.14}
\lf\{\fint_{S_j(B)}\lf|A_{r_B}f(x)\r|^2\,d\mu(x)\r\}^{1/2}\le\alpha(j)
\lf\{\fint |f(x)|^{p_0}\,d\mu(x)\r\}^{1/p_0}
\end{equation}
when $j\ge2$. If $\sum_{j=2}^\fz \alpha(j)2^{nj}<\infty$, then $T$ is of weak
type $(p_0, p_0)$.
\end{lemma}

\begin{lemma}\label{Auscher's thm 2}
Let $p_0\in (2,\fz]$ and $\{A_t\}_{t>0}$ be a family of linear operators acting on $L^2(\cx)$.
Suppose that $T$ is a sublinear operator acting on $L^2(\cx)$. Assume that there exists a
positive constant $C$ such that, for all balls
$B:=B(x_B,r_B)$, $y\in B$ and $f\in L^{p_0}(\cx)$ supported in $B$,
\begin{equation}\label{2.15}
\lf\{\fint_{B}\lf|T(I-A_{r_B})f(x)\r|^{p_0}\,d\mu(x)\r\}^{1/p_0}\le C\lf[\cm(|f|^2)(y)\r]^{1/2}
\end{equation}
and
\begin{equation*}
\lf\{\fint_{B}\lf|TA_{r_B}f(x)\r|^{p_0}\,d\mu(x)\r\}^{1/p_0}\le C[\mathcal{M}
(|Tf|^2)(y)]^{1/2},
\end{equation*}
where $\mathcal{M}$ denotes the Hardy-Littlewood maximal function
(see Lemma \ref{l2.4}(vi)).
Then $T$ is of strong type $(p_0, p_0)$.
\end{lemma}

Now we prove Theorem \ref{thm on G{L,m}} by using Lemmas \ref{Auscher's thm 1}
and \ref{Auscher's thm 2}.

\begin{proof}[Proof of Theorem \ref{thm on G{L,m}}]

For the sake of simplicity, we only give the proof for $k=1$.
Since $G_{L}$ is bounded on $L^2(\cx)$, we can assume that $p_L<2<q_L$.
We now consider the following two cases.

\textbf{Case 1). $p\in(p_L,2)$}.

In this case, we apply Lemma \ref{Auscher's thm 1}
with $A_{r_B}:=I-(I-e^{-r_B^{2m}L})^M$,
$M\in\nn$ and $M>(n+\theta_1)/2m$.
From Assumption {\rm (B)}, we deduce that
$e^{-tL} \in \mathcal{O}_m(L^p-L^2)$. Thus,
\eqref{2.14} holds. It remains to show that
for all $j\in\nn$ with $j\ge3$, balls $B$ and $f\in L^p(\cx)$ supported in $B$,
it holds that
\begin{eqnarray}\label{2.17}
&&\lf\{\fint_{S_j(B)}\lf|G_{L}(I-e^{-r_B^{2m}L})^Mf(x)\r|^2\,d\mu(x)\r\}^{1/2}\\ \nonumber
&&\hs\ls2^{-j(mM-\theta_1)}\lf\{\fint_B |f(x)|^{p}\,d\mu(x)\r\}^{1/p}.
\end{eqnarray}

First, we write
\begin{eqnarray}\label{2.18}
&&\lf\{\fint_{S_j(B)}\lf|G_{L}(I-e^{-r_B^{2m}L})^Mf(x)\r|^2
\,d\mu(x)\r\}^{1/2}\\ \nonumber
&&\hs=\lf\{\fint_{S_j(B)}\int_0^\infty\lf|t^{2m}
Le^{-t^{2m}L}(I-e^{-r_B^{2m}L})^Mf(x)\r|^2
\,\frac{dt}{t}d\mu(x)\r\}^{1/2}\\ \nonumber
&&\hs\le\lf\{\fint_{S_j(B)}\int_0^{2^jr_B}
\lf|t^{2m}Le^{-t^{2m}L}(I-e^{-r_B^{2m}L})^Mf(x)\r|^2
\,\frac{dt}{t}d\mu(x)\r\}^{1/2}\\ \nonumber
&&\hs\hs + \lf\{\fint_{S_j(B)}\int^\infty_{2^jr_B}
\cdots\r\}^{1/2}=:\mathrm{I} +\mathrm{II}.
\end{eqnarray}

We first estimate $\mathrm{I}$. Write
$$
\lf(I-e^{-r_B^{2m}L}\r)^M=\int_0^{r^m_B}\cdots
\int_0^{r^{2m}_B}L^M
e^{-(s_1+\cdots+s_M)L}ds_1\cdots ds_M.$$
Thus,
\begin{eqnarray}\label{2.19}
\mathrm{I}&&\le\lf\{\fint_{S_j(B)}\int_0^{2^jr_B}
\int_0^{r^{2m}_B}\cdots \int_0^{r^{2m}_B}\lf|t^{2m}L^{M+1}
e^{-(t^{2m}+s_1+\cdots+s_M)L}f(x)
\r|^2\r.\\ \nonumber
&&\hs\times ds_1\cdots ds_M \frac{dt}{t}d\mu(x)\Bigg\}^{1/2}.
\end{eqnarray}
By the fact that $((t^{2m}+s_1+\cdots+s_M)L)^{M+1}e^{-(t^{2m}+s_1+\cdots+s_M)L}
\in \mathcal{O}_m(L^p-L^2)$, we know that
\begin{eqnarray*}
&&\lf\{\fint_{S_j(B)}\lf|t^{2m}L^{M+1}e^{-(t^{2m}
+s_1+\cdots+s_M)L}f(x)\r|^2\,d\mu(x)\r\}^{1/2}\\
&&\hs\ls\frac{t^{2m}}{(t^{2m}+s_1+\cdots+s_M)^{M+1}}
2^{j\theta_1}\lf[\frac{t^{2m}+s_1
+\cdots+s_M}{\lf(2^jr_B\r)^{2m}}\r]^{M+1}\lf\{\fint_B|f(x)|^p\,d\mu(x)\r\}^{1/p}\\
&&\hs\ls2^{j\theta_1}t^{2m}(2^jr_B)^{-2m(M+1)}
\lf\{\fint_B|f(x)|^p\,d\mu(x)\r\}^{1/p},
\end{eqnarray*}
which, together with Minkowski's integral inequality and
\eqref{2.19}, implies that
\begin{eqnarray}\label{2.20}
\mathrm{I}
&&\ls\int_0^{2^jr_B}\int_0^{r^{2m}_B}\cdots\int_0^{r^{2m}_B}
2^{j\theta_1}t^{2m}(2^jr_B)^{-2m(M+1)}ds_1\cdots\,
ds_M \frac{dt}{t}\\ \nonumber
&&\hs\times\lf\{\fint_B|f(x)|^p\,d\mu(x)\r\}^{1/p}
\ls2^{-j(2mM-\theta_1)}\lf\{\fint_B|f(x)|^p\,d\mu(x)\r\}^{1/p}.
\end{eqnarray}
Likewise, we also see that
$$
\mathrm{II}\ls2^{-j(2mM-\theta_1)}\lf\{\fint_B|f(x)|^p\,d\mu(x)\r\}^{1/p}.
$$
This, combined with \eqref{2.18} and \eqref{2.20}, shows that \eqref{2.17}
holds as long as $M>(n+\theta_1)/2m$. Thus, as a direct consequence of
Lemma \ref{2.14} and interpolation, we conclude that
$G_L$ is bounded on $L^p(\cx)$ for all $p\in(p_L,\,2)$.

\textbf{Case 2). $p\in(2,\,q_L)$.}

In this case, we first prove \eqref{2.15}  for $T:=G_L$
and $A_{r_B}:=I-(I-e^{-r_B^{2m}L})^M$ with $M>(n+\theta_1)/2m$. To do this,
we decompose $f=\sum_{j=0}^\fz f_j$, where for each $j\in\zz_+$, $f_j:=f\chi_{S_j(B)}$.
Then, by an argument similar to that used in the proof of Case 1), we conclude that
\eqref{2.15} holds true for $T:=G_L$ and $A_{r_B}:=I-(I-e^{-r_B^{2m}L})^M$
with $M>(n+\theta_1)/2m$. We now claim that, for all $f\in L^2(\cx)$,
balls $B:=B(x_B,r_B)$ and $y\in B$,
\begin{equation}\label{2.21}
\lf\{\fint_{B}\lf|G_L [I-(I-e^{-r_B^{2m}L})^M]f(x)\r|^{p}\r\}^{1/p}\ls\lf[\cm(|Tf|^2)
(y)\r]^{1/2}.
\end{equation}

Since $I-(I-e^{-r_B^{2m}L})^M=\sum_{j=1}^M c_{(j,\,M)}e^{-jr_B^{2m}L}$, where
$\{c_{(j,\,M)}\}_{j=0}^M$ are constants depending on $j$ and $M$, it follows that
\eqref{2.21} is equivalent to that, for all $j\in\{1,\,\ldots,\,M\}$,
\begin{equation}\label{2.22}
\lf\{\fint_{B}\lf|G_L e^{-jr_B^{2m}L}f(x)\r|^{p}\,d\mu(x)\r\}^{1/p}
\ls\lf[\cm(|Tf|^2)(y)\r]^{1/2}.
\end{equation}
To see this, by Minkowski's inequality, we conclude that
\begin{eqnarray*}
&&\lf\{\fint_{B}\lf|G_Le^{-jr_B^{2m}L}f(x)\r|^{p}\,d\mu(x)\r\}^{1/p}\\
&&\hs\le \lf\{\int_0^\infty\lf[\fint_{B}\lf|e^{-jr_B^{2m}L}\lf(t^{2m}
Le^{-t^{2m}L}f\r)(x)\r|^p
\,d\mu(x)\r]^{2/p}\frac{dt}{t}\r\}^{1/2}.
\end{eqnarray*}
Let $g:=t^{2m}Le^{-t^{2m}L}f$ and $g_i:=g\chi_{S_i(B)}$ with $i\in\zz_+$.
Then, from the fact that $e^{-jr_B^{2m}L}\in \mathcal{O}_m(L^2-L^p)$,
we deduce that, for all $y\in B$,
\begin{eqnarray*}
\lf\{\fint_{B}\lf|G_L e^{-jr_B^{2m}L}f(x)\r|^{p}\,d\mu(x)\r\}^{1/p}&&\ls
\lf\{\int_0^\infty \sum_{i\in\zz_+}\lf[\fint_{B}\lf|e^{-jr_B^{2m}L}g_i(x)\r|^p\,
d\mu(x)\r]^{2/p}\frac{dt}{t}\r\}^{1/2}\\
&&\ls\lf[\cm(|f|^2)(y)\r]^{1/2}.
\end{eqnarray*}
Thus, \eqref{2.22} holds and hence \eqref{2.21} holds true.
By this and Lemma \ref{Auscher's thm 2}, we see that $G_L$ is bounded on
$L^p(\cx)$ for all $p\in(2,\,q_L)$, which completes the proof of Theorem \ref{thm on G{L,m}}.
\end{proof}

For all $k\in \mathbb{N}$, the \emph{non-tangential square functions}
$S_{L,\,k}$ is defined by setting, for all
$f\in L^2(\cx)$ and $x\in\cx$,
\begin{equation}\label{2.23}
S_{L,\,k}(f)(x):=\lf\{\int_{0}^\fz \int_{B(x,t)}
\lf|(t^{2m}L)^ke^{-t^{2m}L}f(y)\r|^2\,\frac{d\mu(y)}{V(x,t)}\frac{dt}{t}\r\}^{1/2}.
\end{equation}
In particular case $k=1$, we omit the subscript $k$ to write $S_L$.
It is easy to show that, for all $f\in L^2(\cx)$, $\|S_{L,\,k}(f)\|_{L^2(\cx)}\ls
\|G_{L,\,k}(f)\|_{L^2(\cx)}$ and hence $S_{L,\,k}$ is bounded on $L^2(\cx)$.
Moreover, we have the following boundedness of $S_{L,\,k}$ on $L^p(\cx)$.
\begin{theorem}\label{thm on S_L}
Let $L$ satisfy Assumptions {\rm (A)} and {\rm (B)}, $k\in\nn$ and $p_L$ and $q_L$ be as
in Assumption {\rm (B)}.
Then $S_{L,\,k}$ is bounded on $L^p(\cx)$ for all $p\in(p_L,\,q_L)$.
\end{theorem}

\begin{proof}
Without the loss of generality, we may assume that $k=1$. Similar to the proof
of Theorem \ref{thm on G{L,m}}, we also consider the following two cases for $p$.

\textbf{Case 1). $p\in(p_L,2)$.}

In this case, we apply Theorem \ref{Auscher's thm 1} to this
situation for $T:=S_L$ and $A_{r_B}:=I-(I-e^{-r_B^{2m}L})^M$ with $M>(2n+\theta_1)/2m$.
Due to the fact that $(tL)^k e^{-tL}\in \mathcal{O}_m(L^p-L^2)$ for all $k\in\zz_+$,
we only need to show \eqref{2.14}, namely, for all $j\in\nn$ with
$j\ge3$, balls $B:=B(x_B,r_B)$ and $f\in L^p(\cx)$ supported in $B$, it holds that
\begin{eqnarray}\label{2.24}
&&\lf\{\fint_{S_j(B)}\lf|S_{L}(I-e^{-r_B^{2m}L})^Mf(x)\r|^2\,d\mu(x)\r\}^{1/2}\\ \nonumber
&&\hs\ls2^{-j(2mM-\theta_1)}\lf\{\fint_B |f(x)|^{p}\,d\mu(x)\r\}^{1/p}.
\end{eqnarray}

To show this, we first write
\begin{eqnarray*}
&&\fint_{S_j(B)}\lf|S_{L}(I-e^{-r_B^{2m}L})^Mf(x)\r|^2\,d\mu(x)\\
&&\hs=\fint_{S_j(B)}\int_0^{\frac{d(x,x_B)}{4}}\int_{B(x,t)} \lf|t^{2m}
Le^{-t^{2m}L}(I-e^{-r_B^{2m}L})^Mf(y)\r|^2\,\frac{d\mu(y)}{V(x,t)}\frac{dt}
{t}d\mu(x)\\&&\hs\hs+ \fint_{S_j(B)}\int^\infty_{\frac{d(x,x_B)}{4}}
\cdots=: \mathrm{I}_1 + \mathrm{I}_2.
\end{eqnarray*}
Let us first estimate $\mathrm{I}_1$. Let
$$F_j(B):=\lf\{z\in\cx:\ \text{there exists $x\in S_j(B)$ such that $d(x,z)
<\frac{d(x,x_B)}{4}$}\r\}.$$
Then $F_j(B)\subset S_{j-1}(B)\cup S_j(B)\cup S_{j+1}(B)=:U_j(B)$.
This, together with the fact that $\int_{d(x,y)<t}\frac{1}{V(x,t)}d\mu(x)\ls1$, implies that
\begin{eqnarray*}
\mathrm{I}_1&&\le\frac{1}{\mu(2^jB)}\int_{F_j(B)}\int_0^{\frac{d(x,x_B)}{4}}
\lf|t^{2m}Le^{-t^{2m}L}(I-e^{r_B^{2m}L})^Mf(y)\r|^2\,\frac{dt}{t}d\mu(y)\\
&&\ls\frac{1}{\mu(2^jB)}\int_{U_j(B)}\int_0^{2^jr_B}\lf|t^{2m}L
e^{-t^{2m}L}(I-e^{r_B^{2m}L})^Mf(y)\r|^2\,\frac{dt}{t}d\mu(y).
\end{eqnarray*}
At this stage, by an argument used in  Case 1) of the proof of Theorem
\ref{thm on G{L,m}}, we conclude that
\begin{equation*}
\mathrm{I}_1\ls2^{-2j(2mM-\theta_1)}\lf\{\fint_B |f(x)|^{p}\,d\mu(x)\r\}^{2/p}.
\end{equation*}
Likewise, for $\mathrm{I}_2$, we write
\begin{eqnarray*}
\mathrm{I}_2&&\le\fint_{S_j(B)}\int^\infty_{\frac{d(x,x_B)}{4}}\int_{B(x,t)}
\lf|t^{2m}Le^{-t^{2m}L}(I-e^{-r_B^{2m}L})^Mf(y)\r|^2\,\frac{d\mu(y)}{V(x,t)}
\frac{dt}{t}d\mu(x)\\
&&\ls\frac{1}{\mu(2^jB)}\int^\infty_{2^{j-1}r_B}\int_{\cx}
\lf|t^{2m}Le^{-t^{2m}L}(I-e^{-r_B^{2m}L})^Mf(y)\r|^2\,d\mu(y)\frac{dt}{t}\\
&&\ls\frac{1}{\mu(2^jB)}\int^\infty_{2^{j-1}r_B}\int_{4B_t}\lf|t^{2m}Le^{-t^{2m}L}
(I-e^{-r_B^{2m}L})^Mf(y)\r|^2\,d\mu(y)\frac{dt}{t}\\
&&\hs+ \sum_{j=2}^\fz\frac{1}{\mu(2^jB)}\int^\infty_{2^{j-1}r_B}
\int_{S_j(B_t)}\lf|t^{2m}Le^{-t^{2m}L}(I-e^{-r_B^{2m}L})^Mf(y)\r|^2\,d\mu(y)
\frac{dt}{t}\\&&=:\mathrm{K}+\sum_{j=2}^\fz\mathrm{H}_j,
\end{eqnarray*}
where $B_t:=B(x_B,t)$.

Notice that in this situation, $B\subset B_t$ and hence $f=f\chi_{B_t}$. By an argument
similar to that used in the proof of Theorem \ref{thm on G{L,m}} and the fact
that $(tL)^k e^{-tL}\in \mathcal{O}_m(L^p-L^2)$ for all $k\in\zz_+$, we see that
\begin{eqnarray*}
&&\lf\{\fint_{4B_t}\lf|t^{2m}Le^{-t^{2m}L}(I-e^{-r_B^{2m}L})^Mf(y)\r|^2\,
d\mu(y)\r\}^{1/2}\\
&&\hs\ls\int_0^{r_B^{2m}}\cdots\int_0^{r_B^{2m}}
\lf\{\fint_{4B_t}\lf|t^{2m}L^{M+1}e^{-(t^{2m}+s_1+\cdots+s_M)L}
f(y)\r|^2\, d\mu(y)\r\}^{1/2}d\vec{s}\\
&&\hs\ls\int_0^{r_B^{2m}}\cdots\int_0^{r_B^{2m}}\frac{t^{2m}}
{(t^{2m}+s_1+\cdots+s_M)^{M+1}} d\vec{s}\lf\{\fint_{4B_t}|f(x)|^p\,
d\mu(x)\r\}^{1/p}\\
&&\hs\ls\frac{r_B^{2mM}}{t^{2mM}}\lf\{\fint_{4B_t}|f(x)|^p\,d\mu(x)\r\}^{1/p},
\end{eqnarray*}
where $d\vec{s}:=ds_1\cdots ds_M$. This implies that
\begin{eqnarray*}
\mathrm{K}&&\ls\int^\infty_{2^{j-1}r_B}\frac{r_B^{4mM}}{t^{4mM}}
\frac{\mu(B_t)}{\mu(2^jB)}\frac{dt}{t}\lf\{\fint_{B}|f(x)|^p\,
d\mu(x)\r\}^{2/p}\\
&&\ls\int^\infty_{2^{j-1}r_B}\frac{r_B^{4mM}}{t^{4mM}}
\lf(\frac{t}{2^jr_B}\r)^n\frac{dt}{t}
\lf\{\fint_{B}|f(x)|^p\,d\mu(x)\r\}^{2/p}\\
&&\ls2^{-j(4mM-n)}\lf\{\fint_{B}|f(x)|^p\,d\mu(x)\r\}^{2/p}.
\end{eqnarray*}

Likewise, for all $j\in\nn$ with $j\ge2$, we have
\begin{eqnarray*}
&&\lf\{\fint_{S_j(B_t)}\lf|t^{2m}Le^{-t^{2m}L}(I-e^{-r_B^{2m}L})^Mf(y)\r|^2\,
d\mu(y)\r\}^{1/2}\\
&&\hs\le\int_0^{r_B^{2m}}\cdots\int_0^{r_B^{2m}}\lf\{\fint_{S_j(B_t)}
\lf|t^{2m}L^{M+1}e^{-(t^{2m}+s_1+\cdots+s_M)L}f(y)\r|^2\,d\mu(y)\r\}^{1/2}d\vec{s}\\
&&\hs\ls2^{j\theta_1}\lf(\frac{2^jr_{B_t}}{t}\r)^{\theta_2}
\int_0^{r_B^{2m}}\cdots\int_0^{r_B^{2m}}\frac{t^{2m}}{(t^{2m}+
s_1+\cdots+s_M)^{M+1}}
\exp\lf\{-c\frac{(2^jr_{B_t})^{2m}}{t^{2m}}\r\}\,d\vec{s} \\
&&\hs\hs\times\lf\{\fint_{B_t}|f(x)|^p\,d\mu(x)\r\}^{1/p}\ls
2^{-j(n+\epsilon)}\frac{r_B^{2mM}}{t^{2mM}}\lf\{\fint_{B_t}|f(x)|^p\,d\mu(x)\r\}^{1/p}.
\end{eqnarray*}
Therefore,
\begin{eqnarray*}
\mathrm{H}_j&&\ls2^{-j(n+\epsilon)}\int_{2^{j-1}r_B}^\infty
\frac{r_B^{4mM}}{t^{4mM}} \frac{\mu(2^jB_t)}{\mu(2^jB)}
\lf\{\fint_{B}|f(x)|^p\,d\mu(x)\r\}^{2/p}\\
&&\ls2^{-j(n+\epsilon)}\int_{2^{j-1}r_B}^\infty \frac{r_B^{4mM}}{t^{4mM}}
\lf(\frac{2^jt}{2^jr_B}\r)^n \frac{dt}{t}\lf\{\fint_{B}|f(x)|^p\,d\mu(x)\r\}^{2/p}\\
&&\ls2^{-j\epsilon}2^{-j(4mM-n)}\lf\{\fint_{B}|f(x)|^p\,d\mu(x)\r\}^{2/p},
\end{eqnarray*}
where $n$ is the dimension of $\cx$ appearing in \eqref{2.2}.

From these estimates of $\mathrm{K}$ and $\mathrm{H}_j$, we deduce  that
\eqref{2.24} holds and hence $S_L$ is bounded on $L^p(\cx)$ for all
$p\in(p_L,\,2)$.

\textbf{Case 2). $p\in(2,\,q_L)$.}

In this case, for any $h\in L^{(p/2)'}(\cx)$, from Fubini's theorem and H\"older's inequality,
we infer that
\begin{eqnarray*}
\int_\cx [S_L f(x)]^2h(x)\,d\mu(x)&&=\int_\cx\int_0^\infty\int_{B(x,t)}
\lf|t^{2m}Le^{-t^{2m}L}f(y)\r|^2h(x)\,\frac{d\mu(y)}{V(x,t)}\frac{dt}{t}d\mu(x)\\
&&=\int_\cx\int_0^\infty\int_{B(x,t)}\lf|t^{2m}Le^{-t^{2m}L}f(y)\r|^2h(x)\,
\frac{d\mu(x)}{V(x,t)}\frac{dt}{t}d\mu(y)\\
&&\ls\int_\cx\int_0^\infty\lf|t^{2m}Le^{-t^{2m}L}f(y)\r|^2
\mathcal{M}(h)(y)\,\frac{dt}{t}d\mu(y)\\
&&\ls\int_\cx [G_L f(x)]^2\mathcal{M}(h)(y)\,d\mu(y)
\ls\|G_L(f)\|^2_{L^p(\cx)}\lf\|\mathcal{M}(h)\r\|_{L^{(p/2)'}(\cx)}.
\end{eqnarray*}
At this stage, using Theorem \ref{thm on G{L,m}}
and the fact that $\mathcal{M}$ is bounded on $L^p(\cx)$ for all $p\in(1,\fz)$,
we conclude that
$$
\int_\cx [S_Lf(x)]^2h(x)\,d\mu(x)\ls\|f\|^2_{L^p(\cx)}\|h\|_{L^{(p/2)'}(\cx)},
$$
which implies that $S_L$ is bounded on $L^p(\cx)$ for all $p\in(2,q_L)$ and hence completes
the proof of Theorem \ref{thm on S_L}.
\end{proof}

\section{Musielak-Orlicz tent spaces\label{s3}}

\hskip\parindent In this section, we study the Musielak-Orlicz
tent space associated with the growth function.
We first recall some notions as follows.

For any $\nu\in(0,\fz)$ and $x\in\cx$, let
$\bgz_{\nu}(x):=\{(y,t)\in\cx_+:\ d(x,y)<\nu t\}$
be the \emph{cone of aperture $\nu$ with vertex $x\in\cx$}, here
and in what follows, we \emph{always assume that $\mathcal{X}_+:=
\mathcal{X}\times(0,\,\fz)$.} For any
closed subset $F$ of $\cx$, denote by $\ccr_{\nu}F$ the \emph{union of all
cones with vertices in $F$}, namely, $\ccr_{\nu}F:=\cup_{x\in
F}\bgz_{\nu}(x)$ and, for any open subset $O$ of $\cx$, denote the
\emph{tent over $O$} by $T_{\nu}(O)$, which is defined as
$T_{\nu}(O):=[\ccr_{\nu}(O^\complement)]^{\complement}$. It is
easy to see that
$T_{\nu}(O)=\{(x,t)\in\cx_+:\ d(x,O^\complement)
\ge\nu t\}.$
In what follows, we denote $\bgz_1(x)$ and
$T_1(O)$ \emph{simply by $\bgz(x)$ and $\widehat{O}$}, respectively.

For all measurable functions $g$ on $\cx_+$ and
$x\in\cx$, define
$$\ca(g)(x):=\lf\{\int_{\bgz(x)}|g(y,t)|^2
\frac{d\mu(y)}{V(x,t)}\frac{dt}{t}\r\}^{1/2}.
$$
Coifman, Meyer and Stein \cite{cms85} introduced the
tent space $T^p_2(\rr^{n+1}_+)$ for $p\in(0,\fz)$, here and in what
follows, $\rr^{n+1}_+:=\rn\times(0,\fz)$. The tent space
$T^p_2(\cx_+)$ on spaces of homogenous type was introduced
by Russ \cite{ru07}. Recall that a measurable function $g$ is said
to belong to the \emph{tent space} $T^p_2(\cx_+)$ with $p\in(0,\fz)$, if
$\|g\|_{T^p_2(\cx_+)}:=\|\ca(g)\|_{L^p(\cx)}<\fz$. Moreover,
Harboure, Salinas and Viviani \cite{hsv07}, and Jiang and Yang
\cite{jy11}, respectively, introduced the Orlicz tent spaces
$T_{\bfai}(\rr^{n+1}_+)$ and $T_{\bfai}(\cx_+)$.

Let $\fai$ be as in Definition \ref{d2.2}. In what follows, we
denote by $T_\fai(\cx_+)$ the \emph{space of all measurable functions
$g$ on $\cx_+$ such that $\ca(g)\in L^{\fai}(\cx)$} and,
for any $g\in T_\fai(\cx_+)$, its \emph{quasi-norm} is defined by
$$\|g\|_{T_\fai(\cx_+)}:=\|\ca(g)\|_{L^{\fai}(\cx)}=
\inf\lf\{\lz\in(0,\fz):\
\int_{\cx}\fai\lf(x,\frac{\ca(g)(x)}{\lz}\r)\,d\mu(x)\le1\r\}.
$$

Let $p\in(1,\,\fz)$. A function $A$ on $\cx_+$ is called a
\emph{$(T_\fai,\,p)$-atom} if

(i) there exists a ball $B\subset\cx$ such that $\supp
a\subset\widehat{B}$;

(ii) $\|A\|_{T^p_2(\cx_+)}\le
[\mu(B)]^{1/p}\|\chi_B\|_{L^\fai(\cx)}^{-1}$.

Furthermore, if $A$ is a $(T_\fai,p)$-atom for all $p\in (1,\fz)$, we then
call $A$ a \emph{$(\fai,\fz)$-atom}.

For functions in $T_\fai(\cx_+)$, we have the following
atomic decomposition.

\begin{theorem}\label{t3.1}
Let $\fai$ be as in Definition \ref{d2.2}. Then for any $f\in
T_\fai(\cx_+)$, there exist $\{\lz_j\}_j\subset\cc$ and a sequence
$\{A_j\}_j$ of $(T_\fai,\,\fz)$-atoms associated with $\{B_j\}_j$ such that, for almost every
$(x,t)\in\cx_+$,
\begin{equation}\label{3.1}
f(x,t)=\sum_{j}\lz_jA_j(x,t).
\end{equation}
Moreover, there exists a positive constant $C$ such that, for all
$f\in T_\fai(\cx_+)$,
\begin{eqnarray}\label{3.2}
\blz(\{\lz_j A_j\}_j)&&:=\inf\lf\{\lz\in(0,\fz):\
\sum_j\fai\lf(B_j,\frac{|\lz_j|}
{\lz\|\chi_{B_j}\|_{L^\fai(\cx)}}\r)\le1\r\}\\
&&\le
C\|f\|_{T_\fai(\cx_+)}.\nonumber
\end{eqnarray}
\end{theorem}

The proof of Theorem \ref{t3.1} is similar to that of \cite[Theorem 3.1]{yys4}.
We omit the details here.

\begin{corollary}\label{c3.1}
Let $p\in(0,\fz)$ and $\fai$ be as in Definition \ref{d2.2}. If $f\in
T_{\fai}(\cx_+)\cap T^p_2(\cx_+)$, then the
decomposition \eqref{3.1} also holds in both $T_{\fai}(\cx_+)$ and
$T^p_2(\cx_+)$.
\end{corollary}

The proof of Corollary \ref{c3.1} is similar to that of \cite[Corollary 3.5]{yys4}
 and hence we omit the details here.

In what follows, let $T^b_{\fai}(\cx_+)$ and $T^{p,\,b}_2(\cx_+)$ with
$p\in(0,\fz)$ denote, respectively, the \emph{set of all functions
in $T_\fai(\cx_+)$ and $T^p_2(\cx_+)$ with bounded support}. Here
and in what follows, a function $f$ on $\cx_+$ is said to have
\emph{bounded support} means that there exist a ball $B\subset\cx$ and
$0<c_1<c_2<\fz$ such that $\supp f\subset B\times(c_1,c_2)$.

\begin{proposition}\label{p3.1}
Let $\fai$ be as in Definition \ref{d2.2}. Then $T^b_{\fai}(\cx_+)
\subset T^{2,\,b}_2(\cx_+)$ as sets.
\end{proposition}

The proof of Proposition \ref{p3.1} is an application
of the uniformly lower type $p_2$ property of $\fai$ for some
$p_2\in(0,1]$, which is similar to that of
\cite[Proposition 3.5]{hyy}. We omit the details.

\section{The Musielak-Orlicz-Hardy space $H_{\fai,\,L}(\cx)$ and its
molecular characterization\label{ss4}}

\hskip\parindent In this section, we first introduce the Musielak-Orlicz-Hardy
space $H_{\fai,\,L}(\cx)$ associated with the operator $L$ via the Lusin-area function.
Then we establish an equivalent characterization of $H_{\fai,\,L}(\cx)$
in terms of the molecule. We begin with some notions and notations.

Let $L$ satisfy Assumptions {\rm (A)} and {\rm (B)}, and $m\in\nn$ be as in
\eqref{2.10}. For all $f\in L^2(\cx)$, the Lusin-area
function $S_L$ is defined as in \eqref{2.23}.

By Theorem \ref{thm on S_L},
we know that, for any $p\in(p_L,q_L)$, where $p_L$ and $q_L$ are as in
Assumption {\rm (B)}, there exists a positive constant $C(p)$, depending on $p$,
such that, for all $f\in L^p(\cx)$,
\begin{equation}\label{4.1}
\|S_L(f)\|_{L^p(\cx)}\le C(p)\|f\|_{L^p(\cx)}.
\end{equation}

Now we introduce the Musielak-Orlicz-Hardy $H_{\fai,\,L}(\cx)$ via the Lusin-area function
$S_L$.

\begin{definition}\label{d4.1}
Let $\fai$ be as in Definition \ref{d2.2} and $L$ satisfy Assumptions {\rm (A)} and {\rm (B)}.
Assume that $p_L$ and $q_L$ are as in Assumption {\rm (B)}.
A function $f\in L^p(\cx)$ with $p\in(p_L,q_L)$ is said to be in
$\widetilde{H}_{\fai,\,L,\,p}(\cx)$ if $S_L(f)\in L^\fai(\cx)$ and, moreover, define
$$\|f\|_{H_{\fai,\,L,\,p}(\cx)}:=\|S_L f\|_{L^\fai(\cx)}:=\inf\lf\{\lz\in(0,\fz):\ \int_{\cx}
\fai\lf(x,\frac{S_L(f)(x)}{\lz}\r)\,d\mu(x)\le1\r\}.
$$
The \emph{Musielak-Orlicz-Hardy space $H_{\fai,\,L,\,p}(\cx)$} is defined to
be the completion of $\widetilde{H}_{\fai,\,L,\,p}(\cx)$ with respect to the quasi-norm
$\|\cdot\|_{H_{\fai,\,L,\,p}(\cx)}$.
\end{definition}

In what follows, for the simplicity of the notation,
we \emph{write} $H_{\fai,\,L}(\cx):=H_{\fai,\,L,\,2}(\cx)$.

\begin{remark}\label{r4.1}
From the Aoki-Rolewicz theorem in \cite{ao42,ro57}, it follows
that, there exist a quasi-morn $\|\!|\cdot\|\!|$ on
$\wz{H}_{\fai,\,L,\,p}(\cx)$ and $\gz\in(0,1]$ such that, for all
$f\in\wz{H}_{\fai,\,L,\,p}(\cx)$,
$\|\!|f\|\!|\sim\|f\|_{H_{\fai,\,L,\,p}(\cx)}$ and, for any sequence
$\{f_j\}_j\subset\wz{H}_{\fai,\,L,\,p}(\cx)$,
$$\lf\|\!\lf|\sum_j f_j\r\|\!\r|^{\gz}\le\sum_j\|\!|f_j\|\!|^{\gz}.
$$
By the theorem of completion of Yosida \cite[p.\,56]{yo95}, it
follows that $(\wz{H}_{\fai,\,L,\,p}(\cx),\,\|\!|\cdot\|\!|)$ has a
completion space $(H_{\fai,\,L,\,p}(\cx),\|\!|\cdot\|\!|)$; namely,
for any $f\in H_{\fai,\,L,\,p}(\cx)$, there exists a
Cauchy sequence $\{f_k\}_{k=1}^{\fz} \subset
\wz{H}_{\fai,\,L,\,p}(\cx)$ such that
$\lim_{k\to\fz}\|\!|f_k -f\|\!|=0$. Moreover, if
$\{f_k\}_{k=1}^{\fz}$ is a Cauchy sequence in
$\wz{H}_{\fai,\,L,\,p}(\cx)$, then there exists a
unique $f\in H_{\fai,\,L,\,p}(\cx)$ such that
$\lim_{k\to\fz}\|\!|f_k -f\|\!|=0$. Furthermore, by the fact that
$\|\!|f\|\!|\sim\|f\|_{H_{\fai,\,L,\,p}(\cx)}$ for all
$f\in\wz{H}_{\fai,\,L,\,p}(\cx)$, we know that the spaces
$(H_{\fai,\,L,\,p}(\cx), \|\cdot\|_{H_{\fai,\,L,\,p}(\cx)})$ and
$(H_{\fai,\,L,\,p}(\cx), \|\!|\cdot\|\!|)$ coincide with equivalent
quasi-norms.
\end{remark}

To introduce the molecular Musielak-Orlicz-Hardy space, we first
introduce the notion of the molecule associated with the growth function $\fai$.

\begin{definition}\label{d4.2}
Let $\fai$ be as in Definition \ref{d2.2}, $L$ satisfy Assumptions {\rm (A)} and
{\rm (B)}, $p_L$ and $q_L$ be as in Assumption {\rm (B)}.
Let $q\in(p_L,q_L)$, $M\in\nn$ and $\epz\in(0,\fz)$.
A function $\az\in L^q(\cx)$ is called a \emph{$(\fai,\,q,\,M,\,\epz)_L$-molecule}
associated with the ball $B\subset\cx$ if,
for each $k\in\{0,\,\ldots,\,M\}$ and $j\in\zz_+$, it holds that
$$\lf\|\lf(r_B^{-2m} L^{-1}\r)^k\az\r\|_{L^q(S_j(B))}\le
2^{-j\epz}[\mu(2^jB)]^{1/q}\|\chi_{B}\|_{L^\fai(\cx)}^{-1}.$$

Moreover, if $\az$ is a $(\fai,\,q,\,M,\,\epz)_L$-molecule
for all $q\in(p_L,\,q_L)$, then $\az$ is called a \emph{$(\fai,\,M,\,\epz)_L$-molecule}.
\end{definition}

\begin{definition}\label{d4.3}
Let $\fai$ be as in Definition \ref{d2.2}, $L$ satisfy Assumptions {\rm (A)}
and {\rm (B)}, $p_L$
and $q_L$ be as in Assumption {\rm (B)}.
Assume that $q\in(p_L,q_L)$, $M\in\nn$ and $\epz\in(0,\fz)$. The equality
$f=\sum_j\lz_j\az_j$ is called a \emph{molecular $(\fai,\,r,\,q,\,M,\,\epz)$-representation}
of $f$ for some $r\in(p_L,q_L)$, if each $\az_j$ is a $(\fai,\,q,\,M,\,\epz)_L$-molecule
associated to the ball $B_j\subset\cx$, the summation converges in $L^r(\cx)$
and $\{\lz_j\}_j$ satisfies that
$$\sum_j\fai\lf(B_j,|\lz_j|\|\chi_{B_j}\|^{-1}_{L^\fai(\cx)}\r)<\fz.$$
Let
\begin{eqnarray*}
\wz{H}^{M,\,q,\,\epz}_{\fai,\,L}(\cx):=&&\lf\{f:\ f \ \text{has a molecular}\r.\\
&&\hs\lf.(\fai,\,r,\,q,\,M,\,\epz)-\text{representation for some}\ r\in(p_L,q_L)\r\}
\end{eqnarray*}
with the \emph{quasi-norm} $\|\cdot\|_{H^{M,\,q,\,\epz}_{\fai,\,L}(\cx)}$ given by
setting, for all $f\in\wz{H}^{M,\,q,\,\epz}_{\fai,\,L}(\cx)$,
\begin{eqnarray*}
&&\|f\|_{H^{M,\,q,\,\epz}_{\fai,\,L}(\cx)}\\
&&\hs:=\inf\lf\{\blz(\{\lz_j\az_j\}_j):\
f=\sum_j\lz_j\az_j \ \text{is a molecular}
\ (\fai,\,r,\,q,\,M,\,\epz)\text{-representation}\r\},
\end{eqnarray*}
where $\blz(\{\lz_j\az_j\}_j)$ is as in \eqref{3.2}.

The \emph{molecular Musielak-Orlicz-Hardy space} $H^{M,\,q,\,\epz}_{\fai,\,L}(\cx)$ is then
defined as the completion of $\wz{H}^{M,\,q,\,\epz}_{\fai,\,L}(\cx)$ with respect to the
quasi-norm $\|\cdot\|_{H^{M,\,q,\,\epz}_{\fai,\,L}(\cx)}$.
\end{definition}

In what follows, let $L^2_b(\cx_+)$ denote the
\emph{set of all functions $f\in L^2(\cx_+)$ with
bounded support}, $M\in\nn$ and
$M>\frac{n}{2m}[\frac{q(\fai)}{i(\fai)}+\frac{\tz_1}{n}-\frac{2}{nq_L}]$, where
$k$, $q(\fai)$, $i(\fai)$, $\tz_1$ and $q_L$ are respectively as in Definition \ref{d2.3},
\eqref{2.7}, \eqref{2.5}, \eqref{2.11} and Assumption {\rm (B)}.
For all $f\in L^2_b(\cx_+)$ and $x\in\cx$, define
\begin{equation}\label{4.2}
\pi_{L,\,M}(f)(x):=C_{(m,\,M)}\int_0^{\fz}(t^{2m}L)^{M+1}e^{-t^{2m}L}
(f(\cdot,t))(x)\,\frac{dt}{t},
\end{equation}
where $C_{(m,\,M)}$ is a positive constant such that
\begin{equation}\label{4.3}
C_{(m,\,M)}\int_0^\fz t^{2m(M+2)}e^{-2t^{2m}}\,\frac{dt}{t}=1.
\end{equation}
Here $m$ is as in Definition \ref{d2.3}.

For the operator $\pi_{L,\,M}$, we have the following boundedness.

\begin{proposition}\label{p4.1}
Assume that $L$ satisfies Assumptions {\rm (A)} and {\rm (B)}, and $\pi_{L,\,M}$
is as in \eqref{4.2}. Let $\fai$ be as in Definition \ref{d2.2} with
$\fai\in\rh_{(q_L/I(\fai))'}(\cx)$, where $q_L$
and $I(\fai)$ are, respectively, as in Assumption {\rm (B)} and \eqref{2.4}.
Then

{\rm(i)} the operator $\pi_{L,\,M}$, initially defined on the
space $T^{p,\,b}_2(\cx_+)$ with $p\in(p_L,q_L)$, extends to a bounded linear
operator from $T^p_2(\cx_+)$ to $L^p(\cx)$;

{\rm(ii)} the operator $\pi_{L,\,M}$, initially defined on the space
$T^b_{\fai}(\cx_+)$, extends to a bounded linear operator
from $T_{\fai}(\cx_+)$ to $H_{\fai,\,L}(\cx)$.
\end{proposition}

\begin{proof}
The proof of (i) is similar to that of \cite[Proposition 4.1(i)]{jy10}.
We omit the details. Now we prove (ii). Let
$f\in T^b_{\fai}(\cx_+)$. Then by Proposition \ref{p3.1},
Corollary \ref{c3.1} and (i), we know that
$$\pi_{L,\,M} f=\sum_j\lz_j\pi_{L,\,M}A_j
=:\sum_j\lz_j\az_j$$
in $L^2(\cx)$, where $\{\lz_j\}_j$ and $\{A_j\}_j$ satisfy
\eqref{3.1} and \eqref{3.2}. Recall that for each $j$,
$\supp A_j\subset\widehat{B_j}$ and $B_j$ is a ball of $\cx$. Moreover,
from the fact that $S_L$ is bounded on $L^2(\cx)$,
we deduce that for almost every  $x\in\cx$,
$S_L(\pi_{L,\,M}(f))(x)\le\sum_j|\lz_j|S_L(\az_j)(x)$. This,
combined with Lemma \ref{l2.1}(i), yields
\begin{equation*}
\int_{\cx}\fai\lf(x,S_L(\pi_{L,\,M}(f))(x)\r)\,d\mu(x)\ls
\sum_j\int_{\cx}\fai\lf(x,|\lz_j|S_L(\az_j)(x)\r)\,d\mu(x).
\end{equation*}

We now claim that for some $\epz\in(nq(\fai)/i(\fai),\fz)$,
$\az_j= \pi_{L,\,M}(A_j)$ is a $(\fai,\,M,\,\epz)_L$-molecule, up to a
harmless constant, associated to the ball $B_j$ for each $j$.
Indeed, assume that $A$ is a $(T_\fai,\fz)$-atom associated to
the ball $B:=B(x_B, r_B)$ and $q\in(p_L,q_L)$. Since for $q\in(p_L,2)$,
each $(\fai,\,2,\,M,\,\epz)_L$-molecule is also a $(\fai,\,q,\,M,\,\epz)_L$-molecule,
to prove the above claim, it suffices to show that $\az:=\pi_{L,\,M}(A)$
is a $(\fai,\,q,\,M,\,\epz)_L$-molecule, up to a harmless constant,
adapted to $B$ with $q\in[2,q_L)$.

Let $q\in[2,q_L)$. When $j\in\{0,\,\ldots,\,4\}$,
by (i), we know that
\begin{eqnarray}\label{4.4}
\|\az\|_{L^q(S_j(B))}&&=\|\pi_{L,\,M}A\|_{L^q(S_j(B))}
\ls\|A\|_{T^q_2(\cx_+)}\\ \nonumber
&&\ls[\mu(B)]^{1/q}
\|\chi_B\|^{-1}_{L^\fai(\cx)}\sim2^{-j\epz}[\mu(2^jB)]^{1/q}
\|\chi_B\|^{-1}_{L^\fai(\cx)}.
\end{eqnarray}

When $j\in\nn$ with $j\ge5$, take $h\in L^{q'}(\cx)$ satisfying
$\|h\|_{L^{q'}(\cx)}\le1$ and $\supp h\subset S_j(B)$.
Then from H\"older's inequality and
$q'\in(q_L',2]$, we infer that
\begin{eqnarray}\label{4.5}
\qquad|\langle \pi_{L,\,M} A,h\rangle|&&\!\!\le\int_\cx
\int_0^\fz |A(x,t)(t^{2m}L^\ast)^{M+1}e^{-t^{2m}L^\ast}(h)(x)|\frac{dt}
{t}\,d\mu(x)\\ \nonumber
&&\le\|\ca(A)\|_{L^q(\cx)}\lf\|\ca\lf(\chi_{\widehat{B}}(t^{2m}L^\ast)^{M+1}
e^{-t^{2m}L^\ast}(h)\r)\r\|_{L^{q'}(\cx)}\\ \nonumber
&&\ls\|A\|_{T^q_2(\cx_+)}
[\mu(B)]^{1/q'-1/2}\\ \nonumber
&&\hs\times\lf\{\int_{\widehat{B}}\lf|(t^{2m}L^\ast)^{M+1}e^{-t^{2m}L^\ast}
(h)(x,t)\r|^2\,\frac{d\mu(x)\,dt}{t}\r\}^{1/2}.
\end{eqnarray}
Moreover, by Assumption (B), we see that
\begin{eqnarray*}
&&\int_{\widehat{B}}\lf|(t^{2m}L^\ast)^{M+1}e^{-t^{2m}L^\ast}
(h)(x,t)\r|^2\,\frac{d\mu(x)\,dt}{t}\\
&&\hs\ls\int_0^{r_B}\lf\{2^{j\tz_1}\lf[\Upsilon\lf(\frac{2^jr_B}{t}\r)\r]^{\tz_2}
[\mu(B)]^{1/2}[\mu(2^jB)]^{-1/q'}\exp\lf[-\lf(\frac{2^jr_B}
{t}\r)^{2m/(2m-1)}\r]\r\}^2\frac{dt}{t}\\
&&\hs\ls2^{2\tz_1j}\mu(B)[\mu(2^jB)]^{-2/q'}\int_0^{r_B}
\lf(\frac{2^jr_B}{t}\r)^{-2(\epz+\tz_1)}\frac{dt}{t}\ls2^{-2\epz j}
\mu(B)[\mu(2^jB)]^{-2/q'},
\end{eqnarray*}
which, together with \eqref{4.5}, implies that
\begin{eqnarray*}
|\langle \pi_{L,\,M}a,h\rangle|\ls2^{-\epz j}[\mu(B)]^{1/q}
\|\chi_B\|^{-1}_{L^\fai(\cx)}
\ls2^{-\epz j}[\mu(2^jB)]^{1/q}
\|\chi_B\|^{-1}_{L^\fai(\cx)}.
\end{eqnarray*}
From this and the choice of $h$, we deduce that, for each
$j\in\nn$ with $j\ge5$,
\begin{eqnarray}\label{4.6}
\|\az\|_{L^{q}(S_j(B))}=
 \|\pi_{L,\,M}(a)\|_{L^{q}(S_j(B))}
 \ls2^{-\epz j}[\mu(2^jB)]^{1/q}
\|\chi_B\|^{-1}_{L^\fai(\cx)}.
\end{eqnarray}

Moreover, let $k\in\{1,\,\ldots,\,M\}$. When $j\in\{1,\,\ldots,\,4\}$,
take $h\in L^{q'}(\cx)$ satisfying
$\|h\|_{L^{q'}(\cx)}\le1$ and $\supp h\subset S_j(B)$. Then it follows, from H\"older's
inequality and the $L^{q'}(\cx)$-boundedness of $S_{L^\ast,\,M+1-k}$, that
\begin{eqnarray*}
&&|\langle(r_B^{-2m}L^{-1})^k\pi_{L,\,M}(a),h\rangle|\\
&&\hs\ls\int_0^{r_B}\int_B
\lf(\frac{t}{r_B}\r)^{2km}|a(x,t)|\lf|(t^{2m}L^\ast)^{M+1-k}e^{-t^{2m}
L^\ast}(h)(x)\r|\,\frac{d\mu(x)\,dt}{t}\\
&&\hs\ls\|\ca(a)\|_{L^q(\cx)}\lf\|S_{L^\ast,\,M+1-k}(h)\r\|_{L^{q'}(\cx)}\\
&&\hs\ls\|a\|_{T^q_2(\cx_+)}\ls[\mu(B)]^{1/q}\|\chi_B\|^{-1}_{L^\fai(\cx)}\ls
2^{-j\epz}[\mu(2^jB)]^{1/q}\|\chi_B\|^{-1}_{L^\fai(\cx)},
\end{eqnarray*}
which implies that, for each $k\in\{1,\,\ldots,\,M\}$ and $j\in\{0,\,\ldots,\,4\}$,
\begin{eqnarray}\label{4.7}
\lf\|(r^{-2m}_BL^{-1})^k\az\r\|_{L^{q}(S_j(B))}
 \ls2^{-j\epz}[\mu(2^jB)]^{1/q}\|\chi_B\|^{-1}_{L^\fai(\cx)}.
\end{eqnarray}
When $j\in\nn$ with $j\ge5$, similar to the proof of \eqref{4.5}, we know that, for each
$m\in\{1,\,\ldots,\,M\}$,
$$\lf\|\lf(r^{-2m}_BL^{-1}\r)^k\az\r\|_{L^{q}(S_j(B))}\ls2^{-\epz j}[\mu(2^jB)]^{1/q}
\|\chi_B\|^{-1}_{L^\fai(\cx)},
$$
which, together with \eqref{4.4}, \eqref{4.6} and \eqref{4.7}, implies that
$\az$ is a $(\fai,\,q,\,M,\,\epz)_L$-molecule.

Let $\epz>nq(\fai)/i(\fai)$ and $M\in\nn$ with $M>\frac{n}{2m}[\frac{q(\fai)}{i(\fai)}+
\frac{\tz_1}{n}-\frac{2}{nq_L}]$. By $\fai\in\rh_{(q_L/I(\fai))'}(\cx)$,
$\epz>nq(\fai)/i(\fai)$ and $M>\frac{n}{2m}[\frac{q(\fai)}{i(\fai)}+\frac{\tz_1}{n}
-\frac{2}{nq_L}]$, we find that there exist $p_1\in[I(\fai),1]$,
$p_2\in(0,i(\fai))$, $q_0\in(q(\fai),\fz)$ and $q\in[2,q_L)$
such that $\fai$ is of uniformly upper type $p_1$ and lower type $p_2$,
$\fai\in\aa_{q_0}(\cx)$, $\fai\in\rh_{(q/p_1)'}(\cx)$, $\epz>nq_0/p_2$ and
$M>\frac{n}{2m}[\frac{q_0}{p_2}+\frac{\tz_1}{n}
-\frac{2}{nq}]$.
We now claim that, for any $\lz\in\cc$
and $(\fai,\,q,\,M,\epz)_L$-molecule $\az$ associated with the ball $B\subset\cx$,
\begin{equation}\label{4.8}
\int_{\cx}\fai(x,S_L(\lz
\az)(x))\,d\mu(x)\ls\fai\lf(B,\frac{|\lz|}{\|\chi_B\|_{L^\fai(\cx)}}\r).
\end{equation}

If \eqref{4.8} holds, from this, the facts that, for all $\lz\in(0,\fz)$,
$$S_L(\pi_{L,\,M}(f/\lz))=S_L(\pi_{L,\,M}(f))/\lz\
\text{and} \ \pi_{L,\,M}(f/\lz)=\sum_j\lz_j\az_j/\lz,$$
and $S_L(\pi_{L,\,M}(f))\le\sum_j|\lz_j|S_L(\az_j)$,
it follows that, for all $\lz\in(0,\fz)$,
$$\int_{\cx}\fai\lf(x,\frac{S_L(\pi_{L,\,M}(f))(x)}
{\lz}\r)\,d\mu(x)\ls\sum_j
\fai\lf(B_j,\frac{|\lz_j|}{\lz\|\chi_{B_j}\|_{L^\fai(\cx)}}\r),$$
which, together with \eqref{3.2}, implies that
$$\|\pi_{L,\,M}f\|_{H_{\fai,\,L,\,2}(\cx)}\ls
\blz(\{\lz_j\az_j\}_j)\ls\|f\|_{T_{\fai}(\cx_+)},$$
and hence completes the proof of (ii).

Now we prove \eqref{4.8}. By the definition of $\az$, we see that
\begin{eqnarray}\label{4.9}
\hs\hs\hs&&\int_\cx\fai(x,S_L(\lz\az)(x))\,d\mu(x)\\ \nonumber
&&\hs\ls\sum_{j=0}^\fz
\int_{\cx}\fai\lf(x,\lf\{\int_0^{r_B}\int_{B(x,t)}\lf|t^{2m}Le^{-t^{2m}L}
\lf(\lz\az\chi_{S_j(B)}\r)(y)\r|^2\frac{d\mu(y)\,dt}
{V(x,t)t}\r\}^{1/2}\r)\,dx\\ \nonumber
&&\hs\hs+\sum_{j=0}^\fz
\int_{\cx}\fai\lf(x,\lf\{\int_{r_B}^\fz\int_{B(x,t)}\lf|t^{2m}Le^{-t^{2m}L}
(r^{2m}_B L)^M\lf(\lz\chi_{S_j(B)}(r^{2m}_B L)^{-M}\az\r)(y)\r|^2\r.\r.\\ \nonumber
&&\hs\hs\lf.\lf.\times\frac{d\mu(y)\,dt}{V(x,t)t}\r\}^{1/2}\r)\,dx
=:\sum_{j=0}^\fz\mathrm{E}_j+\sum_{j=0}^\fz\mathrm{F}_j.
\end{eqnarray}
For any $j\in\zz_+$, let $B_j:=2^j B$. Then
\begin{eqnarray*}
\hs\hs\mathrm{E}_j&&=\sum_{i=0}^{\fz}\int_{S_i(B_j)}\fai\lf(x,|\lz|
\lf\{\int_0^{r_B}\int_{B(x,t)}
\lf|t^{2m}Le^{-t^{2m}L}\lf(\az\chi_{S_j(B)}\r)(y)\r|^2\r.\r.\\ \nonumber
&&\hs\times\lf.\lf.\frac{d\mu(y)\,dt}{V(x,t)t}\r\}^{1/2}\r)\,dx
=:\sum_{i=0}^{\fz}\mathrm{E}_{i,\,j}.
\end{eqnarray*}

When $i\in\{0,\,1,\,\ldots,\,4\}$, by the uniformly upper type $p_1$ and lower type
$p_2$ properties of $\fai$, we see that
\begin{eqnarray}\label{4.11}
\mathrm{E}_{i,\,j}&&\ls
\|\chi_B\|_{L^\fai(\cx)}^{p_1}\int_{S_i(B_j)}\fai\lf(x,|\lz|
\|\chi_B\|_{L^\fai(\cx)}^{-1}\r)\lf[S_L\lf(\az\chi_{S_j(B)}\r)(x)\r]^{p_1}
\,d\mu(x)\\ \nonumber &&\hs+
\|\chi_B\|_{L^\fai(\cx)}^{p_2}\int_{S_i(B_j)}\fai\lf(x,|\lz|
\|\chi_B\|_{L^\fai(\cx)}^{-1}\r)\lf[S_L\lf(\az\chi_{S_j(B)}\r)(x)\r]^{p_2}\,d\mu(x)\\
\nonumber &&=: \mathrm{G}_{i,\,j}+\mathrm{H}_{i,\,j}.
\end{eqnarray}

Now we estimate $\mathrm{G}_{i,\,j}$. From H\"older's inequality,
Theorem \ref{thm on S_L}, $\fai\in\rh_{(q/p_1)'}(\cx)$ and Lemma \ref{l2.4}(vi),
we deduce that
\begin{eqnarray}\label{4.12}
\mathrm{G}_{i,\,j}&&\ls\|\chi_B\|^{p_1}_{L^\fai(\cx)}
\lf\{\int_{U_i(B_j)}\lf[S_L\lf(\az\chi_{S_j(B)}\r)(x)\r]^q\,
d\mu(x)\r\}^{p_1/q}\\ \nonumber&&\hs\times\lf\{\int_{S_i(B_j)}\lf[\fai\lf(x,|\lz|
\|\chi_B\|_{L^\fai(\cx)}^{-1}\r)\r]^{(q/p_1)'}\,d\mu(x)\r\}^{\frac{1}{(q/p_1)'}}\\
\nonumber &&\ls\|\chi_B\|^{p_1}_{L^\fai(\cx)}
\|\az\|^{p_1}_{L^q(S_j(B))}[\mu(2^{i+j}B)]^{-p_1/q}\fai\lf(2^{i+j}B,|\lz|
\|\chi_B\|_{L^\fai(\cx)}^{-1}\r)\\ \nonumber
&&\ls2^{-jp_1[\epz-nq_0/p_1]}\fai\lf(B,|\lz|
\|\chi_B\|_{L^\fai(\cx)}^{-1}\r).
\end{eqnarray}
For $\mathrm{H}_{i,\,j}$, similarly, we have
\begin{eqnarray*}
\mathrm{H}_{i,\,j}\ls2^{-jp_2(\epz-nq_0/p_2)}\fai\lf(B,|\lz|
\|\chi_B\|_{L^\fai(\cx)}^{-1}\r),
\end{eqnarray*}
which, together with \eqref{4.11} and \eqref{4.12}, implies that, for each
$j\in\zz_+$ and $i\in\{0,\,1,\,\ldots,\,4\}$,
\begin{eqnarray}\label{4.13}
\mathrm{E}_{i,\,j}\ls 2^{-jp_2(\epz-nq_0/p_2)}\fai\lf(B,|\lz|
\|\chi_B\|_{L^\fai(\cx)}^{-1}\r).
\end{eqnarray}

For all $j\in\zz_+$ and $x\in\cx$, let
$$\mathrm{H}_j(x):=\lf\{\int_0^{r_B}\int_{B(x,t)}\lf|t^{2m}Le^{-t^{2m}L}
\lf(\az\chi_{S_j(B)}\r)(y)\r|^2\frac{d\mu(y)\,dt}{V(x,t)t}\r\}^{1/2}.
$$
Now we estimate $\int_{S_i(B_j)}[\mathrm{H}_j(x)]^q\,d\mu(x)$. For any $i,\,j\in\zz_+$,
let
$$\widetilde{S}_i(B_j):=\lf\{y\in\cx:\ 2^{i-3}2^jr_B\le d(y,r_B)\le2^{i+1}2^jr_B\r\}.$$
It is easy to see that when $i\ge5$, $d(S_j(B),S_i(B_j))\gs2^{i+j}r_B$.
By $M>\frac{n}{2m}(\frac{q_0}{p_2}+\frac{\tz_1}{n}
-\frac{2}{nq_L})$, we know that $2mMq+\tz_2 q+q/2+1>(\frac{nq_0}{p_2}+\frac{1}{2}+
\tz_1+\tz_2)q-1$. Let $s\in([\frac{nq_0}{p_2}+\frac{1}{2}+\tz_1+\tz_2]q-1,
2mMq+\tz_2 q+q/2+1)$. Then by H\"older's inequality, Fubini's theorem and Assumption
{\rm (B)}, we conclude that
\begin{eqnarray}\label{4.14}
&&\int_{S_i(B_j)}[\mathrm{H}_j(x)]^q\,d\mu(x)\\ \nonumber
&&\hs\le\int_{S_i(B_j)}\lf\{\int_0^{r_B}\int_{B(x,t)}\lf|t^{2m}Le^{-t^{2m}L}
\lf(\az\chi_{S_j(B)}\r)(y)\r|^q\frac{d\mu(y)\,dt}{V(x,t)t^{q/2}}\r\}\\ \nonumber
&&\hs\hs\times\lf\{\int_0^{r_B}\int_{B(x,t)}\frac{d\mu(y)\,dt}
{V(x,t)}\r\}^{(q-2)/2}\,d\mu(x)\\ \nonumber
&&\hs\ls r_B^{(q-2)/2}\int_0^{r_B}
\int_{\widetilde{S}_i(B_j)}\lf|t^{2m}Le^{-t^{2m}L}
\lf(\az\chi_{S_j(B)}\r)(y)\r|^q\frac{d\mu(y)\,dt}{t^{q/2}}\\ \nonumber
&&\hs\ls r_B^{(q-2)/2}\int_0^{r_B}\Bigg\{2^{i\tz_1}\lf[\Upsilon\lf(\frac{2^{i+j}r_B}{t}\r)
\r]^{\tz_2}[\mu(2^{i+j}B)]^{1/q}[\mu(2^jB)]^{-1/q}\\ \nonumber
&&\hs\hs\times e^{-(\frac{2^{i+j}r_B}{t})^{2m/(2m-1)}}\|\az\|_{L^q(S_j(B))}\Bigg\}^q
\frac{dt}{t^{q/2}}\\ \nonumber
&&\hs\ls r_B^{(q-2)/2}2^{i\tz_1q}2^{-j\epz q}(2^{i+j}r_B)^{\tz_2q}
\|\chi_B\|^{-q}_{L^\fai(\cx)}\mu(2^{i+j}B)\\ \nonumber
&&\hs\hs\times\lf\{\int_0^{r_B}\lf(\frac{t}
{2^{i+j}r_B}\r)^st^{-(\tz_2q+q/2)}\,dt\r\}\\\nonumber
&&\hs\ls2^{-i[s-(\tz_1+\tz_2)q]}2^{-j(s+\epz-\tz_2q)}\mu(2^{i+j}B)
\|\chi_B\|^{-q}_{L^\fai(\cx)}.
\end{eqnarray}

By using \eqref{4.14}, similar to the proof of \eqref{4.13},
we know that for any $j\in\zz_+$ and $i\in\nn$ with $i\ge5$,
\begin{eqnarray}\label{4.15}
\mathrm{E}_{i,\,j}\ls2^{-p_2[s/q-(\tz_1+\tz_2)-nq_0/p_2]i}
2^{-p_2(s/q+\epz-\tz_2-nq_0/p_2)j}\fai\lf(B,
|\lz|\|\chi_B\|^{-1}_{L^\fai(\cx)}\r).
\end{eqnarray}

Now we deal with $\mathrm{F}_{j}$. Let
\begin{eqnarray*}
\hs\hs\mathrm{F}_j&&=\sum_{i=0}^{\fz}\int_{S_i(B_j)}\fai\lf(x,|\lz|
\lf\{\int_{r_B}^\fz\int_{B(x,t)}
\lf|t^{2m}Le^{-t^{2m}L}(r^{2m}L)^M
\bigg(\chi_{S_j(B)}\r.\r.\r.\\ \nonumber
&&\hs\times\lf.\lf.\lf.\lf.\lf(r^{2m}_BL\r)^{-M}\az\r)(y)\r|^2
\frac{d\mu(y)\,dt}{V(x,t)t}\r\}^{1/2}\r)\,d\mu(x)
=:\sum_{i=0}^{\fz}\mathrm{F}_{i,\,j}.
\end{eqnarray*}
When $i\in\{0,\,1,\,\ldots,\,4\}$, similar to the proof of \eqref{4.13},
we conclude that
\begin{eqnarray}\label{4.17}
\mathrm{F}_{i,\,j}\ls 2^{-jp_2(\epz-nq_0/p_2)}\fai\lf(B,|\lz|
\|\chi_B\|_{L^\fai(\cx)}^{-1}\r).
\end{eqnarray}
For each $j\in\zz_+$ and all $x\in\cx$, let
$$\mathrm{G}_{j}(x):=\lf\{\int_{r_B}^\fz\int_{B(x,t)}
\lf|t^{2m}Le^{-t^{2m}L}\lf(r^{2m}L\r)^M\lf(\chi_{S_j(B)}
\lf(r^{-2m}_BL^{-1}\r)^M\az\r)(y)\r|^2\frac{d\mu(y)\,dt}{V(x,t)t}\r\}^{1/2}.
$$
Now we estimate $\int_{S_i(B_j)}[\mathrm{G}_j(x)]^q\,d\mu(x)$. We
first see that, for all $x\in\cx$,
\begin{eqnarray}\label{4.18}
\hs\hs\hs\hs\mathrm{G}_{j}(x)&&\le\lf\{\int_{r_B}^{2^{i+j-3}r_B}\int_{B(x,t)}
\lf|t^{2m}Le^{-t^{2m}L}(r^{2m}L)^M\lf(\chi_{S_j(B)}
\lf(r^{-2m}_BL^{-1}\r)^M\az\r)(y)\r|^2\r.\\ \nonumber
&&\hs\times\lf.\frac{d\mu(y)\,dt}{V(x,t)t}\r\}^{1/2}
+\lf\{\int_{2^{i+j-3}r_B}^\fz\cdots\r\}^{1/2}=:\mathrm{G}_{j,\,1}(x)
+\mathrm{G}_{j,\,2}(x).
\end{eqnarray}

For $\mathrm{G}_{j,\,1}$, similar to \eqref{4.14},
we conclude that, when $i\in\nn$ with $i\ge5$,
\begin{eqnarray}\label{4.19}
\int_{S_i(B_j)}[\mathrm{G}_{j,\,1}(x)]^q\,d\mu(x)
&&\ls2^{-i[s+1-(\tz_1+\tz_2+1/2)q]}\\ \nonumber
&&\hs\times2^{-j[s+1+\epz-(\tz_2+1/2)q]}\mu(2^{i+j}B)
\|\chi_B\|^{-q}_{L^\fai(\cx)}.
\end{eqnarray}
For $\mathrm{G}_{j,\,2}$, by Theorem \ref{thm on S_L},
we find that
\begin{eqnarray*}
\int_{S_i(B_j)}[\mathrm{G}_{j,\,2}(x)]^q\,d\mu(x)
&&\ls\frac{r_B^{2mMq}}{(2^{i+j}r_B)^{2mMq}}\int_{S_i(B_j)}\lf[S_{L,\,M+1}
\lf(\chi_{S_j(B)}(r^{2m}_B L)^{-M}\az\r)(x)\r]^q\,d\mu(x)\\
&&\ls2^{-2mMq(i+j)}\lf\|(r^{2m}_B L)^{-M}\az\r\|^q_{L^q(S_j(B))}\\
&&\ls2^{-2mMqi}2^{-j(2mMq+\epz q)j}\mu(2^jB)\|\chi_B\|^{-q}_{L^\fai(\cx)},
\end{eqnarray*}
which, together with \eqref{4.18} and \eqref{4.19}, implies that
\begin{eqnarray*}
\int_{S_i(B_j)}[\mathrm{G}_{j}(x)]^q\,d\mu(x)
&&\ls2^{-i[s+1-(\tz_1+\tz_2+1/2)q]}2^{-j[s+1+\epz-(\tz_2+1/2)q]}
\mu(2^{i+j}B)\|\chi_B\|^{-q}_{L^\fai(\cx)}.
\end{eqnarray*}
By using this estimate, similar to the proof of \eqref{4.15},
we see that, for all $j\in\zz_+$ and $i\in\nn$ with $i\ge5$,
\begin{eqnarray*}
\mathrm{F}_{i,\,j}\ls2^{p_2[(s+1)/q-(\tz_1+\tz_2+1/2)-nq_0/p_2]}
2^{p_2[(s+1)/q+\epz-(\tz_1+1/2)-nq_0/p_2]}\fai\lf(B,|\lz|\|\chi_B\|^{-1}_{L^\fai(\cx)}\r),
\end{eqnarray*}
which, together with \eqref{4.9} through \eqref{4.17} and $s>[\frac{nq_0}{p_2}+
\frac{1}{2}+\tz_1+\tz_2]q-1$,
implies that \eqref{4.8} holds true, and hence completes
the proof of Proposition \ref{p4.1}.
\end{proof}

\begin{proposition}\label{p4.2}
Let $\fai$ be as in Definition \ref{d2.2}, $L$ satisfy Assumptions {\rm (A)} and {\rm (B)},
$\epz\in(nq(\fai)/i(\fai),\fz)$ and
$M\in\nn$ with $M>\frac{n}{2m}[\frac{q(\fai)}{i(\fai)}+
\frac{\tz_1}{n}-\frac{2}{nq_L}]$.
Then, for all $f\in H_{\fai,\,L}(\cx)
\cap L^2(\cx)$, there exist $\{\lz_j\}_j\subset\cc$
and a sequence $\{\az_j\}_j$ of $(\fai,\,M,\,\epz)_L$-molecules, respectively,
associated with the balls $\{B_j\}_j$ such that
$f=\sum_j\lz_j\az_j$
in both $H_{\fai,\,L}(\cx)$ and $L^2(\cx)$. Moreover, there exists
a positive constant $C$ such that, for all $f\in
H_{\fai,\,L}(\cx)\cap L^2(\cx)$,
$$\blz(\{\lz_j\az_j\}_j):=\inf\lf\{\lz\in(0,\fz):\ \sum_j
\fai\lf(B_j,\frac{|\lz_j|}{\lz\|\chi_{B_j}\|_{L^\fai(\cx)}}\r)\le1\r\}\le
C\|f\|_{H_{\fai,\,L}(\cx)}.$$
\end{proposition}

\begin{proof}
Let $f\in H_{\fai,\,L}(\cx)\cap L^2(\cx)$. Then by the
$H_{\fz}$-functional calculi for $L$ and \eqref{4.1}, we know that
$$f=C_{(m,\,M)}\int_0^\fz(t^{2m}L)^{M+2}e^{-2t^{2m}L}f\frac{dt}{t}
=\pi_{L,\,M}\lf(t^{2m}Le^{-t^{2m}L}f\r)
$$
in $L^2(\cx)$. Moreover, from Definition \ref{d4.1} and
the $L^2(\cx)$-boundedness of $S_L$, we infer that $t^{2m}Le^{-t^{2m}L}f\in
T_\fai(\cx_+)\cap T^2_2(\cx_+)$. Applying Theorem
\ref{t3.1}, Corollary \ref{c3.1} and
Proposition \ref{p4.1} to $t^{2m}Le^{-t^{2m}L}f$, we conclude that
$$f=\pi_{L,\,M}(t^{2m}Le^{-t^{2m}L}f)=\sum_j\lz_j
\pi_{L,\,M}A_j=:\sum_j\lz_j\az_j
$$
in $L^2(\cx)\cap H_{\fai,\,L}(\cx)$, and
$\blz(\{\lz_j\az_j\}_j)\ls\|t^{2m}Le^{-t^{2m}L}f\|_{T_\fai(\cx_+)}\sim
\|f\|_{H_{\fai,\,L}(\cx)}$. Furthermore, by the proof of
Proposition \ref{p4.1}, we know that, for each $j$, $\az_j$ is a
$(\fai,\,M,\,\epz)_L$-molecule  up to a harmless constant, which completes the
proof of Proposition \ref{p4.2}.
\end{proof}

The proofs of Propositions \ref{p4.1} and
\ref{p4.2} imply immediately the following corollary.

\begin{corollary}\label{sufficient condition for molecule}
Let $\fai$ be as in Definition \ref{d2.2}, $L$ satisfy Assumptions ${\rm (A)}$
and ${\rm (B)}$, $p_L$ and $q_L$ be as in Assumption ${\rm (B)}$, $q\in(p_L,\,q_L)$ and
$M\in\nn$ satisfying $M>\frac{n}{2m}[\frac{q(\fai)}{i(\fai)}+
\frac{\tz_1}{n}-\frac{2}{nq}]$, where
$q(\fai)$, $i(\fai)$ and $\tz_1$ are respectively as in \eqref{2.7},
\eqref{2.5} and \eqref{2.11}.
Suppose that $T$ is a linear (resp. nonnegative
sublinear) operator which maps $L^2(\cx)$ continuously into weak $L^2(\cx)$.
If there exists a positive constant $C$ such that, for all $\lz\in\cc$ and $(\fai,\,
q,\,M,\,\epsilon)_L$-molecule $\az$ associated with the ball $B$,
\begin{eqnarray*}
\int_{\cx}\fai\lf(x,\, T(\lz \az)(x)\r)\,d\mu(x)\le C \fai\lf(B,\,\frac{|\lz|}
{\|\chi_B\|_{L^\fai(\cx)}}\r),
\end{eqnarray*}
then $T$ can extend to be a bounded linear (resp. sublinear) operator from
$H^{M,\,q,\,\epz}_{\fai,\,L}(\cx)$ to $L^\fai(\cx)$.
\end{corollary}

\begin{theorem}\label{t4.1}
Let $\fai$ be as in Definition \ref{d2.2} and $L$ satisfy Assumptions {\rm (A)} and {\rm (B)}.
Assume that $q\in[2,q_L)\cap([r(\fai)]'I(\fai),q_L)$, $M\in\nn$ with $M>\frac{n}{2m}
[\frac{q(\fai)}{i(\fai)}+\frac{\tz_1}{n}-\frac{2}{nq}]$, $\epz\in(nq(\fai)/i(\fai),\fz)$,
where $q_L$, $r(\fai)$, $I(\fai)$, $q(\fai)$ and $i(\fai)$ are, respectively, as in
Assumption (B), \eqref{2.8}, \eqref{2.4}, \eqref{2.7} and \eqref{2.5}. Then
$H_{\fai,\,L}(\cx)$ and $H^{M,\,q,\,\epz}_{\fai,\,L}(\cx)$ coincide
with equivalent quasi-norms.
\end{theorem}

\begin{proof}
We first prove that
$$\widetilde{H}^{M,\,q,\,\epz}_{\fai,\,L}(\cx)\cap L^2(\cx)
\subset \widetilde{H}_{\fai,\,L,\,2}(\cx)$$
and the inclusion is continuous.
Let $f\in\widetilde{H}^{M,\,q,\,\epz}_{\fai,\,L}(\cx)\cap L^2(\cx)$.
Then there exist $\{\lz_j\}_{j\in\nn}\subset\cc$ and a sequence $\{\az_j\}_{j\in\nn}$ of
$(\fai,\,q,\,M,\,\epz)_L$-molecules such that
$f=\sum_{j=1}^\infty \lz_j \az_j$, where the summation converges in $L^r(\cx)$ for some
$r\in(p_L,q_L)$. By Theorem \ref{thm on S_L}, we see that, for each
$i\in\nn$, $S_L(\sum_{j=1}^i \lz_j \az_j -f)(x)\le \sum_{j=i+1}^\infty
|\lz_j| S_L(\az_j)(x)$ for almost every $x\in\cx$, which, together with \eqref{4.8} and
$f\in H^{M,\,q,\,\epz}_{\fai,\,L}(\cx)$, implies that
\begin{equation}\label{4.20}
S_L\lf(\sum_{j=1}^i \lz_j \az_j -f\r)\in L^\fai(\cx)
\end{equation}
and
$$\lf\|S_L\lf(\sum_{j=1}^i\lz_j \az_j -f\r)\r\|_{L^\fai(\cx)}\to0$$
as $i\to\fz$. Moreover, by the definition of $(\fai,\,q,\,M,\,\epz)_L$-molecules,
H\"older's inequality
and $q\ge2$, we conclude that for each $j\in\nn$, $\az_j\in L^2(\cx)$, which, together with
$f\in L^2(\cx)$, implies that, for any $i\in\nn$, $f-\sum_{j=1}^i\lz_j\az_j\in L^2(\cx)$.
From this and \eqref{4.20}, it follows that $f\in\widetilde{H}_{\fai,\,L,\,2}(\cx)$.
Furthermore, by the fact that $S_L(f)\le\sum_{j=1}^\fz|\lz_j|S_L(\az_j)$ and \eqref{4.8},
we see that
\begin{equation}\label{4.21}
\|f\|_{H_{\fai,\,L}(\cx)}\ls\|f\|_{H^{M,\,q,\,\epz}_{\fai,\,L}(\cx)}.
\end{equation}

Now we prove that $\widetilde{H}_{\fai,\,L,\,2}(\cx)\subset\widetilde{
H}^{M,\,q,\,\epz}_{\fai,\,L}(\cx)\cap L^2(\cx)$ and the inclusion is continuous.
Let $f\in\widetilde{H}_{\fai,\,L,\,2}(\cx)$.
Then by Proposition \ref{p4.2}, we know that
there exist $\{\lz_j\}_j\subset\cc$ and a sequence $\{\az_j\}_j$ of
$(\fai,\,q,\,M,\,\epz)_L$-molecules such that
$f=\sum_j\lz_j\az_j$
in $H_{\fai,\,L}(\cx)\cap L^2(\cx)$ and $\blz(\{\lz_j\az_j\}_j)
\ls\|f\|_{H_{\fai,\,L}(\cx)}$,
which implies that $f\in\wz{H}^{M,\,q,\,\epz}_{\fai,\,L}(\cx)\cap L^2(\cx)$ and
\begin{equation*}
\|f\|_{H^{M,\,q,\,\epz}_{\fai,\,L}(\cx)}\ls\|f\|_{H_{\fai,\,L}(\cx)}.
\end{equation*}
From this and \eqref{4.21},
we infer that
$\wz{H}_{\fai,\,L,\,2}(\cx)=\wz{H}^{M,\,q,\,\epz}_{\fai,\,L}(\cx)\cap L^2(\cx)$ and, for all
$f\in\wz{H}_{\fai,\,L,\,2}(\cx)$,
$$\|f\|_{H_{\fai,\,L}(\cx)}\sim\|f\|_{H^{M,\,q,\,\epz}_{\fai,\,L}(\cx)}.$$

To finish the proof of Theorem \ref{t4.1}, it suffices to prove that
$\widetilde{H}_{\fai,\,L}(\cx)$ and $\widetilde{H}^{M,\,q,\,\epz}_{\fai,\,L}(\cx)
\cap L^2(\cx)$ are dense in $H_{\fai,\,L}(\cx)$ and
$H^{M,\,q,\,\epz}_{\fai,\,L}(\cx)$, respectively.
Indeed, if these hold true, by these and a standard density argument, we conclude that
$H_{\fai,\,L,\,2}(\cx)$ and $H^{M,\,q,\,\epz}_{\fai,\,L}(\cx)$ coincide with
equivalent norms. Obviously, $\widetilde{H}_{\fai,\,L,\,2}(\cx)$ is dense in
$H_{\fai,\,L}(\cx)$. Now we prove that $\widetilde{H}^{M,\,q,\,\epz}_{\fai,\,L}(\cx)
\cap L^2(\cx)$ is dense in $H^{M,\,q,\,\epz}_{\fai,\,L}(\cx)$.
Let $f\in\widetilde{H}^{M,\,q,\,\epz}_{\fai,\,L}(\cx)$. Then there exist a sequence
$\{\lz_j\}_{j\in\nn}\subset\cc$ and a sequence $\{\az_j\}_{j\in\nn}$ of
$(\fai,\,q,\,M,\,\epz)_L$-molecules such that $f=\sum_{j\in\nn}\lz_j\az_j$ in
$L^r(\cx)$ with some  $r\in(p_L,q_L)$.
For any $N\in\nn$, let $f_N:=\sum_{j=1}^N\lz_j\az_j$. From the definition of $(\fai,\,q,\,
M,\,\epz)$-molecules, $q\ge2$ and H\"older's inequality, we deduce that, for all $j\in\nn$,
$\az_j\in L^2(\cx)$, which implies that, for any $N\in\nn$, $f_N\in L^2(\cx)$.
Thus, for any $N\in\nn$, $f_N\in\widetilde{H}^{M,\,q,\,\epz}_{\fai,\,L}(\cx)\cap L^2(\cx)$,
and $\|f-f_N\|_{H^{M,\,q,\,\epz}_{\fai,\,L}(\cx)}\to0$ as $N\to\fz$. By this,
we see that $\widetilde{H}^{M,\,q,\,\epz}_{\fai,\,L}(\cx)\cap L^2(\cx)$ is dense in
$\widetilde{H}^{M,\,q,\,\epz}_{\fai,\,L}(\cx)$ and hence dense in $H^{M,\,q,\,
\epz}_{\fai,\,L}(\cx)$. This finishes the proof of Theorem \ref{t4.1}.
\end{proof}

\begin{theorem}\label{t4.2}
Let $\fai$ be as in Definition \ref{d2.2} and $L$ satisfy Assumptions {\rm (A)} and
{\rm (B)}. Assume that $\fai\in\rh_{(q_L/I(\fai))'}(\cx)$.
Then the spaces $H_{\fai,\,L}(\cx)$ and $H_{\fai,\,L,\,s}(\cx)$, with $s\in(p_L,q_L)$,
coincide with equivalent quasi-norms.
\end{theorem}

\begin{proof}
Let $s\in(p_L,q_L)$. By the definitions of the spaces $H_{\fai,\,L}(\cx)$ and
$H_{\fai,\,L,\,s}(\cx)$, we see that $\widetilde{H}_{\fai,\,L,\,2}(\cx)\cap L^s(\cx)$ and
$\widetilde{H}_{\fai,\,L,\,s}(\cx)\cap L^2(\cx)$ coincide with equivalent norms. Similar to
the proof of Theorem \ref{t4.1}, we need to prove that $\widetilde{H}_{\fai,\,L,\,2}(\cx)
\cap L^s(\cx)$  and $\widetilde{H}_{\fai,\,L,\,s}(\cx)
\cap L^2(\cx)$ are dense in $\widetilde{H}_{\fai,\,L,\,2}(\cx)$ and
$\widetilde{H}_{\fai,\,L,\,s}(\cx)$, respectively.

We first prove that $\widetilde{H}_{\fai,\,L,\,2}(\cx)\cap L^s(\cx)$ is dense in
$\widetilde{H}_{\fai,\,L,\,2}(\cx)$. Let $f\in\widetilde{H}_{\fai,\,L,\,2}(\cx)$. Then by
Proposition \ref{p4.1}, we know that there exist $\{\lz_j\}_{j\in\nn}\subset\cc$ and a
sequence $\{\az_j\}_{j\in\nn}$ of $(\fai,\,q,\,M,\,\epz)_L$-molecules,
with $q\in(\max\{s,\,2\},q_L)$, such that
$f=\sum_{j=1}^\fz\lz_j\az_j$
in $H_{\fai,\,L}(\cx)\cap L^2(\cx)$. From $q\in(s,q_L)\cap[2,\fz)$ and H\"older's
inequality, we deduce that, for each $j\in\nn$, $\az_j$ is a $(\fai,\,2,\,M,\,
\epz)_L$-molecule and also a $(\fai,\,s,\,M,\,\epz)_L$-molecule, which implies that
for any $N\in\nn$, $\sum_{j=1}^N\lz_j\az_j\in L^2(\cx)\cap L^s(\cx)$. Moreover,
by \eqref{4.8}, we see that $S_L(\sum_{j=1}^N\lz_j\az_j)\in L^\fai(\cx)$. Thus, for
any $N\in\nn$, $\sum_{j=1}^N\lz_j\az_j\in \widetilde{H}_{\fai,\,L,\,2}(\cx)\cap L^s(\cx)$.
Furthermore, from $f=\sum_{j=1}^\fz\lz_j\az_j$ in $H_{\fai,\,L}(\cx)$, we infer that
$\|f-\sum_{j=1}^N\lz_j\az_j\|_{H_{\fai,\,L}(\cx)}\to0$ as $N\to\fz$. Thus,
$\widetilde{H}_{\fai,\,L,\,2}(\cx)\cap L^s(\cx)$ is dense in
$\widetilde{H}_{\fai,\,L,\,2}(\cx)$.

Let $f\in\widetilde{H}_{\fai,\,L,\,s}(\cx)$. By the definition of $H_{\fai,\,L,\,s}(\cx)$,
we see that $t^{2m}Le^{-t^{2m}L}f\in T_{\fai}(\cx_+)$. For any $N\in\nn$,
let $f_N:=\pi_{L,\,M}(t^{2m}Le^{-t^{2m}L}f\chi_{O_N})$, where
$$O_N:=\lf\{(y,t)\in\cx
\times(0,\fz):\ d(y,x_0)<N,\ t\in(N^{-1},N)\r\}$$
with some $x_0\in\cx$. Then from Proposition
\ref{p3.1}, we infer that $t^{2m}Le^{-t^{2m}L}f\chi_{O_N}\in T_\fai(\cx_+)\cap
T_2^2(\cx_+)$, which implies that $f_N\in \widetilde{H}_{\fai,\,L,\,2}(\cx)$.
Moreover, by $f\in L^s(\cx)$ and the $L^s(\cx)$-boundedness of $S_L$, we conclude that
$S_L(f)\in L^s(\cx)$, which implies that $t^{2m}Le^{-t^{2m}L}f\in T^{s}_2(\cx_+)$.
From this and the definition of $T^s_2(\cx_+)$, it follows that
$t^{2m}Le^{-t^{2m}L}f\chi_{O_N}\in T^{s,\,b}_2(\cx_+)$, which, together with
proposition \ref{p4.1}(i), implies that $f_N\in L^s(\cx)$. Thus,
$f_N\in\widetilde{ H}_{\fai,\,L,\,2}(\cx)\cap L^s(\cx)$. Moreover,
\begin{eqnarray*}
\|S_L(f_N - f)\|_{H_{\fai,\,L,\,s}(\cx)}\ls\lf\|t^{2m} L e^{-t^{2m} L}f\chi_{
(O_N)^\complement}\r\|_{T_\fai(\cx_+)}\to 0,
\end{eqnarray*}
as $N\to\infty$. Thus, $\widetilde{H}_{\fai,\,L,\,s}(\cx)\cap L^2(\cx)$ is dense in
$\widetilde{H}_{\fai,\,L,\,s}(\cx)$, which completes the proof of Theorem \ref{t4.2}.
\end{proof}

As a corollary of Theorem \ref{t4.2}, we have the following conclusion. We omit the details.

\begin{corollary}\label{cs4.1}
Let $L$ satisfy Assumptions ${\rm (A)}$ and ${\rm (B)}$, and $\fai$
be as in Definition \ref{d2.2} with $\fai\in \rh_{(q_L/I(\fai))'}(\cx)$, where $q_L$
and $I(\fai)$ are respectively as in Assumption {\rm(B)} and \eqref{2.4}.
Then, for all $s\in(p_L,q_L)$, the space $L^s(\cx)\cap H_{\fai,\,L}(\cx)$
is dense in $H_{\fai,\,L}(\cx)$.
\end{corollary}

\section{The atomic characterization of $H_{\fai,\,L}(\cx)$}\label{cs4}

\hskip\parindent In this section, we establish  the atomic characterization
of the Musielak-Orlicz-Hardy space $H_{\fai,\,L}(\cx)$. To obtain
the support condition of $H_{\fai,\,L}(\cx)$ atoms by using
the finite propagation speed for the wave equation, we have to restrict
to a special case of operators satisfying Assumptions ${\rm (A)}$ and ${\rm (B)}$.
More precisely, throughout this section, we assume that
the considered operator $L$ satisfies the following assumptions as in \cite{bckyy}:

\medskip

\noindent {\bf Assumption $\mathrm{(H_1)}$.} $L$ is a
non-negative and self-adjoint operator in $L^2(\cx)$.

\medskip

\noindent {\bf Assumption $\mathrm{(H_2)}$.} There exists a constant $p_L\in [1,2)$
such that the semigroup $\{e^{-tL}\}_{t>0}$, generated by $L$,
satisfies the reinforced $(p_L, \,p_L',\,1)$ off-diagonal estimates on
balls as in Assumption ${\rm (B)}$.

\smallskip

\begin{remark}\label{equivalence between off-diagonal estimates on
balls and off-diagonal estimates}
(i) It is easy to see that if an operator $L$ satisfying Assumptions ${\rm (H_1)}$ and
${\rm (H_2)}$ is one-to-one, then it falls in the scope of operators satisfying
Assumptions ${\rm (A)}$ and ${\rm (B)}$. For the more general case, by using the
functional calculus via the spectral theorem, all the results obtained in the
above sections still hold true in this situation. Here, the
Hardy space $H_{\fai,\,L}(\cx)$ is defined as in Definition \ref{d4.1}. This
is a little different from the version of Hofmann et al. in \cite{hlmmy},
where the dense subspace $H^2(\cx)$ of the Hardy space
is defined to be the completion of  the \emph{range} of $L$ in $L^2(\cx)$,
$\overline {\mathcal{R}(L)}$ (see \cite{hlmmy} for more details). Recall
that $L^2(\cx)=\mathcal{N}(L)\bigoplus \overline {\mathcal{R}(L)}$, where
$\mathcal{N}(L)$ denotes the \emph{kernel} of $L$. We know that these Hardy
spaces are different from a kernel space $\mathcal{N}(L)$, which is not essential
for our purpose. We make this change in the definition of the Hardy space,
because it brings us some conveniences; for example, when $p=2$, we
obtain $H_L^p(\cx)=L^2(\cx)$.

(ii) The following definition of the $L^q$ off-diagonal estimates is from \cite{au07}.
For all $q\in(1,\,\fz)$, a family $\{T_t\}_{t>0}$ of operators is said to
satisfy the $L^q$ \emph{off-diagonal estimates}, if there exist two
positive constants $C$ and $c$ such that
\begin{equation*}
\|e^{-tL}f\|_{L^q(F)}\le C  e^{-\frac{[d(E,F)]^2}
{ct}} \|f\|_{L^q(E)}
\end{equation*}
holds true for every closed sets $E,\,F\subset \cx$, $t\in(0,\fz)$ and $f\in L^q(E)$.
From \cite{am07ii}, we deduce that $\{T_t\}_{t>0}\in \mathcal{O}_{1}(L^q-L^q)$ if and only
if $\{T_t\}_{t>0}$ satisfies the $L^q$ off-diagonal estimates. Thus, Assumption ${\rm (H_2)}$
implies that $\{T_t\}_{t>0}$ satisfies the $L^q$ off-diagonal estimates.
\end{remark}

To establish the atomic characterization of $H_{\fai,\,L}(\cx)$,
we first introduce the notion of the following atoms.

\begin{definition}\label{cd4.1}
Let $\fai$ be as in Definition \ref{d2.2}, $L$ satisfy Assumptions $\mathrm{(H_1)}$
and $\mathrm{(H_2)}$, and $p_L$ be as in Assumption $\mathrm{(H_2)}$.
Assume that $q\in(p_L,p_L')$, $M\in\nn$ and $B\subset\cx$ is a ball.
A function $a\in L^q(\cx)$ is called a \emph{$(\fai,\,q,\,M)_L$-atom} associated
with $B$, if there exists a function $b\in\mathcal{D}(L^M)$ such that

(i) $a=L^M b$;

(ii) for all $k\in\{0,\,\ldots,\,M\}$, $\supp (L^k b)\subset B$;

(iii) $\|(r_B^2L)^kb\|_{L^q(\cx)}\le r_B^{2M}[\mu(B)]^{1/q}\|\chi_B\|^{-1}_{L^\fai(\cx)}$,
where $r_B$ is the radius of $B$ and $k\in\{0,\,\ldots,\,M\}$.

Moreover, if $a$ is a $(\fai,\,q,\,M)_L$-atom
for all $q\in(p_L,p_L')$, then $a$ is called a \emph{$(\fai,\,M)_L$-atom}.
\end{definition}

Based on this kind of atoms, we introduce the following atomic
Musielak-Orlicz-Hardy space.

\begin{definition}\label{cd4.2}
Let $\fai$ be as in Definition \ref{d2.2}, $L$ satisfy Assumptions $\mathrm{(H_1)}$
and $\mathrm{(H_2)}$, and $p_L$ be as in Assumption $\mathrm{(H_2)}$.
Assume that $q\in(p_L,\,p_L')$ and $M\in\nn$. For $f\in L^2(\cx)$,
$f=\sum_j\lz_ja_j$ is called an \emph{atomic $(\fai,\,q,\,M)_L$-representation} of $f$,
if, for all $j$, $a_j$ is a $(\fai,\,q,\,M)_L$-atom associated with the
ball $B_j\subset\cx$, the summation converges in $L^2(\cx)$ and $\{\lz_j\}_j\subset\cc$
satisfies that
$$\sum_j\fai\lf(B_j,\,\frac{|\lz_j|}{\|\chi_{B_j}\|_{L^\fai(\cx)}}\r)<\fz.$$
Let
\begin{eqnarray*}
\wz{H}^{M,\,q}_{\fai,\,L,\,\rm{at}}(\cx):=&&\lf\{f:\ f \ \text{has an atomic} \
(\fai,\,q,\,M)_L\text{-representation}\r\}
\end{eqnarray*}
with the \emph{quasi-norm} given by
\begin{eqnarray*}
&&\|f\|_{H^{M,\,q}_{\fai,\,L,\,\rm{at}}(\cx)}\\
&&\hs:=\inf\lf\{\blz(\{\lz_j a_j\}_j):\
f=\sum_j\lz_j a_j \ \text{is an atomic}
\ (\fai,\,q,\,M)_L\text{-representation}\r\},
\end{eqnarray*}
where the infimum is taken over all the atomic
$(\fai,\,q,\,M)_L$-representations of $f$ and
\begin{eqnarray*}
\Lambda\lf(\lf\{\lz_j a_j\r\}_{j}\r):=\dinf \lf\{\lz\in(0,\,\fz):\
\dsum_{j}\fai\lf(B_j,\,\frac{|\lz_j|}{\lz \|\chi_{B_j}\|_{L^\fai(\cx)}}\r)\le 1\r\}.
\end{eqnarray*}

The \emph{atomic Musielak-Orlicz-Hardy space} $H^{M,\,q}_{\fai,\,L,\,\rm{at}}(\cx)$ is then
defined as the completion of $\wz{H}^{M,\,q}_{\fai,\,L,\,\rm{at}}(\cx)$ with respect to the
\emph{quasi-norm} $\|\cdot\|_{H^{M,\,q}_{\fai,\,L,\,\rm{at}}(\cx)}$.
\end{definition}

We have the following atomic characterization of the Musielak-Orlicz-Hardy
space $H_{\fai,\,L}(\cx)$.

\begin{theorem}\label{ct4.3}
Let $\fai$ be as in Definition \ref{d2.2}, $L$ satisfy Assumptions $\mathrm{(H_1)}$
and $\mathrm{(H_2)}$, $p_L$ be as in Assumption $\mathrm{(H_2)}$
and $M\in\nn$ satisfying $M>\frac{n}{2}
(\frac{q(\fai)}{i(\fai)}-\frac{1}{p_L'})$, where
$q(\fai)$ and $i(\fai)$ are respectively as in \eqref{2.7}
and \eqref{2.5}. Assume further that $q\in([r(\fai)]' I(\fai),p_L')\cap(p_L,p_L')$,
where $r(\fai)$ and $I(\fai)$ are, respectively,  as in \eqref{2.8} and
\eqref{2.4}. Then, $H_{\fai,\,L}(\cx)$ and $H^{M,\,q}_{\fai,\,L,\,\rm{at}}(\cx)$ coincide
with equivalent quasi-norms.
\end{theorem}

\begin{remark}\label{cr4.1}
When $\cx:=\rn$ and for all $x\in\rn$ and $t\in[0,\fz)$,
$\fai(x,t):=t^pw(x)$ with $p\in(0,1]$ and $w$ a Muckenhoupt weight, Theorem \ref{ct4.3}
is just \cite[Theorem 3.8]{bckyy}.
\end{remark}

To prove Theorem \ref{ct4.3}, we need to introduce some operator $\pi_{\Phi,\,L,\,k}$, which
can be viewed as a retraction operator from the Musielak-Orlicz-tent space
$T_{\fai}(\cx_+)$, introduced in Section \ref{s3},
to $H_{\fai,\,L}(\cx)$.
To this end, we first give some notations. In what follows, for any operator $T$, we
let $K_T$ be its \emph{integral kernel}. Let ${ cos}(t\sqrt{L})$ with $t\in(0,\,\fz)$
be the \emph{cosine function operator} generated by $L$. By
\cite[Theorem 3.4]{cs08plms} (see also \cite[Proposition 3.4]{hlmmy}),
we know that there exists a positive constant $C_0$ such that
\begin{eqnarray}\label{c4.3}
\supp K_{{ cos}(t\sqrt{L})}\subset \{(x,\,y)\in \cx\times\cx:\ d(x,\,y)\le C_0 t \}.
\end{eqnarray}
Moreover, let $\psi\in \rm{C}_{\rm{c}}^\fz(\mathbb{R})$ be even and $\supp\psi\subset
(-C_0^{-1}, C_0^{-1})$, where $C_0$ is as in \eqref{c4.3}. Let $\Phi$
denote the \emph{Fourier transform} of $\psi$. Then, for all $k\in \mathbb{N}$
and $t\in(0,\fz)$, the kernel of $(t^2L)^k\Phi(t\sqrt{L})$ satisfies that
\begin{eqnarray}\label{c4.4}
\supp K_{(t^2L)^k\Phi(t\sqrt{L})}\subset \{(x,y)\in \cx\times \cx:\
d(x,\,y)\leq t\}.
\end{eqnarray}

Now, let $M\in\nn$ with
$M>\frac{n}{2}(\frac{q(\fai)}{i(\fai)}-\frac{1}{p_L'})$, where
$q(\fai)$ and $i(\fai)$ are respectively as in \eqref{2.7}
and \eqref{2.5}.
Assume that $\Phi$ is as in \eqref{c4.4}.  Then, for all
 $k\in\nn$,  $f\in L^2_{b}(\cx_+)$ and $x\in\cx$, the \emph{operator}
 $\pi_{\Phi,\,L,\,k}$ is defined by
\begin{equation*}
\pi_{\Phi,\,L,\,k}(f)(x):=C_{(\Phi,\,k)}\int_0^{\fz}\lf(t^2L\r)^{k+1}
\Phi(t\sqrt{L})(f(\cdot,\,t))(x)\,\frac{dt}{t},
\end{equation*}
where $C_{(\Phi,\,k)}$ is a positive constant such that
\begin{equation}\label{c4.6}
C_{(\Phi,\,k)}\int_0^\fz  t^{2(k+1)}\Phi(t)t^2e^{-t^2} \,\frac{dt}{t}=1.
\end{equation}

Using Minkowski's integral inequality and the quadratic estimates (see also
\cite[(3.14)]{hlmmy}), we
easily see that $\pi_{\Phi,\,L,\,k}$ can be  continuously extended from
$T^2(\cx_+)$ to $L^2(\cx)$.  Moreover, we have the
following boundedness of $\pi_{\Phi,\,L,\,M}$, which can be
viewed as an extension of \cite[Proposition 4.6]{yys4}.

\begin{proposition}\label{cp4.4}
Let $\fai$ be as in Definition \ref{d2.2}, $L$ satisfy Assumptions $\mathrm{(H_1)}$
and $\mathrm{(H_2)}$, $p_L$ be as in Assumption $\mathrm{(H_2)}$, $q\in(p_L,p_L')$ and
$M\in\nn$ satisfying $M>\frac{n}{2}(\frac{q(\fai)}{i(\fai)}-\frac{1}{p_L'})$, where
$q(\fai)$ and $i(\fai)$ are respectively as in \eqref{2.7}
and \eqref{2.5}.  Assume further that $\fai\in\rh_{(p_L'/I(\fai))'}(\cx)$, where $I(\fai)$ is
as in \eqref{2.4}. Then the operator $\pi_{\Phi,\,L,\,M}$, initially defined on the space
$T^{ {b}}_{\fai}(\cx_+)$, extends to a bounded linear operator
from $T_{\fai}(\cx_+)$ to $H_{\fai,\,L}(\cx)$.
\end{proposition}

\begin{proof}
Without loss of generality, we may only prove Proposition \ref{cp4.4} under the
assumption that $q\in[2,\,p_L')$. For the case when $q\in(p_L,\,2)$,
the following proof is still valid, only need to make a few
modifications when using H\"older's
inequality. Let $f\in T^b_{\fai}(\cx_+)$.
From Proposition \ref{p3.1}, Theorem \ref{t3.1},
Corollary \ref{c3.1} and the fact that
$\pi_{\Phi,\,L,\,M}$ is bounded from
$T_2^2(\cx_+)$ to $L^2(\cx)$, we deduce that there
exist a family $\{A_j\}_j$ of $(T_\fai,\,\fz)$-atoms associated respectively
to the balls $\{B_j\}_j$ and $\{\lz_j\}_j\subset \cc$ such that
\begin{eqnarray*}
\pi_{\Phi,\,L,\,M}(f)=\sum_j\lz_j\pi_{\Phi,\,L,\,M}(A_j)=:\sum_j\lz_j a_j
\end{eqnarray*}
in $L^2(\cx)$ and
\begin{eqnarray}\label{c4.7}
\Lambda(\{\lz_j A_j\}_j)\ls \|f\|_{T_{\fai}(\cx_+)},
\end{eqnarray}
where $\Lambda(\{\lz_jA_j\}_j)$ is as in \eqref{3.2}.
Moreover, since the square function
$S_L$ is nonnegative (which means that, for all
$f\in\mathcal{D}(S_L)$ and $x\in\cx$,
$S_L(f)(x)\ge0$), sublinear and $S_L$ is bounded on
$L^2(\cx)$, we know that,
for almost every $x\in\cx$, $S_L(\pi_{\Phi,\,L,\,M}(f))(x)\le
\sum_{j}\lz_j S_L(a_j)(x)$.
This, combined with Lemma \ref{l2.1}(i), implies that, for all $\lz\in(0,\fz)$,
\begin{equation}\label{c4.8}
\int_{\cx}\fai\lf(x,\,\frac{S_L(\pi_{\Phi,\,L,\,M}(f))(x)}{\lz}\r)\,d\mu(x)\ls
\sum_j\int_{\cx}\fai\lf(x,\,|\lz_j|S_L(\lz_j a_j(x)/\lz)\r)\,d\mu(x).
\end{equation}
We first prove that, for each $j$, $a_j$ is a $(\fai,\,M)_L$-atom
associated with $B_j$. Indeed, let
\begin{eqnarray}\label{c4.9}
b_j:=C_{(\Phi,\,M)}\int_0^{\fz} t^{2(M+1)} L\Phi(t\sqrt{L})(A_j(\cdot,\,t))
\,\frac{dt}{t},
\end{eqnarray}
where $C_{(\Phi,\,M)}$ is as in \eqref{c4.6}. From \eqref{c4.4}, we infer that, for
all $k\in\{0,\,\ldots,\,M\}$, $\supp L^k b_j\subset B_j$,
which is the support condition of a $(\fai,\,q,\,M)_L$-atom
as in Definition \ref{cd4.1}.

On the other hand, for any $h\in L^{q'}(B_j)\cap L^2(B_j)$, By \eqref{c4.9},
Assumption ${\rm (H_1)}$, Fubini's theorem, the fact that $\supp
A_j\subset \widehat{B}_j$ and H\"older's inequality, we conclude that, for all
$k\in\{0,\,\ldots,\,M\}$,
\begin{eqnarray*}
&&\lf|\dint_{\cx} \lf(r_{B_j}^2L\r)^k b_j(x)h(x)\,d\mu(x)\r|\\
&& \hs\sim
\lf|\int_0^{\fz} \dint_{\cx}  A_j(x,\,t) \lf(r_{B_j}^2L\r)^k
t^{2(M+1)}L \Phi(t\sqrt{L})h(x)\,\frac{d\mu(x)\,dt}{t}\r|\\
&& \hs\ls r_{B_j}^{2M} \int_0^{\fz} \int_{\cx}
\lf|A_j(x,\,t) \lf(t^2L\r)^{k+1} \Phi(t\sqrt{L})
h(x)\r|\,\frac{d\mu(x)\,dt}{t}\\
&& \hs\ls r_{B_j}^{2M} \lf\|\mathcal{A}(A_j)\r\|_{L^q(\cx)}
\lf\|\lf[\int_{\Gamma(x)}  \lf| \lf(t^2L\r)^{k+1} \Phi(t\sqrt{L}) h(x)\r|^2
\,\frac{d\mu(x)\,dt}{V(x,t)t}\r]^{\frac{1}{2}}\r\|_{L^{q'}(\cx)}.
\end{eqnarray*}

Following the same argument as that used in the proof of \cite[Lemma 5.3]{bckyy},
we easily see that, for all $q'\in (p_L,\,p_L')$,
$\pi_{\Phi,\,L,\,k}$ is bounded on $L^{q'}(\cx)$. This, together with
the arbitrariness of $h$ and
the fact that $A_j$ is a $(T_\fai,\,\fz)$-atom associated with $B_j$, implies that
\begin{eqnarray*}
\lf\|\lf(r_{B_j}^2L\r)^k b_j\r\|_{L^q(\cx)}\ls r_{B_j}^{2M}[\mu(B)]^{1/q}
\|\chi_{B_j}\|_{L^\fai(\cx)}^{-1},
\end{eqnarray*}
which is the size condition of a $(\fai,\,q,\,M)_L$-atom as in Definition \ref{cd4.1}(iii).
Thus, we conclude that, for each $j$, $a_j$ is a $(\fai,\,M)_L$-atom
associated with $B_j$.

We claim that, to finish the proof of Proposition \ref{cp4.4},
 it suffices to show that, for any $\lz\in\cc$ and $(\fai,\,M)_L$-atom
 $a$ associated with the ball $B\subset\cx$,
\begin{eqnarray}\label{c4.10}
\dint_{\cx}\fai \lf(x,\,S_L(\lz a)(x)\r)\,d\mu(x)\ls \fai\lf(B,\,\frac{|\lz|}
{\|\chi_B\|_{L^\fai(\cx)}}\r).
\end{eqnarray}
Indeed, if \eqref{c4.10} holds, then by \eqref{c4.8},
we see immediately that, for all $f\in T_\fai^b (\cx_+ )$ and $\lz\in
(0,\,\fz)$,
\begin{eqnarray*}
\int_{\cx}\fai \lf(x,\,\frac{S_L(\pi_{\Phi,\,L,\,M}(f))(x)}{\lz}\r)\,d\mu(x)
\ls\sum_{j}\fai\lf(B_j,\,\frac{|\lz_j|}{\lz\|\chi_{B_j}\|_{L^\fai(\cx)}}\r),
\end{eqnarray*}
which, together with \eqref{c4.7}, implies that $\|\pi_{\Phi,\,L,\,M}f\|\ls
\Lambda\lf(\{\lz_j A_j\}_j\r)\ls\|f\|_{T_\fai(\cx_+)}$. Thus,
$\pi_{\Phi,\,L,\,M}$ can be extended to a bounded operator from
$T_\fai(\cx_+)$ to $H_{\fai,\,L}(\cx)$. This proves the claim.
By $M>\frac{n}{2}(\frac{q(\fai)}{i(\fai)}-\frac{1}{p_L'})$,
$\fai\in\rh_{(p_L'/I(\fai))'}(\cx)$ and Lemma \ref{l2.4}(iv),
we know that there exist $q_0\in(q(\fai),\fz)$, $p_2\in(0,i(\fai))$,
$p_1\in [I(\fai),\,1]$ and $q\in (I(\fai)[r(\fai)]',\,p_L')\cap (p_L,\,p_L')$
such that $\fai$ is of uniformly upper type $p_1$ and lower type $p_2$, $\fai\in
\aa_{q_0}(\cx)$, $M>\frac{n}{2}(\frac{q_0}{p_2}-\frac{1}{q})$ and
\begin{eqnarray}\label{c4.11}
\fai\in \rh_{(q/p_1)'}(\cx).
\end{eqnarray}
Now, for $j\in\{0,\,\ldots,\,4\}$, similar to the proof of \eqref{4.13}, we conclude that
\begin{eqnarray*}
&&\int_{S_j(B)}\fai\lf(x,\,S_L(\lz a)(x)\r)\,d\mu(x)\ls\fai
\lf(B,\,\frac{|\lz|} {\|\chi_B\|_{L^\fai(\cx)}}\r).
\end{eqnarray*}

Now, we turn to the case when $j\in\nn$ and $j\ge 5$.
From the fact that $\fai$ is of uniformly upper type $p_1$ and
lower type $p_2$, we deduce that
\begin{eqnarray*}
&&\dint_{S_j(B)}\fai\lf(x,\,S_L(\lz a)(x)\r)\,d\mu(x)\\
&&\nonumber\hs \ls \dint_{S_j(B)} \lf[S_L(a)(x)\|\chi_{B}\|_{L^\fai(\cx)}\r]^{p_1}
\fai\lf(x,\, \frac{|\lz|}{\|\chi_B\|_{L^\fai(\cx)}}\r)\,d\mu(x)\\
&&\nonumber\hs\hs+
\dint_{S_j(B)} \lf[S_L(a)(x)\|\chi_{B}\|_{L^\fai(\cx)}\r]^{p_2}
\fai\lf(x,\, \frac{|\lz|}{\|\chi_B\|_{L^\fai(\cx)}}\r)\,d\mu(x)=:
\mathrm{I}_j+\mathrm{J}_j.
\end{eqnarray*}
To estimate $\mathrm{I}_j$, let $\wz{\mathrm{I}}_j:=
\|S_L(a)\|_{L^q(S_j(B))}^{q}$. By H\"older's inequality
and \eqref{c4.11}, we find that
\begin{eqnarray}\label{c4.15}
\hs\hs\mathrm{I}_j&&\ls \|\chi_{B}\|_{L^\fai(\cx)}^{p_1}
\lf\|S_L(a)\r\|_{L^q(S_j(B))}^{p_1}
\lf\{\int_{S_j(B)} \lf[\fai\lf(x,\, \frac{|\lz|}
{\|\chi_B\|_{L^\fai(\cx)}}\r)\r]^{(\frac{q}{p_1})'}\,d\mu(x)\r\}^{\frac{1}{(\frac{q}{p_1})'}}
\\&& \nonumber\ls \lf[\mu(2^jB)\r]^{-\frac{p_1}{q}}
\|\chi_{B}\|_{L^\fai(\cx)}^{p_1}\wz{\mathrm{I}}_j^{\frac{p_1}{q}}
\fai\lf(S_j(B),\,\frac{|\lz|}{\|\chi_B\|_{L^\fai(\cx)}}\r).
\end{eqnarray}
To estimate $\wz{\mathrm{I}}_j$, we write $\wz{\mathrm{I}}_j$ into
\begin{eqnarray}\label{c4.16}
\wz{\mathrm{I}}_j\ls &&\int_{S_j(B)}\lf[\int_0^{\frac{d(x,\,x_B)}{4}}
\dint_{B(x,t)}\lf|t^2Le^{-t^2L}(a)(y)\r|^2 \, \frac{d\mu(y)\,dt}{t
\,V(x,t)}\r]^{\frac{q}{2}}\,d\mu(x)\\
&& \nonumber\hs+\int_{S_j(B)}\lf[ \int_{\frac{d(x,\,x_B)}{4}}^\fz
\dint_{B(x,t)} \cdots \,\frac{d\mu(y)\,dt}{t
\,V(x,t)}\r]^{\frac{q}{2}}\,d\mu(x)=:\mathcal{A}_j+\mathcal{B}_j.
\end{eqnarray}
We first estimate $\mathcal{A}_j$. For $j\ge5$, let
\begin{eqnarray*}
G_j(B):=\lf\{y\in\cx:\ \text{there exists}\ x\in S_j(B) \ \text{such that}\
d(y,x)<\frac{1}{4}d(x,x_B)\r\},
\end{eqnarray*}
where $x_B$ denotes the center of $B$. Moreover,
by the triangle inequality, we easily see that, for all $y\in G_j(B)$, $d(y,x_B)\le
2^{k+1}r_B$ and $d(y,x_B)\ge 2^{k-2}r_B$. Thus, $G_j(B)\subset \cup_{i=j-1}^{i+1}
S_i(B)=:\wz{\wz{S}}_j(B)$. This, combining with H\"older's inequality,
Fubini's theorem, the definition of the function $b$ as in \eqref{c4.9}, Assumption
$\mathrm{(H_2)}$ and the fact that $a$ is a $(\fai,\,M)_L$-atom, implies that
\begin{eqnarray}\label{c4.17}
\hs\hs\mathcal{A}_j&&\ls(2^jr_B)^{\frac{q}{2}-1}
\int_{S_j(B)} \lf[\int_0^{\frac{d(x,\,x_B)}{4}}
\int_{B(x,t)}\lf|t^2Le^{-t^2L}(a)(y)\r|^q \, \frac{d\mu(y)\,dt}{t^{\frac{q}{2}}
\,V(x,t)}\r]\,d\mu(x)\\
&& \nonumber\ls (2^jr_B)^{\frac{q}{2}-1}\int_0^{2^{j-2}r_B}
\int_{\wz{\wz S}_j(B)}  \lf|\lf(t^2L\r)^{M+1} e^{-t^2L}
(b)(y)\r|^q \,\frac{d\mu(y)\,dt}{t^{q(\frac{1}{2}+2M)}}
\\ && \nonumber\ls (2^jr_B)^{\frac{q}{2}-1} \|b\|_{L^q(B)}^q
\dint_0^{2^{j-2}r_B} \exp\lf\{ -C\frac{[2^jr_B]^2}{t^2}\r\} \,
\frac{dt}{t^{q(\frac{1}{2}+2M)}}
\\ && \nonumber\ls 2^{-2jqM} \mu(B) \|\chi_{B}\|_{L^\fai(\cx)}^{-q}.
\end{eqnarray}
The estimate of $\mathcal{B}_j$ is similar to that of $\mathcal{A}_j$ and, via
 replacing Assumption $\mathrm{(H_2)}$ by the $L^q(\cx)$-boundedness
of the family of operators $\{(t^2L)^Me^{-t^2L}\}_{t>0}$, we conclude that
\begin{eqnarray*}
\hs\hs\mathcal{B}_j&&\ls (2^jr_B)^{-4M(\frac{q}{2}-1)}
\int_{S_j(B)} \lf[\int_{\frac{d(x,x_B)}{4}}^\fz
\int_{B(x,t)} \lf|t^2Le^{-t^2L}(b)(y)\r|^q \,\frac{d\mu(y)\,dt}{t^{4M+1}
\,V(x,t)}\r]\,d\mu(x)\\
&& \nonumber\ls (2^jr_B)^{-4M(\frac{q}{2}-1)}
\|b\|_{L^{q}(\cx)}^{q}\int_{2^{j-3}r_B}^\fz
\,\frac{dt}{t^{4M+1}}\ls 2^{-2jqM} \mu(B)\|\chi_B\|_{L^\fai(\cx)}^{-q},
\end{eqnarray*}
which, together with \eqref{c4.16} and \eqref{c4.17}, shows immediately that
\begin{eqnarray*}
\wz{\mathrm{I}}_j\ls 2^{-2jqM} \mu(B)\|\chi_B\|_{L^\fai(\cx)}^{-q}.
\end{eqnarray*}
Thus, from this, \eqref{c4.15} and Lemma \ref{l2.4}(vii), we deduce that
\begin{eqnarray}\label{c4.18}
\hs\hs\mathrm{I}_j&&\ls 2^{-2jp_1M} \lf[\mu(2^jB)\r]^{-\frac{p_1}{q}}
 \mu(B) \fai\lf(S_j(B),\,\frac{|\lz|} {\|\chi_B\|_{L^\fai(\cx)}}\r)\\
&&\nonumber \ls 2^{-j[2p_1M-n(q_0-\frac{p_1}{q})]}
\fai\lf(B,\, \frac{|\lz|} {\|\chi_{B}\|_{L^\fai(B)}}\r)\sim
 2^{-j\epsilon_0} \fai\lf(B,\, \frac{|\lz|} {\|\chi_{B}\|_{L^\fai(B)}}\r),
\end{eqnarray}
where $\epsilon_0:=2p_1M-n(q_0-{p_1}/{q})$.
The estimate of $\mathrm{J}_j$ is similar to that of $\mathrm{I}_j$.
We only need to point out that, from Lemma \ref{l2.4}(ii) and the fact that
$(\frac{q}{p_2})'<(\frac{q}{p_1})'$, it follows that $\fai\in \rh_{(\frac{q}{p_2})'}(\cx)$.
Thus, we conclude that
\begin{eqnarray}\label{c4.19}
\mathrm{J}_j\ls 2^{-j[2p_2M-n(q_0-\frac{p_2}{q})]} \fai\lf(B,\, \frac{|\lz|}
{\|\chi_{B}\|_{L^\fai(\cx)}}\r)\sim2^{-j\wz{\epsilon}_0} \fai\lf(B,\, \frac{|\lz|}
{\|\chi_{B}\|_{L^\fai(\cx)}}\r),
\end{eqnarray}
where $\wz{\epsilon}_0:=2p_2M-n(q_0-{p_2}/{q})$.
Let $\wz{\wz\epsilon}_0:=\min\{\epsilon_0,\,\wz{\epsilon}_0\}>0$.
Combining \eqref{c4.18} and \eqref{c4.19}, we immediately conclude that
\begin{eqnarray}\label{c4.20}
\int_{\cx}\fai\lf(x,\,S_L(\lz a)(x)\r)\,d\mu(x)&&\ls \sum_{j\in\nn}
2^{-j\wz{\wz{\epsilon}}_0} \fai\lf(B,\, \frac{|\lz|}
{\|\chi_{B}\|_{L^\fai(\cx)}}\r)\\
&& \nonumber\ls\fai\lf(B,\,\frac{|\lz|}{\|\chi_{B}\|_{L^\fai(\cx)}}\r),
\end{eqnarray}
which completes the proof of \eqref{c4.10} and hence Proposition
\ref{cp4.4}.
\end{proof}

Before turning to the proof of Theorem \ref{ct4.3}, we introduce
a sufficient condition which guarantees a given operator to be bounded on
the atomic Musielak-Orlicz-Hardy space.

\begin{lemma}\label{cl4.5}
Let $\fai$ be as in Definition \ref{d2.2}, $L$ satisfy Assumptions $\mathrm{(H_1)}$
and $\mathrm{(H_2)}$, $p_L$ be as in Assumption $\mathrm{(H_2)}$, $q\in(p_L,p_L')$ and
$M\in\nn$ satisfying $M>\frac{n}{2}(\frac{q(\fai)}{i(\fai)}-\frac{1}{p_L'})$, where
$q(\fai)$ and $i(\fai)$ are respectively as in \eqref{2.7}
and \eqref{2.5}. Suppose that $T$ is a linear (resp. nonnegative
sublinear) operator which maps $L^2(\cx)$ continuously into weak-$L^2(\cx)$.
If there exists a positive constant $C$ such that, for any $\lz\in\cc$ and $(\fai,\,
q,\,M)_L$-atom $a$ associated with the ball $B$,
\begin{eqnarray}\label{c4.21}
\int_{\cx}\fai\lf(x,\, T(\lz a)(x)\r)\,d\mu(x)\le C \fai\lf(B,\,\frac{|\lz|}
{\|\chi_B\|_{L^\fai(\cx)}}\r),
\end{eqnarray}
then $T$ can extend to be a bounded linear (resp. sublinear) operator from
$H^{M,\,q}_{\fai,\,L,\,\rm{at}}(\cx)$ to $L^\fai(\cx)$.
\end{lemma}

The proof of Lemma \ref{cl4.5} is similar to that of \cite[Lemma 5.6]{yys4}.
See also the proof of \cite[Lemma 5.1]{jy10} and \cite[Lemma 4.1]{bckyy}.
We omit the details here.

Now we prove Theorem \ref{ct4.3} by using Proposition \ref{cp4.4} and Lemma \ref{cl4.5}.

\begin{proof}[Proof of Theorem \ref{ct4.3}]
By Definition \ref{cd4.2}, we see that, to show Theorem \ref{ct4.3},
it suffices to prove that
\begin{eqnarray}\label{c4.22}
L^2(\cx)\cap
H_{\fai,\,L}(\cx)=L^2(\cx)\cap H^{M,\,q}_{\fai,\,L,\,\rm{at}}(\cx)
\end{eqnarray}
with equivalent norms. We divide the proof of \eqref{c4.22} into
the following two steps.

{\bf Step 1.} We first prove the inclusion $L^2(\cx)\cap
H_{\fai,\,L}(\cx)\subset L^2(\cx)\cap H^{M,\,q}_{\fai,\,L,\,\rm{at}}(\cx)$.
For any $f\in L^2(\cx)\cap H_{\fai,\,L}(\cx)$, by the bounded functional
calculus in $L^2(\cx)$, we know that there exists a positive constant
$C_{(\Phi,\,M)}$ such that
\begin{eqnarray*}
f=C_{(\Phi,\,M)} \int_0^\fz \lf(t^2L\r)^{M+1}
\Phi(t\sqrt{L})t^2Le^{-t^2L}f\,\frac{dt}{t}=\pi_{\Phi,\,L,\,M}
\lf(t^2L e^{-t^2L} f\r)
\end{eqnarray*}
in $L^2(\cx)$.  Moreover, from the fact that $t^2L e^{-t^2L}f\in
T_\fai(\cx_+)$, we deduce that there exist
$\{\lz_j\}_j\subset\cc$ and $\{A_j\}_j$
of $(T_\fai,\,\fz)$-atoms, respectively, associated with $\{B_j\}_j$ such that
\begin{eqnarray*}
t^2L e^{-t^2L}f=\sum_j \lz_j A_j
\end{eqnarray*}
in $T_\fai(\cx_+)\cap T_2^2(\cx_+)$ and
$\Lambda(\{\lz_jA_j\}_j)\ls\|t^2Le^{-t^2L}f\|_{T_\fai(\cx_+)}$,
which, together with Proposition \ref{cp4.4}, implies that
\begin{eqnarray*}
f=\pi_{\Phi,\,L,\,M}\lf(t^2L e^{-t^2L}f\r)=\dsum_{j}\lz_j \pi_{\Phi,\,L,\,M}(A_j)
\end{eqnarray*}
in $L^2(\cx)$. This, together with the fact that $\pi_{\Phi,\,L,\,M}(A_j)$
is a $(\fai,\,q,\,M)_L$-atom associated with $B_j$, immediately shows that
$f\in L^2(\cx)\cap H^{M,\,q}_{\fai,\,L,\,\rm{at}}(\cx)$. Thus, $L^2(\cx)\cap
H_{\fai,\,L}(\cx)\subset L^2(\cx)\cap H^{M,\,q}_{\fai,\,L,\,\rm{at}}(\cx)$.

{\bf Step 2.} We now prove the inclusion $ L^2(\cx)
\cap H^{M,\,q}_{\fai,\,L,\,\rm{at}}(\cx) \subset L^2(\cx)\cap H_{\fai,\,L}(\cx)$.
By \eqref{c4.20}, we know that, for any $\lz\in\cc$ and $(\fai,\,q,\,M)_L$-atom
$a$ associated with the ball $B$, \eqref{c4.21} holds, with $S_L$ in place of
$T$. Thus, by Lemma \ref{cl4.5}, we conclude that $S_L$ is bounded from
$H^{M,\,q}_{\fai,\,L,\,\rm{at}}(\cx)$ to $L^\fai(\cx)$, which immediately
implies that, for all $f\in L^2(\cx)\cap H^{M,\,q}_{\fai,\,L,\,\rm{at}}(\cx)$,
$\|f\|_{H_{\fai,\,L}(\cx)}\ls \|f\|_{H^{M,\,q}_{\fai,\,L,\,\rm{at}}(\cx)}$. This
shows that $ L^2(\cx)
\cap H^{M,\,q}_{\fai,\,L,\,\rm{at}}(\cx) \subset L^2(\cx)\cap H_{\fai,\,L}(\cx)$,
which, together with Step 1, completes the proof of \eqref{c4.22} and hence
Theorem \ref{ct4.3}.
\end{proof}

Now, we consider the question of replacing the role of $L^2(\cx)$ norm
by the more general $L^s(\cx)$ norm for $s\in(p_L,\,p_L')$, in the definition of the atomic
Musielak-Orlicz-Hardy space $H^{M,\,q}_{\fai,\,L,\,\rm{at}}(\cx)$. We also introduce
the following notion of the $L^s(\cx)$-adapted atomic Musielak-Orlicz-Hardy space
$\wz{H}^{M,\,q,\,s}_{\fai,\,L,\,\rm{at}}(\cx)$.

\begin{definition}\label{cd4.8}
Let $\fai$ be as in Definition \ref{d2.2}, $L$ satisfy Assumptions $\mathrm{(H_1)}$
and $\mathrm{(H_2)}$, and $p_L$ be as in Assumption $\mathrm{(H_2)}$.
Assume that $q,\,s\in(p_L,\,p_L')$ and $M\in\nn$. For $f\in L^2(\cx)$,
$f=\sum_j\lz_ja_j$ is called an \emph{atomic $(\fai,\,q,\,s,\,M)_L$-representation} of $f$,
if each $a_j$ is a $(\fai,\,q,\,M)_L$-atom associated with the
ball $B_j\subset\cx$, the summation converges in $L^s(\cx)$ and $\{\lz_j\}_j\subset\cc$
satisfies that
$$\sum_j\fai\lf(B_j,\,\frac{|\lz_j|}{\|\chi_{B_j}\|_{L^\fai(\cx)}}\r)<\fz.$$
Let
\begin{eqnarray*}
\wz{H}^{M,\,q,\,s}_{\fai,\,L,\,\rm{at}}(\cx):=&&\lf\{f:\ f \ \text{has an atomic}
\ (\fai,\,q,\,s,\,M)_L\text{-representation}\r\}
\end{eqnarray*}
with the \emph{quasi-norm} given by
\begin{eqnarray*}
&&\|f\|_{H^{M,\,q}_{\fai,\,L,\,\rm{at}}(\cx)}\\
&&\hs:=\inf\lf\{\blz(\{\lz_j a_j\}_j):\
f=\sum_j\lz_j a_j \ \text{is an atomic}
\ (\fai,\,q,\,s,\,M)_L\text{-representation}\r\},
\end{eqnarray*}
where the infimum is taken over all the atomic
$(\fai,\,q,\,s,\,M)_L$-representations of $f$ and
\begin{eqnarray*}
\Lambda\lf(\lf\{\lz_j a_j\r\}_{j}\r):=\inf \lf\{\lz\in(0,\,\fz):\
\sum_{j}\fai\lf(B_j,\,\frac{|\lz_j|}{\lz \|\chi_{B_j}\|_{L^\fai(\cx)}}\r)\le 1\r\}.
\end{eqnarray*}

The \emph{atomic Musielak-Orlicz-Hardy space} $H^{M,\,q,\,s}_{\fai,\,L,\,\rm{at}}(\cx)$
is then defined as the completion of $\wz{H}^{M,\,q,\,s}_{\fai,\,L,\,\rm{at}}(\cx)$
with respect to the \emph{quasi-norm} $\|\cdot\|_{H^{M,\,q,\,s}_{\fai,\,L,\,\rm{at}}(\cx)}$.
\end{definition}

From its definition, we know that the space $H^{M,\,q}_{\fai,\,L,\,\rm{at}}(\cx)$
as in Definition \ref{cd4.2}
can be viewed as the $L^2(\cx)$-adapted atomic Musielak-Orlicz-Hardy
space $H^{M,\,q,\,2}_{\fai,\,L,\,\rm{at}}(\cx)$.
Moreover, we have the following equivalence between
$H^{M,\,q,\,s}_{\fai,\,L,\,\rm{at}}(\cx)$ and $H^{M,\,q}_{\fai,\,L,\,\rm{at}}(\cx)$.

\begin{theorem}\label{ct4.9}
Let $\fai$ be as in Definition \ref{d2.2}, $L$ satisfy Assumptions $\mathrm{(H_1)}$
and $\mathrm{(H_2)}$, $p_L$ be as in Assumption $\mathrm{(H_2)}$ and $M\in\nn$
satisfying $M>\frac{n}{2}(\frac{q(\fai)}{i(\fai)}-\frac{1}{p_L'})$, where
$q(\fai)$ and $i(\fai)$ are respectively as in \eqref{2.7}
and \eqref{2.5}. Then, for all $s\in (p_L,\,p_L')$ and $q\in(I(\fai)[r(\fai)]',p_L')$, where
$I(\fai)$ and $r(\fai)$ are respectively as in \eqref{2.4} and \eqref{2.8},
$H^{M,\,q,\,s}_{\fai,\,L,\,\rm{at}}(\cx)$ and $H^{M,\,q}_{\fai,\,L,\,\rm{at}}(\cx)$ coincide
with equivalent quasi-norms.
\end{theorem}

To prove Theorem \ref{ct4.9}, we need a few lemmas.
This first one is a variant of Lemma \ref{cl4.5},
whose proof is similar.
We omit the details.

\begin{lemma}\label{cl4.10}
Let $\fai$ be as in Definition \ref{d2.2}, $L$ satisfy Assumptions $\mathrm{(H_1)}$
and $\mathrm{(H_2)}$, $p_L$ be as in Assumption $\mathrm{(H_2)}$, $q,\,s\in(p_L,\,p_L')$ and
$M\in\nn$ satisfying $M>\frac{n}{2}(\frac{q(\fai)}{i(\fai)}-\frac{1}{p_L'})$, where
$q(\fai)$ and $i(\fai)$ are respectively as in \eqref{2.7}
and \eqref{2.5}. Suppose that $T$ is a linear (resp. nonnegative
sublinear) operator which maps $L^s(\cx)$ continuously into weak-$L^s(\cx)$.
If there exists a positive constant $C$ such that, for any $\lz\in\cc$ and $(\fai,\,
q,\,M)_L$-atom $a$ associated with the ball $B$,
\begin{eqnarray*}
\int_{\cx}\fai\lf(x,\, T(\lz a)(x)\r)\,d\mu(x)\le C \fai\lf(B,\,\frac{|\lz|}
{\|\chi_B\|_{L^\fai(\cx)}}\r),
\end{eqnarray*}
then $T$ can extend to be a bounded linear (resp. sublinear) operator from
$H^{M,\,q,\,s}_{\fai,\,L,\,\rm{at}}(\cx)$ to $L^\fai(\cx)$.
\end{lemma}

\begin{lemma}\label{cl4.11}
Let $\fai$ be as in Definition \ref{d2.2}, $L$ satisfy Assumptions $\mathrm{(H_1)}$
and $\mathrm{(H_2)}$, $p_L$ be as in Assumption $\mathrm{(H_2)}$, $q,\,s\in(p_L,p_L')$ and
$M\in\nn$ satisfying $M>\frac{n}{2}(\frac{q(\fai)}{i(\fai)}-\frac{1}{p_L'})$, where
$q(\fai)$ and $i(\fai)$ are respectively as in \eqref{2.7}
and \eqref{2.5}.  Assume further that $\fai\in\rh_{(p_L'/I(\fai))'}(\cx)$,
where $I(\fai)$ is  as in \eqref{2.4}. Then, for all $k\in\nn$,

{\rm(i)} the operator $\pi_{\Phi,\,L,\,k}$, initially defined on the space
$T^{s,\,b}_2(\cx_+)$, extends to a bounded linear operator
from $T^{s}_2(\cx_+)$ to $L^s(\cx)$;

{\rm(ii)} for all $t\in(0,\,\fz)$, the operator $t^2Le^{-t^2L}$, initially defined on
$L^2(\cx)$, extends to a bounded linear operator
from $L^{s}(\cx)$ to $T^{s}_2(\cx_+)$.
\end{lemma}

\begin{proof}
We first prove (i). Let $f\in T_2^s(\cx_+)\cap
T^{2,\,b}_2(\cx_+)$. For any $g\in L^{s'}(\cx)\cap L^2(\cx)$, by
Fubini's theorem, Assumption $\mathrm{(H_1)}$,
H\"older's inequality and the $L^{s'}(\cx)$-boundedness of
the square function $S_{\Phi,\,L,\,k}$, we conclude that
\begin{eqnarray*}
\lf|\int_{\cx} \pi_{\Phi,\,L,\,k} (f)(x)g(x)\,d\mu(x)\r|&&=
\lf|\int_0^\fz \dint_{\cx}  f(x,t) \lf(t^2L\r)^k \Phi(t\sqrt{L})(g)(x)\,
\frac{d\mu(x)\,dt}{t}\r|\\
&&\ls\|\mathcal{A}f\|_{L^s(\cx)}\|S_{\Phi,\,L,\,k}(g)\|_{L^{s'}(\cx)}\ls
\|f\|_{T^s(\cx_+)}\|g\|_{L^{s'}(\cx)},
\end{eqnarray*}
which, together with the dual representation of $L^s(\cx)$ norm and
a density argument, implies that $\pi_{\Phi,\,L,\,k}$
extends to a bounded linear operator
from $T_2^{s}(\cx_+)$ to $L^s(\cx)$. This shows that (i) is
valid.

We now turn to the proof of (ii). By the definition of the Hardy space $H_L^p(\cx)$
(with $p\in(0,\,\fz)$)
associated with operators satisfying Assumptions $\mathrm{(H_1)}$ and $\mathrm{(H_2)}$
in \cite{hlmmy}, together with an argument similar to that used in the proof of
\cite[Proposition 9.1(v)]{hmm}, we see that, for all $s\in(p_L,\,p_L')$,
$H_L^s(\cx)=L^s(\cx)$. This, combined with the definition of
$H_L^s(\cx)$, immediately implies that the operator $t^2Le^{-t^2L}$
extends to a bounded linear operator
from $L^s(\cx)$ to $T_2^{s}(\cx_+)$. This shows (ii), which
completes the proof of Lemma \ref{cl4.11}.
\end{proof}

We now turn to the proof of Theorem \ref{ct4.9}.

\begin{proof}[Proof of Theorem \ref{ct4.9}]
The inclusion that $H^{M,\,q,\,s}_{\fai,\,L,\,\rm{at}}(\cx)
\subset H^{M,\,q}_{\fai,\,L,\,\rm{at}}(\cx)$ follows immediately from
Theorem \ref{ct4.3},  Lemma \ref{cl4.11} and \eqref{c4.10}.
We now prove the inclusion $H^{M,\,q}_{\fai,\,L,\,\rm{at}}(\cx)
\subset H^{M,\,q,\,s}_{\fai,\,L,\,\rm{at}}(\cx) $. To this end, we first recall
the following Calder\'on reproducing formula, which is deduced
from the bounded functional calculus in $L^2(\cx)$. More precisely,
let $\Phi$ be as in \eqref{c4.6}. For any
$f\in L^2(\cx)\cap L^s(\cx)\cap H^{M,\,q}_{\fai,\,L,\,
\rm{at}}(\cx)$, we see that there exists a positive constant $C_{(\Phi,\,M)}$
such that
\begin{eqnarray}\label{c4.25}
f=C_{(\Phi,\,M)}\dint_0^\fz \lf(t^2L\r)^{M+1}\Phi\lf(t\sqrt{L}\r)t^2Le^{-t^2L}f\,
\frac{dt}{t}
\end{eqnarray}
in $L^2(\cx)$.
Moreover, by Lemma \ref{cl4.11}(ii),
we know that $t^2Le^{-t^2L}f\in T^s_2(\cx_+)\cap T_\fai(\cx_+)$. Thus, by a slight
modification of the proof of \cite[Corollary 3.4]{hyy}, we conclude that
there exist $\{\lz_j\}_j\subset \cc$ and a sequence of $(T_\fai,\,\fz)$-atoms
$\{A_j\}_j$ associated with the balls $\{B_j\}_j$ such that
$$t^2Le^{-t^2L}f=\sum_j \lz_j A_j$$
in $T^s(\cx_+)$ and $T_\fai(\cx_+)$.
Now, let $g\in L^2(\cx)\cap L^{s'}(\cx)$. From \eqref{c4.25},
Fubini's theorem and Assumption $\mathrm{(H_1)}$,
we deduce that
\begin{eqnarray}\label{c4.26}
\int_{\cx}f(x){g}(x)\,d\mu(x)&&=
\dint_0^\fz\dint_{\cx} t^2L e^{-t^2L}f(x)
{\lf(t^2L\r)^M\Phi(t\sqrt{L})g}(x)\,\frac{d\mu(x)\,dt}{t}\\
&&\nonumber=\dsum_{j}
\dint_0^\fz\dint_{\cx} \lz_j A_j(y,\,t)
{\lf(t^2L\r)^M\Phi(t\sqrt{L})g}(x)\,\frac{d\mu(x)\,dt}{t}\\
&&\nonumber=\dsum_{j}\dint_{\cx} \lz_j \dint_0^\fz \lf(t^2L\r)^M
\Phi(t\sqrt{L}) (A_j)(x)\frac{dt}{t}g(x)\,d\mu(x)\\
&&\nonumber=:\dsum_{j}\dint_{\cx} \lz_j a_j(x) g(x)\,d\mu(x).
\end{eqnarray}
By the proof of Proposition \ref{cp4.4}, we conclude that, for each $j\in\nn$,
$a_j$ is a $(\fai,\,q,\,M)_L$-atom associated with $B_j$. This, together with
\eqref{c4.26}, implies that  $f$ has a $(\fai,\,q,\,s,\,M)_L$-atomic representation
$f=\sum_j \lz_ja_j$ as in Definition \ref{cd4.8}.
Thus, $f\in \wz{H}^{M,\,q,\,s}_{\fai,\,L,\,\rm{at}}(\cx)$, which,
together with the fact that $L^2(\cx)\cap L^s(\cx)\cap
H_{\fai,\,L,\,\rm{at}}^{M,\,q}(\cx)$ is
dense in $H^{M,\,q}_{\fai,\,L,\,
\rm{at}}(\cx)$ and
a density argument, completes the proof of the inclusion
$H^{M,\,q}_{\fai,\,L,\,
\rm{at}}(\cx) \subset H^{M,\,q,\,s}_{\fai,\,L,\,\rm{at}}(\cx)$
and hence Theorem \ref{ct4.9}.
\end{proof}

\section{A sufficient condition for the equivalence between the spaces
$H_{\fai,\,L}(\rn)$ and $H_\fai(\rn)$\label{s5}}

\hskip\parindent In this section, we give a sufficient condition
on the operator $L$, satisfying Assumptions {\rm (A)} and {\rm (B)},
such that $H_{\fai,\,L}(\rn)$ and $H_\fai(\rn)$ coincide
with equivalent quasi-norms.
We first recall some notions and properties of $H_\fai(\rn)$.

In what follows, we denote by $\cs(\rn)$ the \emph{space of all Schwartz functions} and
by $\cs'(\rn)$ its \emph{dual space} (namely, the \emph{space of all tempered distributions}).
For $m\in\nn$, define
$$\cs_m(\rn):=\lf\{\phi\in\cs(\rn):\ \sup_{x\in\rn}\sup_{
\bz\in\zz^n_+,\,|\bz|\le m+1}(1+|x|)^{(m+2)(n+1)}|\partial^\bz_x\phi(x)|\le1\r\}.
$$
Then for all $x\in\rn$ and $f\in\cs'(\rn)$, the \emph{non-tangential grand maximal function}
$f^\ast_m$ of $f$ is defined by setting,
$$f^\ast_m(x):=\sup_{\phi\in\cs_m(\rn)}\sup_{|y-x|<t,\,t\in(0,\fz)}|f\ast\phi_t(y)|,
$$
where for all $t\in(0,\fz)$, $\phi_t(\cdot):=t^{-n}\phi(\frac{\cdot}{t})$.
When $m(\fai):=\lfz n[q(\fai)/i(\fai)-1]\rfz$, where $q(\fai)$ and $i(\fai)$ are,
respectively, as in \eqref{2.7} and \eqref{2.5}, we \emph{denote $f^\ast_{m(\fai)}$ simply
by $f^\ast$}.

Now we recall the definition of the Musielak-Orlicz-Hardy $H_\fai(\rn)$ introduced
by Ky \cite{k} as follows.

\begin{definition}\label{d5.1}
Let $\fai$ be as in Definition \ref{d2.2}. The \emph{Musielak-Orlicz Hardy space
$H_\fai(\rn)$} is defined to be the space of all $f\in\cs'(\rn)$ such that
$f^\ast\in L^\fai(\rn)$
with the \emph{quasi-norm}
$\|f\|_{H_\fai(\rn)}:=\|f^\ast\|_{L^\fai(\rn)}$.
\end{definition}

To introduce the molecular Musielak-Orlicz-Hardy space, we first
introduce the notion of molecules associated with the growth function $\fai$.

\begin{definition}\label{d5.2}
Let $\fai$ be as in Definition \ref{d2.2}, $q\in(1,\fz)$, $s\in\zz_+$ and
$\varepsilon\in(0,\fz)$. A function $\az\in L^q(\rn)$ is called a
\emph{$(\fai,\,q,\,s,\,\uc)$-molecule} associated with the ball $B$ if

(i) for each $j\in\zz_+$, $\|\az\|_{L^q(S_j(B))}\le
2^{-j\uc}|2^j B|^{1/q}\|\chi_{B}\|_{L^\fai(\rn)}^{-1}$;

(ii) $\int_{\rn}\az(x)x^{\bz}\,dx=0$ for all $\bz\in\zz_+^n$ with $|\bz|\le s$.
\end{definition}

\begin{definition}\label{d5.3}
Let $\fai$ be as in Definition \ref{d2.2}, $q\in(1,\fz)$, $s\in\zz_+$
and $\uc\in(0,\fz)$. The \emph{molecular Musielak-Orlicz Hardy space},
$H^{q,\,s,\,\uc}_{\fai,\,\mathrm{mol}}(\rn)$, is defined to be the space of all
$f\in\cs'(\rn)$ satisfying that $f=\sum_j\lz_j\az_j$ in $\cs'(\rn)$,
where $\{\lz_j\}_j\subset\cc$ and $\{\az_j\}_j$ is a sequence of
$(\fai,\,q,\,s,\,\uc)$-molecules respectively associated to the balls $\{B_j\}_j$,
and
$$\sum_j\fai\lf(B_j,\,\frac{|\lz_j|}{\|\chi_{B_j}\|_{L^\fai(\rn)}}\r)<\fz,$$
where, for each $j$, the molecule $\az_j$ is associated with the ball $B_j$.
Moreover, define
\begin{eqnarray*}
\|f\|_{H^{q,\,s,\,\uc}_{\fai,\,\mathrm{mol}}(\rn)}:
=\inf\lf\{\Lambda
\lf(\lf\{\lz_j\az_j\r\}_{j}\r)\r\},
\end{eqnarray*}
where the infimum is taken over all decompositions of $f$ as above and
$$\Lambda
\lf(\lf\{\lz_j\az_j\r\}_{j}\r):=
\inf\lf\{\lz\in(0,\fz):\ \sum_{j}\fai\lf(B_j,\frac{|\lz_j|}{
\lz\|\chi_{B_j}\|_{L^\fai(\rn)}}\r)\le1\r\}.$$
\end{definition}

\begin{definition}\label{d5.4}
Let $\fai$ be as in Definition \ref{d2.2}.

(I) For each ball $B\subset\rn$, the \emph{space} $L^q_\fai(B)$ with
$q\in[1,\fz]$ is defined to be the set of all measurable functions
$f$ on $\rn$, supported in $B$, such that
\begin{equation*}
\|f\|_{L^q_{\fai}(B)}:=
\begin{cases}\dsup_{t\in (0,\fz)}
\lf[\frac{1}
{\fai(B,t)}\dint_{\rn}|f(x)|^q\fai(x,t)\,dx\r]^{1/q}<\fz,& q\in [1,\fz),\\
\|f\|_{L^{\fz}(B)}<\fz,&q=\fz.
\end{cases}
\end{equation*}

(II) A triplet $(\fai,\,q,\,s)$ is said to be \emph{admissible},
if $q\in(q(\fai),\fz]$ and $s\in\zz_+$ satisfying $s\ge\lfz
n[\frac{q(\fai)}{i(\fai)}-1]\rfz$, where
$q(\fai)$ and $i(\fai)$ are respectively as in \eqref{2.7}
and \eqref{2.5}. A measurable function $a$ on
$\rn$ is called a \emph{$(\fai,\,q,\,s)$-atom}, if there exists a ball
$B\subset\rn$ such that

$\mathrm{(i)}$ $\supp a\subset B$;

$\mathrm{(ii)}$
$\|a\|_{L^q_{\fai}(B)}\le\|\chi_B\|_{L^\fai(\rn)}^{-1}$;

$\mathrm{(iii)}$ $\int_{\rn}a(x)x^{\az}\,dx=0$ for all
$\az\in\zz_+^n$ with $|\az|\le s$.

(III) The  \emph{atomic Musielak-Orlicz Hardy space},
$H^{\fai,\,q,\,s}(\rn)$, is defined to be the space of all
$f\in\cs'(\rn)$ satisfying that $f=\sum_j\lz_ja_j$ in $\cs'(\rn)$,
where $\{\lz_j\}_j\subset\cc$ and $\{a_j\}_j$ is a sequence of
$(\fai,\,q,\,s)$-atoms associated with $\{B_j\}_j$, and
$$\sum_j\fai\lf(B_j,\,\frac{|\lz_j|}{\|\chi_B\|_{L^\fai(\rn)}}\r)<\fz.$$
Moreover, let
\begin{eqnarray*}
&&\blz(\{\lz_j a_j\}_j):= \inf\lf\{\lz\in(0,\fz):\ \ \sum_j\fai\lf(B_j,
\frac{|\lz_j|}{\lz\|\chi_{B_j}\|_{L^\fai(\rn)}}\r)\le1\r\}.
\end{eqnarray*}
The \emph{quasi-norm} of $f\in H^{\fai,\,q,\,s}(\rn)$
is defined by $\|f\|_{H^{\fai,\,q,\,s}(\rn)}:=\inf\{
\blz(\{\lz_ja_j\}_j)\}$, where the infimum is taken over all the
decompositions of $f$ as above.
\end{definition}

Then we have the following conclusion, which is just \cite[Theorem 4.11]{hyy}.

\begin{lemma}\label{l5.1}
Let $\fai$ be as in Definition \ref{d2.2}.
Assume that $(\fai,\,q,\,s)$ is
admissible, $\epz\in(\max\{n+s,nq(\fai)/i(\fai)\},\fz)$ and
$p\in(q(\fai)[r(\fai)]',\fz)$,
where $q(\fai)$, $i(\fai)$ and $r(\fai)$ are, respectively, as in \eqref{2.7},  \eqref{2.5}
and \eqref{2.8}. Then $H_\fai(\rn)$, $H^{\fai,\,q,\,s}(\rn)$ and $H^{p,\,s,\,
\epz}_{\fai,\,\mathrm{mol}}(\rn)$ coincide with equivalent quasi-norms.
\end{lemma}

Throughout this section, we always assume that the operator $L$ satisfies the following
additional assumption.

\medskip

\noindent {\bf Assumption (C).} The \emph{distribution
kernels} $h_t$ of $e^{-tL}$ satisfy
that there exist positive constants $C,\,\wz{c},\,c\in(0,\fz)$ and $\nu\in(0,1]$
such that, for all $t\in(0,\fz)$ and almost every
$x,\,y,\,h\in\rn$ with $2|h|\le t^{1/2m}+|x-y|$,
\begin{eqnarray}\label{5.0}
&&|h_t(x,y)|+|h_t(y,x)|\le
\frac{C}{t^{n/2m}}\exp\lf\{-\frac{c|x-y|^{2m/(2m-1)}}{t^{1/(2m-1)}}\r\},
\end{eqnarray}
\begin{eqnarray}\label{5.1}
&&|h_t(x+h,y)-h_t(x,y)|+|h_t(x,y+h)-h_t(x,y)|\\ \nonumber
&&\hs\le
\frac{C}{t^{n/2m}}\lf(\frac{|h|}{t^{1/2m}+|x-y|}\r)^\nu
\exp\lf\{-\frac{\wz{c}|x-y|^{2m/(2m-1)}}{t^{1/(2m-1)}}\r\}
\end{eqnarray}
and
\begin{eqnarray}\label{5.2}
\int_\rn h_t(x,y)\,dx\equiv1\equiv\int_\rn h_t(x,y)\,dy.
\end{eqnarray}

\smallskip

\begin{remark}\label{r5.1}
(i) If the operator $L$ satisfies Assumption (C), then $L$ satisfies Assumption {\rm (B)}
with $p_L=1$ and $q_L=\fz$.

(ii) Let $L:=-\mathrm{div}(A\nabla)$ be the divergence form elliptic operator in
$L^2(\rn)$, where $A$ has real entries when $n\ge3$ and complex entries when $n\in\{1,\,2\}$.
By \cite[Chapter 1]{at98}, we know that the operator $L$ satisfies Assumptions
{\rm (A)} and (C).
\end{remark}

We now in the position to state the main result of this section.

\begin{theorem}\label{t5.1}
Let $\fai$ be as in Definition \ref{d2.2} and $L$ satisfy Assumptions {\rm (A)} and {\rm (C)}.
Assume that $q(\fai)<\frac{n+\nu}{n}$ and $i(\fai)\in (\frac{nq(\fai)}{n+\nu},\,1]$,
where $\nu$, $q(\fai)$ and $i(\fai)$ are, respectively,
as in \eqref{5.1}, \eqref{2.7} and \eqref{2.5}. Then $H_{\fai,\,L}(\rn)$ and
$H_\fai(\rn)$ coincide with equivalent quasi-norms.
\end{theorem}

\begin{proof}
Let $f\in H_\fai(\rn)\cap L^2(\rn)$ and $q\in(q(\fai)[r(\fai)]',\fz)$.
Then from Lemma \ref{l5.1},
we deduce that there exist $\{\lz_j\}_j\subset\cc$ and a sequence $\{a_j\}_j$ of
$(\fai,\,q,\,0)$-atoms such that
\begin{equation}\label{5.3}
f=\sum_j\lz_j a_j
\end{equation}
in $\cs'(\rn)$
and
\begin{equation}\label{5.4}
\|f\|_{H_\fai(\rn)}\sim\blz(\{\lz_j a_j\}_j).
 \end{equation}
Moreover, by $f\in L^2(\rn)$ and the proof of \cite[Theorem 3.4]{k},
we know that \eqref{5.3} also holds in $L^2(\rn)$. Thus, to prove $f\in H_{\fai,\,L}(\rn)$,
it suffices to show that, for any $\lz\in\cc$ and $(\fai,\,q,\,0)$-atom $a$ associated
with the ball $B$,
\begin{equation}\label{5.5}
\int_\rn\fai\lf(x,S_L(\lz a)(x)\r)\,dx\ls\fai\lf(B,\,
\frac{|\lz|}{\|\chi_B\|_{L^\fai(\rn)}}\r).
\end{equation}
Indeed, if \eqref{5.5} holds true, by \eqref{5.3} and \eqref{5.4}, we see that
$\|f\|_{H_{\fai,\,L}(\rn)}\ls\blz(\{\lz_ja_j\}_j)\sim\|f\|_{H_\fai(\rn)}$.

Now we prove \eqref{5.5}. From \eqref{5.1} and \eqref{5.2}, we deduce that the kernel of
$t^{2m}Le^{-t^{2m}L}$, $q_{t^{2m}}$, satisfies that, for any $\gz\in(0,\nu)$, there exists
a positive constant $C_1$ such that for all $t\in(0,\fz)$ and almost every
$x,\,y,\,h\in\rn$ with $2|h|\le t+|x-y|$,
\begin{eqnarray}\label{5.6}
&&|q_{t^{2m}}(x+h,y)-q_{t^{2m}}(x,y)|+|q_{t^{2m}}(x,y+h)-
q_{t^{2m}}(x,y)|\\ \nonumber
&&\hs\ls\frac{1}{t^n}\lf(\frac{|h|}{t+|x-y|}\r)^\gz
\exp\lf\{-\frac{C_1|x-y|^{2m/(2m-1)}}{t^{2m/(2m-1)}}\r\}
\end{eqnarray}
and
\begin{eqnarray*}
\int_\rn q_{t^{2m}}(x,y)\,dx=0=\int_\rn q_{t^{2m}}(x,y)\,dy.
\end{eqnarray*}

Write
\begin{equation}\label{5.8}
\int_\rn\fai\lf(x,S_L(\lz a)(x)\r)\,dx=\int_{8B}\fai\lf(x,S_L(\lz a)(x)\r)\,dx
+\int_{(8B)^\complement}\cdots=:\mathrm{I}_1+\mathrm{I}_2.
\end{equation}

Moreover, since $q>q(\fai)[r(\fai)]'$, it follows that there exists $p\in(1,\fz)$
such that $p>[r(\fai)]'$ and $q/p>q(\fai)$, which implies that
$\fai\in\aa_{q/p}(\rn)$. From this, H\"older's inequality and Definition \ref{d2.1},
we infer that
\begin{eqnarray*}
\|a\|^p_{L^p(B)}&&\le\lf\{\int_B|a(x)|^q\fai
\lf(x,|\lz|\|\chi_B\|^{-1}_{L^\fai(B)}\r)\,dx\r\}^{p/q}\\
&&\hs\times\lf\{\int_B\lf[\fai\lf(x,|\lz|\|\chi_B\|^{-1}_{L^\fai(B)}\r)
\r]^{-\frac{p}{q}(\frac{q}{p})'}
\,dx\r\}^{1/(q/p)'}\ls\|a\|^p_{L^q_\fai(B)}|B|,
\end{eqnarray*}
which implies that
\begin{equation}\label{5.9}
\|a\|_{L^p(B)}\ls\|a\|_{L^q_\fai(B)}|B|^{1/p}.
\end{equation}
By this, the uniformly upper type 1 property of $\fai$, H\"older's inequality and
the $L^q(\rn)$-boundedness of $S_L$, we see that
\begin{eqnarray}\label{5.10}
\hs\hs\mathrm{I}_1&&\ls\int_{8B}\fai\lf(x,|\lz|\|\chi_B\|^{-1}_{L^\fai(B)}\r)\lf[1+
\|\chi_B\|_{L^\fai(B)}S_L(a)(x)\r]\,dx\\ \nonumber
&&\ls\lf[1+\|\chi_B\|_{L^\fai(B)}|8B|^{-1/p}\|a\|_{L^p(B)}\r]
\fai\lf(B,|\lz|\|\chi_B\|^{-1}_{L^\fai(B)}\r)\\ \nonumber
&&\ls\fai\lf(B,|\lz|\|\chi_B\|^{-1}_{L^\fai(B)}\r).
\end{eqnarray}

Now we estimate $\mathrm{I}_2$. For any $x\in(8B)^\complement$, we first write
\begin{eqnarray}\label{5.11}
[S_L(a)(x)]^2&&=\int_0^{r_B}\int_{B(x,t)}\lf|\int_B q_{t^{2k}}(y,z)a(z)\,dz\r|^2\,
\frac{dy\,dt}{t}+\int_{r_B}^\fz\cdots\\ \nonumber
&&=:\mathrm{E}_1(x)+\mathrm{E}_2(x).
\end{eqnarray}
 Notice that, for any $x\in(8B)^\complement$, $y\in B(x,t)$ with $t\in(0,r_B)$
 and $z\in B$, it holds that
\begin{eqnarray}\label{5.12}
|y-z|&&\ge|x-z|-|x-y|\ge|x-x_B|-|z-x_B|-r_B\\ \nonumber
&&\ge|x-x_B|-2r_B\ge\frac{1}{2}|x-x_B|.
\end{eqnarray}
Moreover, similar the proof of \eqref{5.9},
we see that
\begin{eqnarray}\label{5.13}
\|a\|_{L^1(B)}\le\|a\|_{L^q_\fai(B)}|B|.
\end{eqnarray}

Let $s\in[n+\nu,\fz)$. Then from \eqref{5.12}, \eqref{5.13} and \eqref{5.0},
we deduce that, for all $x\in(8B)^\complement$,
\begin{eqnarray}\label{5.14}
\mathrm{E}_1(x)
&&\ls\int_0^{r_B}\int_{B(x,t)}
\lf[\int_B\frac{1}{t^n}e^{-(\frac{\wz{c}|y-z|}{t})
^{2m/(2m-1)}}|a(z)|\,dz\r]^2\,\frac{dy\,dt}{t^{n+1}}\\ \nonumber
&&\ls\int_0^{r_B}\int_{B(x,t)}\frac{1}{t^{2n}}\lf(\frac{t}{|x-x_B|}\r)^{2s}
\|a\|^2_{L^1(B)}\,\frac{dy\,dt}{t^{n+1}}\\ \nonumber
&&\ls\frac{r_B^{2(s-n)}}{|x-x_B|^{2s}}\|a\|^2_{L^1(B)}
\ls\frac{r_B^{2s}}{|x-x_B|^{2s}}\|a\|^2_{L^q_\fai(B)}.
\end{eqnarray}

Now we deal with $\mathrm{E}_2(x)$. By $n+\nu>nq(\fai)/i(\fai)$, we know that
there exist $q_0\in(q(\fai),\fz)$, $p_0\in(0,i(\fai))$ and $\gz\in(0,\nu)$ such that
$n+\gz>nq_0/p_0$, which further implies that there exists $\gz_1\in(0,\gz)$ satisfying that
\begin{equation}\label{5.15}
n+\gz-\gz_1>nq_0/p_0.
\end{equation}
Moreover, for any $x\in(8B)^\complement$, $y\in B(x,t)$ with $t\in(r_B,\fz)$, and
$z\in B$, it holds that $t+|y-z|\ge|x-z|\ge\frac{1}{2}|x-x_B|$, which,
together with $\int_B a(x)\,dx=0$, \eqref{5.6} and \eqref{5.13}, implies that
\begin{eqnarray*}
\mathrm{E}_2(x)&&\ls\int_{r_B}^\fz\int_{B(x,t)}\lf[\int_B|q_{t^{2m}}
(y,z)-q_{t^{2m}}(y,x_B)|
|a(z)|\,dz\r]^2\,\frac{dy\,dt}{t^{n+1}}\\ \nonumber
&&\ls\int_{r_B}^\fz\int_{B(x,t)}\lf[\int_B\frac{|z-x_B|^\gz}
{(t+|y-z|)^{n+\gz}}|a(z)|\,dz\r]^2\,\frac{dy\,dt}{t^{n+1}}\\ \nonumber
&&\ls\int_{r_B}^\fz\int_{B(x,t)}\lf[\int_B\frac{r_B^\gz}
{t^{\gz_1}|x-x_B|^{n+\gz-\gz_1}}|a(z)|\,dz\r]^2\,\frac{dy\,dt}{t^{n+1}}\\ \nonumber
&&\ls\frac{r_B^{2\gz}}{|x-x_B|^{2(n+\gz-\gz_1)}}\|a\|^2_{L^1(B)}\int_{r_B}^\fz
t^{-2\gz_1-1}\,dt\\ \nonumber
&&\ls\frac{r_B^{2(\gz-\gz_1)}}
{|x-x_B|^{2(n+\gz-\gz_1)}}\|a\|^2_{L^q_\fai(B)}|B|^2\ls\frac{r_B^{2(n+\gz-\gz_1)}}
{|x-x_B|^{2(n+\gz-\gz_1)}}\|\chi_B\|^{-2}_{L^\fai(B)}.
\end{eqnarray*}
From this, \eqref{5.11} and \eqref{5.14}, it follows that, for all $x\in(8B)^\complement$,
$$S_L(a)(x)\ls\frac{r_B^{n+\gz-\gz_1}}
{|x-x_B|^{n+\gz-\gz_1}}\|\chi_B\|^{-1}_{L^\fai(B)},
$$
which, together with the uniformly lower type $p_0$ properties of $\fai$, Lemma \ref{l2.4}(vi)
and \eqref{5.15}, implies that
\begin{eqnarray*}
\mathrm{I}_2&&\ls\int_{(8B)^\complement}\fai\lf(x,|\lz|S_L(a)(x)\r)\,dx\\
&&\ls\int_{(8B)^\complement}\fai\lf(x,\frac{r_B^{n+\gz-\gz_1}|\lz|}
{|x-x_B|^{n+\gz-\gz_1}}\|\chi_B\|^{-1}_{L^\fai(B)}\r)\,dx\\ \nonumber
&&\ls\sum_{j=3}^\fz\fai\lf(2^jB,2^{-(n+\gz-\gz_1)j}|\lz|
\|\chi_B\|^{-1}_{L^\fai(B)}\r)\\ \nonumber
&&\ls\sum_{j=3}^\fz2^{-(n+\gz-\gz_1)p_0j+nq_0}\fai\lf(B,|\lz|
\|\chi_B\|^{-1}_{L^\fai(B)}\r)\ls\fai\lf(B,|\lz|\|\chi_B\|^{-1}_{L^\fai(B)}\r).
\end{eqnarray*}
By this, \eqref{5.8} and \eqref{5.10}, we see that \eqref{5.5} holds true.

Now we prove that $H_{\fai,\,L}(\rn)\cap L^2(\rn)\subset H_\fai(\rn)\cap L^2(\rn)$.
By Theorem \ref{t4.1} and Lemma \ref{l5.1}, we only need to show that,
for any given $(\fai,\,p_1,\,M,\,\epz)_L$-molecule $\az$, it holds that $\az$ is a
$(\fai,\,p_1,\,0,\,\epz)$-molecule as in Definition \ref{d5.2}, where
$p_1\in(q(\fai),\fz)$,
$M\in\nn$ with $M>\frac{nq(\fai)}{2mi(\fai)}$,
and $\epz\in(nq(\fai)/i(\fai),\fz)$.
Indeed, compared with Definitions \ref{d4.2}
and \ref{d5.2}, we know that, to show our conclusion, it suffices to prove that
\begin{equation}\label{5.16}
\int_\rn\az(x)\,dx=0.
\end{equation}
By the $H_\fz$-bounded functional calculus,
we know that, for all $f\in L^p(\rn)$ with $p\in(1,p_1]$,
$$(I+L)^{-1}f=\int_0^\fz e^{-t}e^{-tL}f\,dt,
$$
which, together with Fubini's theorem and \eqref{5.2}, implies that, for all
$f\in L^p(\rn)\cap L^1(\rn)$,
\begin{equation}\label{5.17}
\int_\rn (I+L)^{-1}f(x)\,dx=\int_0^\fz e^{-t}\int_\rn e^{-tL}f(x)\,dx\,dt=\int_\rn f(x)\,dx.
\end{equation}
Moreover, by the definition of $\az$, we know that $\az\in L^1(\rn)\cap L^p(\rn)$ and
there exists $b\in L^1(\rn)\cap L^p(\rn)$
such that $\az=Lb$. From this and \eqref{5.17}, we deduce that
\begin{eqnarray*}
\int_\rn\az(x)\,dx&&=\int_\rn(I+L)^{-1}Lb(x)\,dx\\
&&=\int_\rn(I+L)^{-1}(I+L)b(x)\,dx-
\int_\rn(I+L)^{-1}b(x)\,dx=0,
\end{eqnarray*}
which completes the proof of \eqref{5.16}.

By the above proofs, we see that $H_{\fai,\,L}(\rn)\cap L^2(\rn)$ and
$H_\fai(\rn)\cap L^2(\rn)$ coincide with equivalent quasi-norms,
which, together with the fact that $H_{\fai,\,L}(\rn)\cap
L^2(\rn)$ and $H_{\fai}(\rn)\cap L^2(\rn)$  are, respectively,
dense in $H_{\fai,\,L}(\rn)$ and $H_{\fai}(\rn)$, and a
density argument, implies that the spaces $H_{\fai,\,L}(\rn)$
and  $H_{\fai}(\rn)$ coincide with equivalent quasi-norms. This
finishes the proof of Theorem \ref{t5.1}.
\end{proof}

\section{The Musielak-Orlicz-Hardy space associated
with the second order elliptic operator in divergence form}\label{cs5}

\hskip\parindent In this section, we study the Musielak-Orlicz-Hardy space
$H_{\fai,\,L}(\rn)$ associated with the second order elliptic operator
in divergence form on $\rn$ with complex bounded measurable coefficients.
By making full use of the special structure of the divergence form elliptic operator,
we establish the radial and non-tangential maximal function
characterizations of $H_{\fai,\,L}(\rn)$
based respectively on the heat and Poisson semigroups of $L$. Moreover, we
establish the boundedness of the associated Riesz transform on
$H_{\fai,\,L}(\rn)$.

\subsection{Maximal function characterizations of $H_{\fai,\,L}(\rn)$}\label{cs5.1}

\hskip\parindent We begin this subsection by recalling some necessary
notions and notation. Let $A$ be an $n\times n$ matrix with entries
$\{a_{i,\,j}\}_{i,\,j=1}^n
\subset L^\fz(\rn,\,\cc)$ satisfying the \emph{ellipticity condition}, namely,
there exist constants $0<\lz_A\le \blz_A<\fz$ such that, for all
$\xi,\,\zeta\in\cc$ and almost every $x\in\rn$,
 \begin{eqnarray*}
 \lz_A|\xi|^2\le \Re e \langle A(x)\xi, \xi\rangle \ \ \text{and}\ \
 \lf|\langle A(x)\xi, \zeta\rangle\r|\le \blz_A|\xi||\zeta|,
 \end{eqnarray*}
 where $\langle \cdot,\, \cdot\rangle$ denotes the \emph{inner product} in $\cc$
 and $\Re e\, \xi$ denotes the \emph{real part} of the complex number $\xi$.
 Then the \emph{second order elliptic operator} $L$ \emph{in divergence form}
 is defined by
\begin{eqnarray}\label{c5.1x}
L f:=-\mathrm{div}(A\nabla f),
\end{eqnarray}
 interpreted in the weak sense via a sesquilinear form. It is well known that
 there exists a positive constant $\omega\in[0,\,\pi/2)$ such that the operator
 $L$ is of type $\omega$ on $L^2(\rn)$ and $L$ has a bounded $H_\fz$-functional
 calculus on $L^2(\rn)$ (see, for example, \cite{adm95,hmm}).
 Moreover, let $(p_-(L),\,p_+(L))$ be the \emph{interior of the maximal interval of exponents}
 $p\in[1,\,\fz]$ \emph{for which the semigroup} $\{e^{-tL}\}_{t>0}$, \emph{generated by}
 $L$, \emph{is} $L^p(\rn)$ \emph{bounded}. By \cite[Proposition 3.2]{au07}
 (see also \cite[Lemma 2.25]{hmm}),
 we conclude that, for all $p_-(L)<p\le q<p_+(L)$,  $\{e^{-tL}\}_{t>0}$ satisfy the
 $L^p-L^q$ off-diagonal estimates. Thus, $L$ satisfies Assumptions {\rm (A)} and
 {\rm (B)} with $k=2$. Therefore, a corresponding theory of the Musielak-Orlicz-Hardy space
 $H_{\fai,\,L}(\rn)$, including its molecular characterization (see Theorem \ref{t4.1})
 is already known.

We also recall the definitions of some maximal functions associated with $L$
from \cite{hm09}.
Let $f\in L^2(\rn)$ and $x\in\rn$, the \emph{radial maximal
functions}, $\mathcal{R}^\az_h$ and $\mathcal{R}^\az_P$, respectively associated
with the heat semigroup and Poisson semigroup
generated by $L$ are defined by setting, for all $\az\in(0,\,\fz)$, $f\in L^2(\rn)$
and $x\in\rn$,
\begin{eqnarray}\label{c5.2}
\mathcal{R}^\az_h(f)(x):=\sup_{t>0}\lf\{\frac{1}{(\az t)^n}\dint_{B(x,\,\az t)}
\lf|e^{-t^2L}(f)(y)\r|^2\,dy\r\}^{\frac{1}{2}}
\end{eqnarray}
and
\begin{eqnarray}\label{c5.3}
\mathcal{R}^\az_P(f)(x):=\sup_{t>0}\lf\{\frac{1}{(\az t)^n}\dint_{B(x,\,\az t)}
\lf|e^{-t\sqrt L}(f)(y)\r|^2\,dy\r\}^{\frac{1}{2}}.
\end{eqnarray}
Similarly, the \emph{non-tangential maximal functions}, $\mathcal{N}^\az_h$
and $\mathcal{N}^\az_P$, respectively associated with the heat semigroup and Poisson semigroup
generated by $L$ are defined by setting, for all $\az\in(0,\,\fz)$, $f\in L^2(\rn)$
and $x\in\rn$,
\begin{eqnarray}\label{c5.4}
\mathcal{N}^\az_h(f)(x):=\sup_{(y,t)\in\Gamma_\az(x)}
\lf\{\frac{1}{(\az t)^n} \int_{B(y,\az t)}
\lf|e^{-t^2 L}(f)(z)\r|^2\,dz\r\}^{\frac{1}{2}}
\end{eqnarray}
and
\begin{eqnarray}\label{c5.5}
\mathcal{N}^\az_P(f)(x):=\sup_{(y,t)\in\Gamma_\az(x)}
\lf\{\frac{1}{(\az t)^n} \int_{B(y,\az t)}
\lf|e^{-t\sqrt L}(f)(z)\r|^2\,dz\r\}^{\frac{1}{2}},
\end{eqnarray}
where above and in what follows,
$\bgz_\az(x):=\{(y,t)\in\rn\times(0,\fz):\ |x-y|<\az t\}$.
In what follows, when $\az=1$, we remove the superscript
$\az$ for simplicity.
We also define the \emph{Lusin-area functions}, $S_h$ and $S_P$,
associated respectively to the
heat semigroup and Poisson semigroup
by setting, for  all $f\in L^2(\rn)$ and $x\in\rn$,
\begin{eqnarray}\label{c5.x6}
S_h(f)(x):=\lf\{\int_{\Gamma (x)}
\lf|t\nabla e^{-t^2 L}(f)(y)\r|^2\,\frac{dy\,dt}{t}\r\}^{\frac{1}{2}}
\end{eqnarray}
and
\begin{eqnarray}\label{c5.6}
S_P(f)(x):=\lf\{\int_{\Gamma (x)}
\lf|t\nabla e^{-t\sqrt L}(f)(y)\r|^2\,\frac{dy\,dt}{t}\r\}^{\frac{1}{2}}.
\end{eqnarray}

We first introduce the Musielak-Orlicz-Hardy space, defined via the above maximal functions,
as follows.

\begin{definition}\label{cd5.1}
Let $\fai$ and $L$ be respectively as in Definition \ref{d2.2} and \eqref{c5.1x},
and $S_P$ as in \eqref{c5.6}.
The $S_P$-\emph{adapted Musielak-Orlicz-Hardy space} $H_{\fai,\,S_P}(\rn)$
is defined to be the completion of the \emph{set}
\begin{eqnarray*}
\lf\{f\in L^2(\rn):\ \|f\|_{H_{\fai,\,S_P}(\rn)}:=
\|S_P(f)\|_{L^\fai(\rn)}<\fz\r\}
\end{eqnarray*}
with respect to the \emph{quasi-norm} $\|\cdot\|_{H_{\fai,\,S_P}(\rn)}$.

In a similar way, the $S_h$-\emph{adapted}, $\mathcal{R}_h$-\emph{adapted},
$\mathcal{R}_P$-\emph{adapted}, $\mathcal{N}_h$-\emph{adapted} and
$\mathcal{N}_P$-\emph{adapted Musielak-Orlicz-Hardy spaces},
$$H_{\fai,\,S_h}(\rn),\ H_{\fai,\,\mathcal{R}_h}(\rn),\ H_{\fai,\,\mathcal{R}_P}(\rn),\
H_{\fai,\,\mathcal{N}_h}(\rn)\ \text{and}\ H_{\fai,\,\mathcal{N}_P}(\rn)$$
are also defined.
\end{definition}

Following \cite{au07}, let $(q_-(L),\,q_+(L))$ be the \emph{interior of
the maximal interval of exponents}
$p\in[1,\,\fz]$, \emph{for which the family of operators}, $\{\sqrt t \nabla
e^{-tL}\}_{t>0}$, \emph{is} $L^p(\rn)$ \emph{bounded}. From
\cite[Proposition 3.7]{au07}, it follows that $q_-(L)=p_-(L)$ and $q_+(L)>2$.
Moreover, by \cite[Corollary 3.8 and Proposition 3.9]{au07}, we know that,
for all $q_-(L)<p\le q<q_+(L)$, the family of operators, $\{\sqrt t \nabla
e^{-tL}\}_{t>0}$, satisfies the $L^p-L^q$ off-diagonal estimates.

For the operator $S_P$, we have the following boundedness.

\begin{lemma}\label{cl5.2}
Let $S_h$ and $S_P$ be respectively
as in \eqref{c5.x6} and \eqref{c5.6}.
Then, for all $p\in(q_-(L),\,q_+(L))$, both $S_h$ and
$S_P$ are bounded on $L^p(\rn)$.
\end{lemma}

The proof of Lemma \ref{cl5.2} follows from a similar method used for the vertical
Lusin-area function associated with the heat semigroup (see \cite[Theorem 6.1]
{au07}). We omit the details here.

\begin{lemma}\label{off-diagonal for the case of Poisson semigroup}
Let $\fai$ and $L$ be respectively as in Definition \ref{d2.2} and \eqref{c5.1x},
$$q\in (q_-(L),\,\min\{q_+(L),\,p_+(L)\})$$
and $M\in\nn$. Then, there exists a positive constant $C$ such that, for all
$t\in(0,\,\fz)$, close sets $E,\,F\subset\rn$ satisfying $d(E,\,F)>0$
and $f\in L^q(\rn)$ with $\supp f\subset E$,
\begin{eqnarray}\label{c5.7}
\lf\{\dint_0^\fz \lf\|s\nabla e^{-s^2L}
\lf(I-e^{-t^2L}\r)^Mf\r\|^2_{L^q(F)}\,\frac{ds}{s}\r\}^{\frac{1}{2}}\le C
\lf[\frac{t}{d(E,\,F)}\r]^{2M}\|f\|_{L^q(E)}
\end{eqnarray}
and
\begin{eqnarray}\label{c5.8}
\lf\{\dint_0^\fz \lf\|s\nabla e^{-s^2L}
\lf(t^2Le^{-t^2L}\r)^Mf\r\|^2_{L^q(F)}\,\frac{ds}{s}\r\}^{\frac{1}{2}}\le C
\lf[\frac{t}{d(E,\,F)}\r]^{2M}\|f\|_{L^q(E)}.
\end{eqnarray}
\end{lemma}

\begin{proof}
We first prove \eqref{c5.7}. To this end, by the change of variable, we write
\begin{eqnarray}\label{c5.9}
&&\lf\{\dint_0^\fz \lf\|s\nabla e^{-s^2L}
\lf(I-e^{-t^2L}\r)^Mf\r\|^2_{L^q(F)}\,\frac{ds}{s}\r\}^{\frac{1}{2}}\\
&&\nonumber\hs\ls
\lf\{\dint_0^t \lf\|s\nabla e^{-s^2(M+1)L}
\lf(I-e^{-t^2L}\r)^Mf\r\|_{L^q(F)}^{2}\,\frac{ds}{s}
\r\}^{\frac{1}{2}}+
\lf\{\dint_t^\fz \cdots\,\frac{ds}{s}\r\}^{\frac{1}{2}}\\
&&\nonumber\hs=:
\mathrm{H}+\mathrm{K}.
\end{eqnarray}
For $\mathrm{H}$, we deduce, from the binomial theorem,
Assumption ${\rm (B)}$ and the fact when $t\ge s$,
$(kt\nabla  e^{-(kt)^2L})(s\nabla e^{-s^2(M+1)L})$
satisfies the $L^q$ off-diagonal estimates
in $(kt)^2$ (see, for example, \cite[Lemma 2.22]{hmm}), that
\begin{eqnarray}\label{c5.10}
\hs\hs\quad\mathrm{H}&&\ls \lf\{\dint_0^t \lf\|s\nabla e^{-s^2(M+1)L}
f\r\|_{L^q(F)}^{2}\,\frac{ds}{s}
\r\}^{\frac{1}{2}}\\&&\nonumber\hs
+\sup_{1\le k\le M} \lf\{\dint_0^t \lf\|\lf(kt\nabla  e^{-(kt)^2L}\r)
\lf(e^{-s^2(M+1)L}\r)
f\r\|_{L^q(F)}^{2}\,\frac{s^2ds}{t^2s}
\r\}^{\frac{1}{2}}\\&&\nonumber
\ls\Bigg\{\lf[\dint_0^t \exp\lf\{-\frac{[d(E,\,F)]^2}
{s^2}\r\}\,\frac{ds}{s}\r]^{\frac{1}{2}}\\
&&\nonumber\hs
+\sup_{1\le k\le M} \lf[\dint_0^t \exp\lf\{-\frac{[d(E,\,F)]^2}{(k t)^2}
\r\}\,\frac{s ds}{t^2}\r]^{\frac{1}{2}}\Bigg\}\|f\|_{L^q(E)}\ls
\lf[\frac{t}{d(E,\,F)}\r]^{2M}\|f\|_{L^q(E)}.
\end{eqnarray}
Similarly, we have
\begin{eqnarray*}
\mathrm{K}&&\ls \lf\{\dint_t^\fz \lf\|s\nabla e^{-s^2L}
\lf(e^{-s^2L}-e^{-(s^2+t^2)L}\r)^Mf\r\|_{L^q(F)}^{2}\,\frac{ds}{s}
\r\}^{\frac{1}{2}}\\&&
\ls\lf[\dint_t^\fz  \exp\lf\{-\frac{[d(E,\,F)]^2}{s^2}\r\}
\lf(\frac{t^2}{s^2}\r)^{2M}\,\frac{ds}{s}
\r]^{\frac{1}{2}}\|f\|_{L^q(E)}\\&&\ls
\lf[\frac{t}{d(E,\,F)}\r]^{2M}\|f\|_{L^q(E)},
\end{eqnarray*}
which, together with \eqref{c5.9} and \eqref{c5.10}, shows immediately
that \eqref{c5.7} holds.
The proof of \eqref{c5.8} is similar to that of \eqref{c5.7}. We omit the details
here.
\end{proof}

Now, we are in the position to state our first main result in this section.

\begin{proposition}\label{cp5.4}
Let $\fai$ and $L$ be respectively as in Definition \ref{d2.2} and \eqref{c5.1x},
$S_h$ and $S_P$ respectively as in \eqref{c5.x6} and \eqref{c5.6}.
Assume further that $\fai\in\rh_{(\min\{p_+(L),\,q_+(L)\}/I(\fai))'}(\cx)$,
where $I(\fai)$ is  as in \eqref{2.4}. Then $H_{\fai,\,S_h}(\rn)$,
$H_{\fai,\,S_P}(\rn)$ and $H_{\fai,\,L}(\rn)$ coincide with equivalent quasi-norms.
\end{proposition}

\begin{proof}
We prove Proposition \ref{cp5.4} by first showing that
$L^2(\rn)\cap H_{\fai,\,S_P}(\rn)$ and $L^2(\rn)\cap H_{\fai,\,L}(\rn)$
coincide with equivalent quasi-norms,
whose proof is divided into two different directions of inclusions.
We first prove the inclusion $L^2(\rn)\cap H_{\fai,\,S_P}(\rn)
\subset L^2(\rn)\cap H_{\fai,\,L}(\rn)$. Let $f\in L^2(\rn)\cap H_{\fai,\,S_P}(\rn)$.
For all $x\in\rn$, let
\begin{eqnarray*}
\wz S_{P}(f)(x):=\lf\{\int_{\Gamma(x)} \lf|t^2Le^{-t \sqrt L}(f)(y)\r|^2
\,\frac{dy\,dt}{t^{n+1}}\r\}^{\frac{1}{2}}.
\end{eqnarray*}
From the proof of \cite[Lemma 5.4]{hm09}, we deduce that,
for all $x\in\rn$,
$\wz S_{P}(f)(x)\ls  S_{P}(f)(x)$,
which immediately implies that
\begin{eqnarray}\label{c5.11}
\lf\|\wz S_{P}(f)\r\|_{L^\fai(\rn)}
\ls \|S_{P}(f)\|_{L^\fai(\rn)}.
\end{eqnarray}
 Moreover, by the
$L^2(\rn)$-boundedness of $S_P$ (see, for example, \cite[(5.15)]{hm09}),
we know that
\begin{eqnarray*}
\lf\|\wz S_{P}(f)\r\|_{L^2(\rn)}\ls \lf\| S_{P}(f)\r\|_{L^2(\rn)}\ls \|f\|_{L^2(\rn)}.
\end{eqnarray*}
Thus, $t^2Le^{-t\sqrt L}f\in T_\fai({\rr^{n+1}_+})\cap T^2_2({\rr^{n+1}_+})$.
Using the bounded $H_\fz$-functional calculus in $L^2(\rn)$, we see that
$f=\pi_{L,\,M}(t^2Le^{-t\sqrt{L}}f)$
in $L^2(\rn)$, where $\pi_{L,\,M}$ is as in \eqref{4.2}, which, combining with
Proposition \ref{p4.1} and \eqref{c5.11}, implies that
\begin{eqnarray*}
\|f\|_{H_{\fai,\,L}(\rn)}&&\ls\lf\|t^2Le^{-t\sqrt{L}}f\r\|_{T_\fai({\rr^{n+1}_+})}
\sim\lf\|\wz{S}_P(f)\r\|_{L^\fai(\rn)}\ls\|f\|_{H_{\fai,\,S_P}(\rn)}<\fz.
\end{eqnarray*}
This shows $f\in L^2(\rn)\cap H_{\fai,\,L}(\rn)$ and hence the
inclusion $L^2(\rn)\cap H_{\fai,\,S_P}(\rn)
\subset L^2(\rn)\cap H_{\fai,\,L}(\rn)$.

Now, we prove the
inclusion $L^2(\rn)\cap H_{\fai,\,L}(\rn)
\subset L^2(\rn)\cap H_{\fai,\,S_P}(\rn)$.
To this end, it suffices to show that
the operator  $S_P$ is bounded from
$H_{\fai,\,L}(\rn)$ to $L^\fai(\rn)$. Moreover,
by Corollary \ref{sufficient condition for molecule}, we
only need to show that, for any $\lz\in\cc$ and
$(\fai,\,q,\,M,\,\epsilon)_L$-molecule $\az$ associated with $B$,
\begin{eqnarray}\label{c5.12}
\dint_{\rn}\fai\lf(x,\,S_P(\lz \az)(x)\r)\,dx\ls \fai\lf(B,\,\frac{|\lz|}
{\|\chi_B\|_{L^\fai(\rn)}}\r),
\end{eqnarray}
where $q\in(p_-(L),\min\{p_+(L),q_+(L)\})$, $\epsilon\in(0,\,\fz)$ and
$M\in\nn$ can be chosen sufficiently large.

To prove \eqref{c5.12}, we first write
\begin{eqnarray}\label{c5.13}
&&\dint_{\rn}\fai\lf(x,\,S_P(\lz \az)(x)\r)\,dx\\
&&\nonumber\hs\ls \dint_{\rn}\fai\lf(x,\,S_P
\lf(I-e^{-r_B^2L}\r)^M(\lz \az)(x)\r)\,dx\\
&&\nonumber\hs\hs+
\dint_{\rn}\fai\lf(x,\,S_P\lf[I-\lf(I-e^{-r_B^2L}\r)^M\r]
(\lz \az)(x)\r)\,dx=:\mathrm{I}+\mathrm{J}.
\end{eqnarray}
For $\mathrm{I}$, let $p_1\in [I(\fai),\,1]$ and
$p_2\in(0,\,i(\fai))$ such that $\fai$ is of uniformly upper type $p_1$ and
lower type $p_2$.
By Minkowski's inequality and Lemma \ref{l2.1}(i), we conclude that
\begin{eqnarray}\label{c5.14}
\mathrm{I}&&\ls \dsum_{i\in\nn} \dsum_{j\in\nn}
\dint_{S_j(2^iB)}\fai\lf(x,\,|\lz| S_P
\lf(I-e^{-r_B^2L}\r)^M\lf(\chi_{S_i(B)} \az\r)(x)\r)\,dx\\
&& \nonumber\ls \dsum_{i\in\nn} \dsum_{j\in\nn}
\lf\{\dint_{S_j(2^iB)}\lf[S_P
\lf(I-e^{-r_B^2L}\r)^M\lf(\chi_{S_i(B)} \az\r)(x)
\|\chi_{B}\|_{L^\fai(\rn)}\r]^{p_1}\r.\\
&&\nonumber\lf.\hs\hs\times
\fai\lf(x,\,\frac{|\lz|}{\|\chi_B\|_{L^\fai(\rn)}}\r)\,dx\r.\\
&& \nonumber\hs\lf.+\dint_{S_j(2^iB)}\lf[S_P
\lf(I-e^{-r_B^2L}\r)^M\lf(\chi_{S_i(B)} \az\r)(x)
\|\chi_{B}\|_{L^\fai(\rn)}\r]^{p_2}\r.\\
&&\nonumber\lf.\hs\hs\times
\fai\lf(x,\,\frac{|\lz|}{\|\chi_B\|_{L^\fai(\rn)}}\r)\,dx\r\}
=:\dsum_{i\in\nn} \dsum_{j\in\nn}\lf\{\wz{\rm{I}}_{i,\,j}+
\wz{\rm{II}}_{i,\,j}\r\}.
\end{eqnarray}
We first estimate $\wz{\rm{I}}_{i,\,j}$ in the case
when $j\in\{0,\,\ldots,\,4\}$.  Let $q\ge2$ and
\begin{eqnarray}\label{rangeofq}
q\in (I(\fai)[r(\fai)]',\,[q_+(L)]')\cap (p_-(L),\,\min\{p_+(L),\,q_+(L)\})
\end{eqnarray}
such that \eqref{c4.11} holds true.
Let $\wz p\in (q(\fai),\,\fz)$. Then $\fai\in \rh_{(\frac{q}{p_1})'}(\rn)\cap
\mathbb{A}_{\wz p}(\rn)$. This, together with H\"older's inequality,
Lemma \ref{off-diagonal for the case of Poisson semigroup}
and the $L^q(\rn)$-boundedness of the semigroup
$\{e^{-tL}\}_{t>0}$ for $q\in (p_-(L),\,p_+(L))$, implies that
\begin{eqnarray*}
\wz{\rm{I}}_{i,\,j}&&\ls \lf\{\dint_{S_j(2^iB)}\lf[S_P
\lf(I-e^{-r_B^2L}\r)^M\lf(\chi_{S_i(B)} \az\r)(x)
\|\chi_{B}\|_{L^\fai(\rn)}\r]^{q}\r\}^{\frac{p_1}{q}}\\
&&\hs\times
\lf\{\frac{1}{\lf|2^{i+j}B\r|} \dint_{S_j(2^iB)}
\lf[\fai\lf(x,\,\frac{|\lz|}{\|\chi_B\|_{L^\fai(\rn)}}\r)\r]^{(\frac{q}{p_1})'}
\,dx\r\}^{\frac{1}{(\frac{q}{p_1})'}}\\
&&\hs\times
\lf|2^{i+j}B\r|^{\frac{1}
{(\frac{q}{p_1})'}}\|\chi_{B}\|_{L^\fai(\rn)}^{p_1}\\
&&\ls \|\az\|_{L^q(S_i(B))}^{p_1}\|\chi_B\|_{L^\fai(\rn)}^{p_1}
\lf|2^{i+j}B\r|^{-\frac{p_1}{q}}\lf\{\dint_{S_j(2^iB)} \fai\lf(x,\,
\frac{|\lz|}{\|\chi_{B}\|_{L^\fai(\rn)}}\r)\,dx\r\}.
\end{eqnarray*}
From this, Definition \ref{d4.2} and Lemma \ref{l2.4}(vii), we deduce that
\begin{eqnarray*}
\wz{\rm{I}}_{i,j}\ls 2^{-ip_1[\epsilon+\frac{n}{q}-\frac{n\wz q}{p_1}]}
\fai\lf(B,\,\frac{|\lz|}{\|\chi_B\|_{L^\fai(\rn)}}\r),
\end{eqnarray*}
when $\epsilon>n(\frac{\wz q}{p_1}-\frac{1}{q})$.

We now estimate $\wz{\rm{I}}_{i,\,j}$ in the case
when $j\ge 5$. Similar to the case when $j\le 4$, we first
have
\begin{eqnarray}\label{c5.16}
\wz{\rm{I}}_{i,\,j}&&\ls 2^{-(i+j)n(\frac{p_1}{q}-\wz q)}
\|\chi_B\|_{L^\fai(\rn)}^{p_1} |B|^{-\frac{p_1}{q}}
\fai\lf(B,\,\frac{|\lz|}{\|\chi_B\|_{L^\fai(\rn)}}\r) \\
&&\nonumber\hs\times  \lf\{ \dint_{S_j(2^iB)}\lf[S_P
\lf(I-e^{-r_B^2L}\r)^M\lf(\chi_{S_i(B)} \az\r)(x)
\|\chi_{B}\|_{L^\fai(\rn)}\r]^{q}\r\}^{\frac{p_1}{q}}\\
&&\nonumber=: 2^{-(i+j)n(\frac{p_1}{q}-\wz q)}
\|\chi_B\|_{L^\fai(\rn)}^{p_1} |B|^{-\frac{p_1}{q}}
\fai\lf(B,\,\frac{|\lz|}{\|\chi_B\|_{L^\fai(\rn)}}\r)\lf(\mathcal{A}_{i,\,j}\r)^{p_1}.
\end{eqnarray}
For $\mathcal{A}_{i,\,j}$, since $q\ge 2$, by the dual norm representation
of the $L^{\frac{q}{2}}(\rn)$-norm, we know that there exists
$g\in L^{(\frac{q}{2})'}(\rn)$, with $\|g\|_{L^{(\frac{q}{2})'}(\rn)}\le 1$, such that
\begin{eqnarray}\label{c5.17}
\hs\hs\quad\hs\mathcal{A}_{i,\,j}
&&\sim \lf\{ \dint_{S_j(2^iB)}\lf[\int_{\Gamma(x)}
\lf|t\nabla e^{-t\sqrt L} \lf(I-e^{-r_B^2L}\r)^M \lf(\chi_{S_i(B)} \az\r)(y)\r|^2
\frac{dy\,dt}{t^{n+1}}\r]^{\frac{q}{2}}\,dx\r\}^{\frac{1}{q}}\\
&&\nonumber\sim\lf\{ \dint_{S_j(2^iB)}\lf[\int_{\Gamma(x)}
\lf|t\nabla e^{-t\sqrt L} \lf(I-e^{-r_B^2L}\r)^M \lf(\chi_{S_i(B)} \az\r)(y)\r|^2
\frac{dy\,dt}{t^{n+1}}\r]g(x)\,dx\r\}^{\frac{1}{2}}\\
&&\nonumber\ls\lf\{ \dint_0^\fz \dint_{\rn\setminus (2^{i+j-2} B)}
\lf|t\nabla e^{-t\sqrt L} \lf(I-e^{-r_B^2L}\r)^M \lf(\chi_{S_i(B)} \az\r)(y)\r|^2\r.\\
&&\nonumber\hs\times\lf.
\mathcal{M}(g)(y)\,\frac{dy\,dt}{t}\r\}^{\frac{1}{2}}
+\sum_{k=0}^{j-2} \lf\{ \dint_{2^i(2^{j-1}-2^k)r_B}^\fz
\int_{S_k(2^iB)}\cdots\,\frac{dy\,dt}{t}\r\}^{\frac{1}{2}}\\
&&\nonumber=:\wz{\mathcal{A}}_{i,\,j}+\dsum_{k=0}^{j-2}\wz{\mathcal{A}}_{i,\,j,\,k},
\end{eqnarray}
where $\mathcal{M}$ denotes the classical Hardy-Littlewood maximal
function.

To estimate $\wz{\mathcal{A}}_{i,\,j}$, we need the following
\emph{subordination formula},
\begin{eqnarray}\label{c5.18}
e^{-t\sqrt L}=C\dint_0^\fz \frac{e^{-u}}{\sqrt{u}}e^{-\frac{t^2L}{4u}}\,du,
\end{eqnarray}
where $C$ is a positive constant. By using
H\"older's inequality, \eqref{c5.18}, Minkowski's integral inequality, the
$L^{(\frac{q}{2})'}(\rn)$-boundedness of the Hardy-Littlewood maximal
function and
Lemma \ref{cl5.2}, we conclude that
\begin{eqnarray}\label{c5.19}
\hs\hs\quad\hs\wz{\mathcal{A}}_{i,\,j}&&\ls
\lf\{ \dint_0^\fz \lf[\dint_{\rn\setminus (2^{i+j-2} B)}
\lf|t\nabla e^{-t\sqrt L} \lf(I-e^{-r_B^2L}\r)^M \lf(\chi_{S_i(B)} \az\r)(y)\r|^q
\,dy\r]^{\frac{2}{q}}\,\frac{dt}{t}\r\}^{\frac{1}{2}}\\
&&\nonumber\ls \dint_0^\fz e^{-u}
\lf\{\dint_0^\fz\lf\|\frac{t}{\sqrt{4u}}\nabla e^{-\frac{t^2}{4u} L}
\lf(I-e^{-r_B^2L}\r)^M\r.\r.\\
&&\nonumber\lf.\hs\times\lf(\chi_{S_i(B)} \az\r)(y)\bigg\|_{L^q(\rn\setminus
(2^{i+j-2} B))}\,\frac{dt}{t}\r\}^{\frac{1}{2}}\,du\\
&&\nonumber\ls \dint_0^\fz e^{-u}\lf[\frac{r_B}{2^{i+j}r_B}\r]^{2M}\,du
\|\az\|_{L^q(S_i(B))}\ls 2^{-(i+j)M}\|\az\|_{L^q(S_i(B))}.
\end{eqnarray}
We continue to estimate $\wz{\mathcal{A}}_{i,\,j,\,k}$. Similar
to the estimates for $\wz{\mathcal{A}}_{i,\,j}$,
we first conclude that
\begin{eqnarray*}
\wz{\mathcal{A}}_{i,\,j,\,k}&&\ls \dint_{0}^\fz e^{-u}
\lf\{ \dint_{2^i(2^{j-1}-2^k)r_B}^\fz\lf\|\frac{t}{\sqrt{4u}}\nabla e^{-\frac{t^2}{4u} L}
\lf(I-e^{-r_B^2L}\r)^M\r.\r.\\
&&\nonumber\hs\times\lf(\chi_{S_i(B)} \az\r)\bigg\|^2_{L^q(S_k(2^iB))}\,\frac{dt}{t}
\Bigg\}^{\frac{1}{2}}\,du\\
&&\nonumber\ls\dint_{0}^\fz e^{-u}
\lf\{ \dint_{\frac{[2^i(2^{j-1}-2^k)r_B]^2}{4u(M+1)}}^\fz
\lf\|\sqrt{s}\nabla e^{-sL}\lf[\frac{s}{r_B^2}
\lf(e^{-sL}-e^{-(s+r_B^2)L}\r)\r]^M\r.\r.\\
&&\nonumber\hs\times\lf(\chi_{S_i(B)} \az\r)\bigg\|^2_{L^q(S_k(2^iB))}
\lf(\frac{r_B^2}{s}\r)^{2M}
\,\frac{ds}{s}\Bigg\}^{\frac{1}{2}}\,du.
\end{eqnarray*}
By the $L^q$ off-diagonal estimates (similar to the estimates used in
\eqref{c5.10}) and the change of variable (let $\wz s:=\frac{(2^{i+j}r_B)^2}
{1+u}\frac{1}{s}$), we further find that
\begin{eqnarray*}
\wz{\mathcal{A}}_{i,\,j,\,k}
&&\ls\|\az\|_{L^q(S_i(B))}\dint_{0}^\fz e^{-u}
\lf\{ \dint_{\frac{[2^i(2^{j-1}-2^k)r_B]^2}{4u(M+1)}}^\fz
\exp\lf\{-\frac{[2^{i+j}r_B]^2}{s(1+u)}\r\}
\lf(\frac{r_B^2}{s}\r)^{2M}\,\frac{ds}{s}\r\}^{\frac{1}{2}}\,du\\
&&\ls\|\az\|_{L^q(S_i(B))}\dint_{0}^\fz e^{-u}
\lf\{ \dint_{0}^{\frac{4u(M+1)}{[2^i(2^{j-1}-2^k)r_B]^2}
\frac{(2^{i+j}r_B)^2}{1+u}}e^{-s}
\lf[\frac{r_B^2 s(1+u)}{(2^{i+j}r_B)^2}\r]^{2M}
\,\frac{ds}{s}\r\}^{\frac{1}{2}}\,du\\
&&\ls 2^{2(i+j)M}
\|\az\|_{L^q(S_i(B))} \dint_{0}^\fz (1+u)^{M}e^{-u}
\lf\{ \dint_{0}^1s^{2M}e^{-s}
\,\frac{ds}{s}\r\}^{\frac{1}{2}}\,du \\
&&\ls2^{2(i+j)M}
\|\az\|_{L^q(S_i(B))},
\end{eqnarray*}
which, together with \eqref{c5.16}, \eqref{c5.17}, \eqref{c5.19}
and Definition \ref{d4.2}, implies that,
when $j\ge 5$,
\begin{eqnarray*}
\hs\hs\wz{\rm{I}}_{i,\,j}&&\ls 2^{-(i+j)n(\frac{p_1}{q}-\wz q)}
\|\chi_B\|_{L^\fai(\rn)}^{p_1} |B|^{-\frac{p_1}{q}}
\fai\lf(B,\,\frac{|\lz|}{\|\chi_B\|_{L^\fai(\rn)}}\r)(\mathcal{A}_{i,\,j})^{p_1}\\
&&\ls 2^{-(i+j)p_1[2M+n(\frac{p_1}{q}-\wz q)]}
\|\chi_B\|_{L^\fai(\rn)}^{p_1} |B|^{-\frac{p_1}{q}}
\fai\lf(B,\,\frac{|\lz|}{\|\chi_B\|_{L^\fai(\rn)}}\r)
\|\az\|_{L^q(S_i(B))}^{p_1}\\
&&\ls 2^{-(i+j)p_1[2M+n(\frac{p_1}{q}-\wz q)]}
2^{-i\epsilon p_1}
\fai\lf(B,\,\frac{|\lz|}{\|\chi_B\|_{L^\fai(\rn)}}\r).
\end{eqnarray*}

Similar to the estimates for $\wz{\rm{I}}_{i,\,j}$, we see that
\begin{eqnarray*}
\wz{\rm{II}}_{i,\,j}\ls 2^{-(i+j)p_2[2M+n(\frac{p_2}{q}-\wz q)]}
2^{-i\epsilon p_2}
\fai\lf(B,\,\frac{|\lz|}{\|\chi_B\|_{L^\fai(\rn)}}\r),
\end{eqnarray*}
which, combining with \eqref{c5.14}, implies that
\begin{eqnarray}\label{c5.20}
\mathrm{I}\ls \lf(B,\,\frac{|\lz|}{\|\chi_B\|_{L^\fai(\rn)}}\r).
\end{eqnarray}
Also, by following the same way as the estimates for
$\mathrm{I}$, we know that
$\mathrm{J}\ls(B,\,\frac{|\lz|}{\|\chi_B\|_{L^\fai(\rn)}})$,
which, together with \eqref{c5.13} and \eqref{c5.20}, shows
that \eqref{c5.12} holds true.

The proof for the equivalence of $H_{\fai,\,S_h}(\rn)$ and $H_{\fai,\,L}(\rn)$
is similar. We omit the details here.
This finishes the proof of Proposition \ref{cp5.4}.
\end{proof}

Now, we state the maximal function characterizations of $H_{\fai,\,L}(\rn)$
as follows.

\begin{theorem}\label{ct5.5}
Let $\fai$ and $L$ be respectively as in Definition \ref{d2.2} and \eqref{c5.1x},
$\mathcal{R}_h$, $\mathcal{R}_P$, $\mathcal{N}_h$ and $\mathcal{N}_P$
respectively as in \eqref{c5.2}, \eqref{c5.3}, \eqref{c5.4} and \eqref{c5.5}.
Assume further that $\fai\in\rh_{(\min\{p_+(L),\,q_+(L)\}/I(\fai))'}(\cx)$,
where $I(\fai)$ is  as in \eqref{2.4}. Then $H_{\fai,\,\mathcal{R}_h}(\rn)$,
$H_{\fai,\,\mathcal{R}_P}(\rn)$, $H_{\fai,\,\mathcal{N}_h}(\rn)$,
$H_{\fai,\,\mathcal{N}_P}(\rn)$ and $H_{\fai,\,L}(\rn)$ coincide
with equivalent quasi-norms.
\end{theorem}

\begin{remark}\label{cr5.1}
Theorem \ref{ct5.5} completely covers
\cite[Theorem 5.2 and Corollary 5.1]{jy10} by taking $\fai$ as in \eqref{1.1}
with $w\equiv1$ and $\Phi$ concave.
\end{remark}

To prove Theorem \ref{ct5.5}, we need a good-$\lz$ inequality
concerning the \emph{non-tangential maximal function}
and the \emph{truncated Lusin-area function} associated with the heat semigroup.
More precisely, let $\az\in(0,\,\fz)$ and $0<\epsilon<R<\fz$. For all
$f\in L^2(\rn)$ and $x\in\rn$, the \emph{truncated Lusin-area function}
$S_h^{\epsilon,\,R,\,\az}$, associated with the heat semigroup, is defined by setting,
\begin{eqnarray}\label{c5.22}
S^{\epsilon,\,R,\,\az}_h(f)(x):=\lf\{\int_{\Gamma_\az^{\epsilon,\,R}(x)}
\lf|t\nabla e^{-t^2 L}(f)(y)\r|^2\,\frac{dy\,dt}{t^{n+1}}\r\}^{\frac{1}{2}},
\end{eqnarray}
where $\Gamma_\az^{\epsilon,\,R}(x):=\{(y,t)\in \rn\times(\epsilon,\,R):\
|y-x|<\az t\}$. We have the following good-$\lz$ inequality.

\begin{lemma}\label{cl5.6}
Let $\fai$ and $L$ be respectively as in Definition \ref{d2.2} and \eqref{c5.1x}.
Assume that $\epsilon,\ R\in(0,\fz)$ with $\epsilon<R$. Then, there exist positive
constants $\epsilon_0$ and $C$, independent of $\epsilon$ and $R$, such that,
for all $\gz\in(0,\,1]$, $\lz,\,s\in(0,\,\fz)$ and $f\in L^2(\rn)$
satisfying $\|\mathcal{N}_h(f)\|_{L^\fai(\rn)}<\fz$,
\begin{eqnarray*}
&&\fai\lf(\lf\{x\in\rn:\  S_{h}^{\epsilon,\,R,\,\frac{1}{20}}(f)(x)>2\lz,\,
\mathcal{N}_h(f)(x)\le \gz \lz\r\},\,s\r)\\
&&\hs\le C \gz^{\epsilon_0}
\fai\lf(\lf\{x\in\rn:\  S_{h}^{\epsilon,\,R,\,\frac{1}{2}}(f)(x)>\lz\r\},\,s\r).
\end{eqnarray*}
\end{lemma}

\begin{proof}
Lemma \ref{cl5.6} can be proved by using the same method as in the proof of
\cite[Lemma 3.3]{yys2}, where a good-$\lz$
inequality was established in the setting of the strongly Lipschitz domain of $\rn$.
In the present situation, the proof is more simple, since we do not need to
take care of the boundary condition and the diameter of the domain.
Here, in order to avoid redundancy, we only give an outline for the proof
of Lemma \ref{cl5.6}.   Let
$$O:=\{x\in\rn:\ S_h^{\epsilon,\,R,\,\frac{1}{20}}
(f)(x)>\fz\}.$$
From the $L^2(\rn)$-boundedness of $S_h$, we deduce that
$|O|<\fz$. By using Whitney's covering lemma, we see that
there exists a family $\{Q_j\}_j$ of dyadic cubes, with the lengths $\{l_j\}_j$,
satisfying that
\begin{enumerate}
\item[(i)] $O=\bigcup_{j} Q_j$ and $\{Q_j\}_j$ are disjoint;
\item[(ii)] $2Q_j\subset O$ and $4Q_j\cap O^\complement\neq\emptyset$.
\end{enumerate}
By this, to show the desired conclusion of Lemma \ref{cl5.6},
we only need to prove that, for all $j$,
\begin{eqnarray}\label{c5.23}
\hs\fai\lf(\lf\{x\in Q_j:\  S_{h}^{\epsilon,\,R,\,\frac{1}{20}}(f)(x)>2\lz,\,
\mathcal{N}_h(f)(x)\le \gz \lz\r\},\,s\r)\ls \gz^{\epsilon_0} \fai\lf(Q_j,\,s\r).
\end{eqnarray}

Moreover, for all $x\in Q_j$ and $x_j\in 4Q_j\cap O^\complement$, using
the fact that $\Gamma_{\frac{1}{20}}^{\max\{10l_j,\,\epsilon\}}(x)\subset
\Gamma_{\frac{1}{2}}^{\max\{10l_j,\,\epsilon\}}(x_j)$ and the definition of
$O$, we conclude that
\begin{eqnarray}\label{c5.24}
S_h^{\max\{10l_j,\,\epsilon\},\,
R,\,\frac{1}{20}}(f)(x)\le \lz.
\end{eqnarray}
Thus, if $\epsilon>10 l_j$, we know that, for all $x\in Q_j$,
$S_h^{\epsilon,\,R,\,\frac{1}{20}}(f)(x)\le \lz$, which contracts with the condition
$S_{h}^{\epsilon,\,R,\,\frac{1}{20}}(f)(x)>2\lz$. This implies that $Q_j=\emptyset$.
Hence, \eqref{c5.23} holds in this case.
If $\epsilon<10l_j$, from the fact that, for all $x\in\rn$,
$S_h^{\epsilon,\,R,\,\frac{1}{20}}(f)(x)\le S_h^{\epsilon,\,10l_j,\,\frac{1}{20}}(f)(x)
+S_h^{10l_j,\,R,\,\frac{1}{20}}(f)(x)$, \eqref{c5.24} and Lemma \ref{l2.4}(viii),
we deduce that, to prove \eqref{c5.23}, it suffices to prove that, for all $j$,
\begin{eqnarray}\label{c5.25}
\lf|\lf\{x\in Q_j\cap F:\  S_{h}^{\epsilon,\,R,\,\frac{1}{20}}(f)(x)>
\lz\r\}\r|\ls \gz^2|Q_j|,
\end{eqnarray}
where $F:=\{x\in\rn:\ \mathcal{N}_h(f)(x)\le \gz \lz\}$. We prove
\eqref{c5.25} by dividing two cases. If $\epsilon\ge 5l_j$, then similar to
the proof of \cite[Lemma 3.4]{yys2} (replace the strong Lipschitz domain $\boz$
and the non-tangential maximal function therein respectively by $\rn$ and
the present version of the non-tangential maximal function in \eqref{c5.4}), we conclude that,
for all $x\in\rn$,
$S_h^{\epsilon,\,R,\,\frac{1}{20}}(f)(x)\ls  \mathcal{N}_h(f)(x)$,
which, combining with the definition of $F$, shows that
\begin{eqnarray*}
\dint_{Q_j\cap F}\lf[S_h^{\epsilon,\,10l_j,\,\frac{1}{20}}(f)(x)\r]^2\,dx\ls
\dint_{Q_j\cap F}\lf[\mathcal{N}_h(f)(x)\r]^2\,dx\ls (\gz\lz)^2|Q_j|.
\end{eqnarray*}
This, together with Chebyshev's inequality, implies the validity of
\eqref{c5.25}. For the case when $\epsilon<5l_j$, let $G:=\{(y,t)\in \rn\times
(\epsilon,10l_j):\ d(y,Q_j\cap F)<\frac{t}{20}\}$. From \eqref{c5.22} and Fubini's
theorem, we infer that
\begin{eqnarray*}
\dint_{Q_j\cap F}\lf[S_h^{\epsilon,\,10l_j,\,\frac{1}{20}}(f)(x)\r]^2\,dx=
\iint_{G}t\lf|\nabla u(y,t)\r|^2\,dy\,dt,
\end{eqnarray*}
where $u(y,t):=e^{-t^2L}(f)(y)$.
To estimate $\int_{G}t|\nabla u(y,t)|^2\,dy\,dt$, we need to introduce
some smooth cut-off function defined on a fatter truncated cone.
More precisely, let
$$G_1:=\lf\{(y,t)\in \rn\times
\lf(\frac{\epsilon}{2},\,20l_j\r):\ d(y,Q_j\cap F)<\frac{t}{10}\r\}$$
and $\eta\in C_0^\fz(G_1)$
satisfying $\eta\equiv 1$ on $G$, $0\le \eta\le 1$ and
$\|\nabla \eta\|_{L^\fz(G_1)}\ls \frac{1}{t}$ (see also the proof of
\cite[Lemma 5.4]{hm09}). Then, by the ellipticity condition, we see that
\begin{eqnarray*}
\int_{G}t\lf|\nabla u(y,t)\r|^2\,dy\,dt&&\le \int_{G_1}t
\lf|\nabla u(y,t)\r|^2\eta(y,t)\,dy\,dt\\
&&\nonumber\ls \Re e \int_{G_1}t A(y)\nabla u(y,t)\overline{\nabla u(y,t)}
\eta(y,t)\,dy\,dt=:\Re e \mathrm{I}.
\end{eqnarray*}
Using Leibniz's rule, the definition of $L$ and the fact that $\pat_t u(y,\,t)=
-2tLu(y,\,t)$, we know that
\begin{eqnarray*}
\mathrm{I}&&=\int_{G_1}t A(y)\nabla u(y,t)\overline{\nabla (\eta u)(y,t)}\,dy\,dt
-\int_{G_1}t A(y)\nabla u(y,t)\overline{ u(y,t)\nabla\eta(y,t)}\,dy\,dt\\
&&=\int_{G_1}t L u(y,t)\overline{(\eta u)(y,t)}\,dy\,dt
-\int_{G_1}t A(y)\nabla u(y,t)\overline{ u(y,t)\nabla\eta(y,t)}\,dy\,dt\\
&&=-\frac{1}{2}\int_{G_1}\pat_t u(y,t)\overline{ (\eta u)(y,t)}\,dy\,dt
-\int_{G_1}t A(y)\nabla u(y,t)\overline{ u(y,t)\nabla\eta(y,t)}\,dy\,dt\\
&&=:-\frac{1}{2}\mathrm{I}_1-\mathrm{J}.
\end{eqnarray*}
Thus, from the fact that $\pat_t(\lf|u(y,t)\r|^2)=
2\Re e(\pat_tu(y,t))\overline{u(y,t)}$, the integral by parts and $\eta\in
C_0^\fz(G_1)$, we deduce that
\begin{eqnarray*}
\Re e \mathrm{I}&&=-\Re e\frac{1}{2}\mathrm{I}_1-\Re e\mathrm{J}\\
&&=-\frac{1}{4}\int_{G_1} \pat_{t}\lf(| u(y,t)|^2\r) \eta(y,t)\,dy\,dt
-\Re e\int_{G_1}t A(y)\nabla u(y,t)\overline{ u(y,t)\nabla\eta(y,t)}\,dy\,dt\\
&&=\frac{1}{4}\int_{G_1\setminus G} | u(y,t)|^2 \pat_{t}\eta(y,t)\,dy\,dt
-\Re e\int_{G_1\setminus G}t A(y)\nabla u(y,t)\overline{ u(y,t)\nabla\eta(y,t)}\,dy\,dt\\
&&=:\mathrm{I}_2+\mathrm{I}_3.
\end{eqnarray*}
The remainding estimates for $\mathrm{I}_2$ and $\mathrm{I}_3$
are obtained by using a decomposition of the set $G_1\setminus G$,
the properties of the cut-off function $\eta$, Besicovith's covering lemma
and a Caccioppoli's inequality (see, for example, \cite[(3.29)-(3.33)]{yys2} for
the detail calculations).  We then conclude that
$\mathrm{I}_2+\mathrm{I}_3\ls (\gz \lz)^2 |Q_k|$.
This, together with an application of Chebyshev's inequality, implies the validity of
\eqref{c5.25}, which completes the proof of Lemma \ref{cl5.6}.
 \end{proof}

With the help of Lemma \ref{cl5.6}, we now prove Theorem \ref{ct5.5}.

\begin{proof}[Proof of Theorem \ref{ct5.5}]
We first prove the following equivalence relationships
$$H_{\fai,\,\mathcal{R}_h}(\rn)=H_{\fai,\,\mathcal{N}_h}(\rn)=H_{\fai,\,L}(\rn).$$
The proof is divided into the following three steps.

{\bf Step 1.} $L^2(\rn)\cap H_{\fai,\,\mathcal{N}_h}(\rn)\subset
L^2(\rn)\cap H_{\fai,\,L}(\rn)$.
Let $f\in L^2(\rn)\cap H_{\fai,\,\mathcal{N}_h}(\rn)$. For any $0<\epsilon<R<\fz$
and  $\gz\in(0,1]$, by Lemma \ref{l2.1}(ii), Fubini's theorem and
Lemma \ref{cl5.6}, we conclude that
\begin{eqnarray*}
&&\int_{\rn}\fai\lf(x,\,S_h^{\epsilon,\,R,\,\frac{1}{20}}(f)(x)\r)\,dx\\
&&\hs\sim\int_0^\fz \fai\lf(\lf\{x\in\rn:\ S_h^{\epsilon,\,R,\,\frac{1}{20}}
(f)(x)>t\r\},\,t\r)\frac{dt}{t}\\
&&\hs\sim\int_0^\fz \fai\lf(\lf\{x\in\rn:\ S_h^{\epsilon,
\,R,\,\frac{1}{20}}(f)(x)>t,\,\mathcal{N}_h(f)(x)\le \gz t\r\},\,t\r)\frac{dt}{t}\\
&&\hs\hs+
\dint_0^\fz \fai\lf(\lf\{x\in\rn:\ \mathcal{N}_h(f)(x)>\gz t\r\},\,
t\r)\frac{dt}{t}\\
&&\hs\ls\gz^{\epsilon_0}\dint_0^\fz \fai\lf(\lf\{x\in\rn:\ S_h^{\epsilon,
\,R,\,\frac{1}{2}}(f)(x)>\frac{t}{2}\r\},\,t\r)\frac{dt}{t}\\
&&\hs\hs+
\dint_0^\fz \fai\lf(\lf\{x\in\rn:\ \mathcal{N}_h(f)(x)>\gz t\r\},\,
t\r)\frac{dt}{t}.
\end{eqnarray*}
By the change of variables and the fact that $\fai$ is of uniformly upper type $1$,
we further see that
\begin{eqnarray}\label{c5.26}
&&\int_{\rn}\fai\lf(x,\,S_h^{\epsilon,\,R,\,\frac{1}{20}}
(f)(x)\r)\,dx\\
&&\nonumber\hs\ls\gz^{\epsilon_0}\int_0^\fz \fai
\lf(\lf\{x\in\rn:\ S_h^{\epsilon,
\,R,\,\frac{1}{2}}(f)(x)>t\r\},\,t\r)\frac{dt}{t}\\
&&\nonumber\hs\hs+\frac{1}{\gz}
\int_0^\fz \fai\lf(\lf\{x\in\rn:\ \mathcal{N}_h(f)(x)> t\r\},\,
t\r)\frac{dt}{t}\\
&&\nonumber\hs\ls\gz^{\epsilon_0}\dint_{\rn}\fai
\lf(x,\, S_h^{\epsilon,\,R,\,\frac{1}{2}}(f)(x)\r)\,dx+\frac{1}{\gz}
\int_{\rn} \fai\lf(x,\,\mathcal{N}_h(f)(x)\r)\,dx.
\end{eqnarray}
Moreover, using an argument similar to that used in the proof of \cite[Lemma 7.7]{yys4},
we conclude that, for all $0<\az<\bz<\fz$ and $s\in(0,\,\fz)$,
\begin{eqnarray*}
\int_{\rn}\fai\lf(x,S_h^{\epsilon,\,R,\,\az}(f)(x)\r)\,dx\sim\int_{\rn}
\fai\lf(x, S_h^{\epsilon,\,R,\,\bz}(f)(x)\r)\,dx.
\end{eqnarray*}
From this and \eqref{c5.26} with $\gz$ sufficient small, it follows that
\begin{eqnarray*}
\int_{\rn}\fai\lf(x,\,S_h^{\epsilon,\,R}
(f)(x)\r)\,dx\ls \int_{\rn} \fai\lf(x,\,\mathcal{N}_h(f)(x)\r)\,dx.
\end{eqnarray*}
Letting $\epz\to 0$ and $R\to \fz$, we immediately know that
\begin{eqnarray*}
\int_{\rn}\fai\lf(x,\,S_h
(f)(x)\r)\,dx\ls \int_{\rn} \fai\lf(x,\,\mathcal{N}_h(f)(x)\r)\,dx,
\end{eqnarray*}
which implies that $\|f\|_{H_{\fai,\,S_h}(\rn)}\ls \|f\|_{H_{\fai,\,\cn_h}(\rn)}$.
Thus, $L^2(\rn)\cap H_{\fai,\,\mathcal{N}_h}(\rn)\subset L^2(\rn)
\cap H_{\fai,\,L}(\rn)$.

{\bf Step 2.} $L^2(\rn)\cap H_{\fai,\,\mathcal{R}_h}(\rn)\subset
L^2(\rn)\cap H_{\fai,\,\mathcal{N}_h}(\rn)$.
Let $f\in L^2(\rn)\cap H_{\fai,\,\mathcal{R}_h}(\rn)$.
By their definitions (see \eqref{c5.2} and \eqref{c5.3}), we see that
$\mathcal{N}^{\frac{1}{2}}_h(f)\le \mathcal{R}_h(f)$. Moreover,
similar to \cite[Lemma 5.3]{jy10}, we conclude that, for any $0<\az<\bz<\fz$,
\begin{eqnarray*}
\lf\|\mathcal{N}^{\az}_h(f)\r\|_{L^\fai(\rn)}\sim
\lf\|\mathcal{N}^{\bz}_h(f)\r\|_{L^\fai(\rn)},
\end{eqnarray*}
which immediately implies that
\begin{eqnarray*}
\lf\|\mathcal{N}_h(f)\r\|_{L^\fai(\rn)}\sim
\lf\|\mathcal{N}^{\frac{1}{2}}_h(f)\r\|_{L^\fai(\rn)}\ls
\lf\|\mathcal{R}_h(f)\r\|_{L^\fai(\rn)}.
\end{eqnarray*}
This establishes the inclusion $L^2(\rn)\cap
H_{\fai,\,\mathcal{R}_h}(\rn)\subset L^2(\rn)\cap
H_{\fai,\,\mathcal{N}_h}(\rn)$.

{\bf Step 3.} $L^2(\rn)\cap H_{\fai,\,L}(\rn)\subset L^2(\rn)
\cap H_{\fai,\,\mathcal{R}_h}(\rn)$. Let $f\in L^2(\rn)\cap H_{\fai,\,L}(\rn)$.
By Theorem \ref{t4.2} and Lemma \ref{cl4.5}, it suffices to
prove that,  for any $\lz\in\cc$ and $(\fai,\,q,\,M,\,\epsilon)_L$-molecule $\az$
associated with the ball $B$,
\begin{eqnarray}\label{c5.27}
\dint_{\rn}\fai\lf(x,\,\mathcal{R}_h(\lz \az)(x)\r)\,dx\ls
\fai\lf(B,\,\frac{|\lz|}{\|\chi_B\|_{L^\fai(\rn)}}\r),
\end{eqnarray}
where $\epsilon\in(0,\fz)$ and $M\in\nn$ can be chosen sufficient large. The estimate
\eqref{c5.27} can be proved by using Assumption {\rm (B)}; see, for example,
the proof of \eqref{4.8}. We omit the details.

From Steps 1 though 3, we deduce that
\begin{eqnarray*}
L^2(\rn)\cap H_{\fai,\,L}(\rn)=L^2(\rn)\cap H_{\fai,\,\cn_h}(\rn)=L^2(\rn)\cap
H_{\fai,\,\ccr_h}(\rn)
\end{eqnarray*}
with equivalent quasi-norms, which, together with the fact that
$$L^2(\rn)\cap H_{\fai,\,L}(\rn),\ L^2(\rn)\cap H_{\fai,\,\cn_h}(\rn)\ \text{and}\
L^2(\rn)\cap H_{\fai,\,\ccr_h}(\rn)$$
are, respectively, dense in $H_{\fai,\,L}(\rn)$, $H_{\fai,\,\cn_h}(\rn)$ and
$H_{\fai,\,\ccr_h}(\rn)$, and a density
argument, then implies that the spaces
$H_{\fai,\,L}(\rn)$, $H_{\fai,\,\cn_h}(\rn)$ and $H_{\fai,\,\ccr_h}(\rn)$
coincide with equivalent quasi-norms
The proof for the equivalence relationships that $H_{\fai,\,L}(\rn)=H_{\fai,\,
\mathcal{N}_P}(\rn)=H_{\fai,\,\mathcal{R}_P}(\rn)$ is similar,
we omit the details here. This finishes the proof of Theorem \ref{ct5.5}.
\end{proof}

\subsection{Boundedness of the Riesz transform $\nabla L^{-1/2}$}\label{cs5.2}

\hskip\parindent In this subsection, we study the boundedness of
the Riesz transform $\nabla L^{-1/2}$ associated with $L$ on $H_{\fai,\,L}(\rn)$
for $i(\fai)\in(\frac{n}{n+1},1]$, and the associated weak boundedness at the endpoint
$i(\fai)=\frac{n}{n+1}$. Our main result is as follows.

\begin{theorem}\label{ct5.7}

Let $\fai$ and $L$ be respectively as in Definition \ref{d2.2} and \eqref{c5.1x}.
Let $I(\fai)$, $i(\fai)$, $q(\fai)$ and $r(\fai)$ be, respectively, as in \eqref{2.4},
 \eqref{2.5}, \eqref{2.7} and \eqref{2.8}.

{\rm(i)} If $r(\fai)>(\frac{\min\{p_+(L),\,q_+(L)\}}
{I(\fai)})'$, then $\nabla L^{-1/2}$ is bounded from $H_{\fai,\,L}(\rn)$
to $L^\fai(\rn)$.

{\rm(ii)}  If $i(\fai)\in(\frac{n}{n+1},1]$,
$q(\fai)<\frac{ni(\fai)}{n+1}$ and $r(\fai)>(\frac{\min\{p_+(L),\,q_+(L)\}}
{q(\fai)})'$, then $\nabla L^{-1/2}$ is bounded from $H_{\fai,\,L}(\rn)$
to $H_\fai(\rn)$.
\end{theorem}

To prove Theorem \ref{ct5.7}, we need a new molecular characterization
of the classical Musielak-Orlicz-Hardy space $H_\fai(\rn)$.

Similar to \cite[Theorem 4.11]{hyy}, we have the following molecular
characterization of $H_\fai(\rn)$.

\begin{proposition}\label{cp5.10}
Let $\fai$ be as in Definition \ref{d2.2} and $q\in(1,\,\fz)$. Assume that $s\in\nn$
satisfying $s\ge \lfloor n[\frac{q(\fai)}{i(\fai)}-1] \rfloor$, $\epsilon\in (\max\{
n+s,\,n\frac{q(\fai)}{i(\fai)}\}-\frac{n}{q},\,\fz)$, $q(\fai)\in [1,\,q)$ and
$r(\fai)>\frac{q}{q-q(\fai)}$, where $i(\fai)$, $q(\fai)$ and
$r(\fai)$ are respectively as in \eqref{2.5}, \eqref{2.7} and \eqref{2.8}. Then,
$H_{\fai,\,\rm{mol}}^{q,\,s,\,\epsilon}(\rn)$ and $H_{\fai}(\rn)$ coincide
with equivalent quasi-norms.
\end{proposition}

The proof of Proposition \ref{cp5.10} is similar to that of \cite[Theorem 4.11]{hyy}.
We omit the details here. Observe that, in \cite[Theorem 4.11]{hyy},
the ranges of the exponents  may be different from those of Proposition
\ref{cp5.10}. More precisely, in \cite[Theorem 4.11]{hyy},
the authors want to relax the range of the Musielak-Orlicz function $\fai$,
by narrowing  the range of the exponent $q$.
However, in the present case, we need more wider range of $q$.

\begin{proof}[Proof of Theorem \ref{ct5.7}]
The proof of (i) depends on Theorem \ref{t4.1} and Lemma
\ref{cl4.5}. We only need to show that, for any $\lz\in\cc$ and $(\fai,\,q,\,M,\,
\epsilon)_L$-molecule $\az$ (with $M$ and $\epsilon$ large enough) associated with the ball
$B$,
\begin{eqnarray}\label{c5.28}
\dint_{\rn}\fai\lf(x,\,\nabla L^{-1/2}(\lz \az)(x)\r)\,dx\ls
\fai\lf(B,\,\frac{|\lz|}{\|\chi_B\|_{L^\fai(\rn)}}\r).
\end{eqnarray}
Using the $L^q(\rn)$-boundedness of $\nabla L^{-1/2}$ for all $q\in
(p_-(L),\,q_+(L))$ and the following off-diagonal estimates that
\begin{eqnarray*}
\lf\|\nabla L^{-1/2}\lf(I-e^{-tL}\r)^Mf\r\|_{L^q(F)}\ls
\lf\{\frac{t}{[d(E,\,F)]^2}\r\}^M\|f\|_{L^q(E)}
\end{eqnarray*}
and
\begin{eqnarray*}
\lf\|\nabla L^{-1/2}\lf(tLe^{-tL}\r)^Mf\r\|_{L^q(F)}\ls
\lf\{\frac{t}{[d(E,\,F)]^2}\r\}^M\|f\|_{L^q(E)}
\end{eqnarray*}
for closed sets $E,\,F\subset\rn$ with $d(E,F)>0$, we conclude
\eqref{c5.28} by using the same method as in \eqref{c4.10}.
This shows (i).

To prove (ii), let $q\in(p_-(L),\,p_+(L))$. For any
$(\fai,\,q,\,M,\,\epsilon)_L$-molecule $\az$ associated with
$B$, similar to \cite[(7.34)]{yys4}, we infer that there
exists a large enough positive constant $\epsilon_0$ such that, for all $j\in\zz_+$,
\begin{eqnarray}\label{c5.x29}
\lf\|\nabla L^{-1/2}(\az)\r\|_{L^q(S_j(B))}\ls 2^{-j\epsilon_0}
|B|^{\frac{1}{q}}\|\chi_{B}\|_{L^\fai(\rn)}^{-1}.
\end{eqnarray}
Moreover, since $1\le q(\fai)<\frac{n}{n+1}i(\fai)$, we know $s:=\lfloor
n[\frac{q(\fai)}{i(\fai)}-1]\rfloor=0$, which, together with the fact that
(see, for example, the proof of \cite[Theorem 7.4]{jy10} when $\fai$ is
an Orlicz function)
\begin{eqnarray*}
\dint_{\rn}\nabla L^{-1/2}(\az)(x)\,dx=0,
\end{eqnarray*}
immediately implies that, for each $(\fai,\,q,\,M,\,\epsilon)_L$-molecule
$m$ associated with the ball $B$, $\nabla L^{-1/2}(\az)$ is a
$(\fai,\,q,\,0,\,\epsilon)$-molecule associated with the same ball $B$. This, together with
the assumptions $q(\fai)<\frac{n}{n+1}i(\fai)$, $r(\fai)>(\frac{\min\{p_+(L),\,q_+(L)\}}
{q(\fai)})'$ and Proposition \ref{cp5.10}, implies that $\nabla L^{-1/2}$ is bounded from
$H_{\fai,\,L}(\rn)$ to $H_\fai(\rn)$, which completes the proof of (ii) and hence Theorem
\ref{ct5.7}.
\end{proof}

Now, we establish the weak boundedness of $\nabla L^{-1/2}$ at the endpoint
$i(\fai)=\frac{n}{n+1}$. Before stating our conclusions, we first recall some necessary
definitions.

\begin{definition}\label{cd5.11}
Let $\fai$ be as in Definition \ref{d2.2}, $\psi\in \mathcal{S}(\rn)$
satisfying $\supp \psi \subset B(0,\,1)$ and $\int_{\rn} \psi(x)\,dx=1$.
The \emph{weak Musielak-Orlicz-Hardy space} $WH_\fai(\rn)$ is defined to be the set
$\{f\in S'(\rn):\ \|f\|_{WH_\fai(\rn)}<\fz\},$
where
\begin{eqnarray*}
&&\|f\|_{WH_\fai(\rn)}\\
&&\hs:=\lf\|\dsup_{t\in(0,\,\fz)} \psi_t\ast f\r\|_{WL^\fai(\rn)}\\
&&\hs:=\inf\lf\{\lz\in(0,\,\fz):\ \dsup_{\eta\in(0,\,\fz)} \fai\lf(
\lf\{x\in\rn:\ \dsup_{t\in(0,\,\fz)} |(\psi_t\ast f)(x)|>\eta\r\},\,
\frac{\eta}{\lz}\r)\le 1\r\}
\end{eqnarray*}
and, for all $t\in(0,\,\fz)$ and $x\in\rn$,
$\psi_t(x):=t^{-n}\psi(\frac{x}{t})$.
\end{definition}

\begin{theorem}\label{ct5.12}

Let $\fai$ and $L$ be respectively as in Definition \ref{d2.2} and \eqref{c5.1x}. Assume
that $i(\fai)$, $I(\fai)$, $q(\fai)$ and $r(\fai)$ are respectively as in \eqref{2.5},
\eqref{2.4}, \eqref{2.7} and \eqref{2.8}. If
$i(\fai)=\frac{n}{n+1}$ is attainable, $\fai\in\aa_1(\rn)$,
$I(\fai)\in(0,\,1)$ and $r(\fai)>(\frac{\min\{p_+(L),\,q_+(L)\}}{q(\fai)})'$,
then $\nabla L^{-1/2}$ is bounded from $H_{\fai,\,L}(\rn)$ to $WH_\fai(\rn)$.
\end{theorem}

\begin{remark}\label{cr5.2}
Theorem \ref{ct5.7} completely covers \cite[Theorems 7.1 and 7.4]{jy10}
by taking $\fai$ as in \eqref{1.1} with $w\equiv1$ and $\Phi$ concave.
Theorem \ref{ct5.12} completely covers \cite[Theorem 1.2]{cyy} by taking $\fai(x,t):=
t^{n/(n+1)}$ for all $x\in\rn$ and $t\in[0,\fz)$.
\end{remark}

To prove Theorem \ref{ct5.12}, we need the following superposition principle
of weak type estimates.

\begin{lemma}\label{cl5.13}
Let $\fai$ be as in Definition \ref{d2.2} satisfying $I(\fai)\in(0,\,1)$, where
$I(\fai)$ is as in \eqref{2.4}. Assume that $\{a_j\}_j$ is a sequence of measurable
functions and $\{\lz_j\}_j\subset \cc$ such that there exists a sequence
$\{B_j\}_j$ of balls, it holds that
\begin{eqnarray*}
\dsum_{j}\fai\lf(B_j,\,\frac{|\lz_j|}{\|\chi_{B_j}\|_{L^\fai(\rn)}}\r) <\fz.
\end{eqnarray*}
Moreover, if there exists a positive constant $C$ such that, for all $\eta\in(0,\,\fz)$
and $j\in\nn$,
\begin{eqnarray}\label{c5.29}
\fai\lf(\{x\in\rn:\ |\lz_j a_j(x)|>\eta\},\,\eta\r)\le C \fai\lf(B,\,\frac{|\lz_j|}
{\|\chi_{B_j}\|_{L^\fai(\rn)}}\r).
\end{eqnarray}
Then, there exists a positive constant $\wz C$ such that, for all $\eta\in(0,\,\fz)$,
\begin{eqnarray*}
\fai\lf(\lf\{x\in\rn:\ \sum_j|\lz_j||a_j(x)|>\eta\r\},\, \eta\r)\le
\wz C\sum_j\fai\lf(B_j,\,\frac{|\lz_j|}{\|\chi_{B_j}\|_{L^\fai(\rn)}}\r).
\end{eqnarray*}
\end{lemma}

\begin{proof}
For any given $\eta\in(0,\fz)$, let $E:=\cup_{j}\{x\in\rn:\ |\lz_j||a_j(x)|>\eta\}$.
From \eqref{c5.29},
we deduce that
\begin{eqnarray}\label{c5.30}
\quad\quad\fai\lf(E,\, \eta\r)\ls \dsum_j\fai\lf(\lf\{x\in\rn:\
|\lz_j||a_j(x)|>\eta\r\},\,\eta\r)\ls \dsum_j\fai\lf(B_j,\,\frac{|\lz_j|}
{\|\chi_{B_j}\|_{L^\fai(\rn)}}\r),
\end{eqnarray}
which is desired.
On the other hand, taking $p_1\in(I(\fai),\,1)$, then we know that $\fai$
is of uniformly upper type $p_1$. This, together with
Chebyshev's inequality and \eqref{c5.29}, implies that
\begin{eqnarray*}
\fai\lf(\lf\{E^{\complement}:\ \dsum_{j}|\lz_j||a_j(x)|>\eta\r\},\,
\eta\r)&&\ls \frac{1}{\eta}\dsum_j\dint_{\lf\{x\in\rn:\
|\lz_j||a_j(x)|\le \eta\r\}}|\lz_j a_j(x)|\fai(x,\,\eta)\,dx\\
&&\ls \frac{1}{\eta}\dsum_j \dint_0^\eta\fai(\lf\{x\in\rn:\
|\lz_j||a_j(x)|>\eta\r\},\,\eta)\,dt\\
&&\ls \frac{1}{\eta} \dint_0^\eta
\lf(\frac{\eta}{t}\r)^{p_1}\,dt \dsum_j \fai\lf(B_j,\,
\frac{|\lz_j|}{\|\chi_{B_j}\|_{L^\fai(\rn)}}\r)\\
&&\ls \dsum_j \fai\lf(B_j,\,
\frac{|\lz_j|}{\|\chi_{B_j}\|_{L^\fai(\rn)}}\r),
\end{eqnarray*}
which, together with \eqref{c5.30}, implies the desired estimates
and hence completes the proof of Lemma \ref{cl5.13}.
\end{proof}

We now turn to the proof of Theorem \ref{ct5.12}.

\begin{proof}[Proof of Theorem \ref{ct5.12}]
To prove this theorem, let $q\in(p_-(L),\,\min\{p_+(L),\,q_+(L)\})$.
We first claim that it suffices to show
that, for all $\lz\in\cc$ and each $(\fai,\,q,\,\epsilon,\,M)_L$-molecules $\az$
associated with the ball $B$ (with $\epsilon$
and $M$ large enough) and all $\eta\in(0,\,\fz)$,
\begin{eqnarray}\label{c5.31}
\hs\hs\hs\hs\fai\lf(\lf\{x\in\rn:\ \dsup_{t>0}\lf|\lf(\psi_t\ast \lf(\nabla L^{-1/2}(\lz \az)
\r)(x)\r)\r|>\eta\r\},\,\eta\r)\ls \fai\lf(B,\,\frac{|\lz|}{\|\chi_{B}\|_{L^\fai(\rn)}}\r).
\end{eqnarray}
Indeed, if \eqref{c5.31} holds true, then for all $f\in L^2(\rn)\cap H_{\fai,\,L}(\rn)$,
by Theorem \ref{t4.1}, we know that, there exist a sequence
$\{\lz_j\}_j\subset \cc$ and a sequence $\{\az_j\}_j$ of $(\fai,\,q,\,M,\,
\epsilon)_L$-molecules associated with the balls $\{B_j\}_j$ such that
$f=\sum_j \lz_j \az_j $ in $L^2(\rn)$ and $\|f\|_{H_{\fai,\,L}(\rn)}\sim
\Lambda(\{\lz_j\az_j\}_j)$. Moreover, by the assumption that $I(\fai)<1$,
the change of variable and Lemma \ref{cl5.13}, we infer that
\begin{eqnarray*}
&&\lf\|\nabla L^{-1/2}(f)\r\|_{WH_\fai(\rn)}\\
&&\hs=\inf\Bigg\{\lz\in(0,\,\fz):\ \\
&&\hs\hs\lf.\sup_{\eta\in(0,\,\fz)} \fai\lf(
\lf\{x\in\rn:\ \sup_{t\in(0,\,\fz)}\lf |\lf(\psi_t\ast\lf[\sum_{j} \lz_j
\nabla L^{-1/2}(\az_j)\r]\r)(x)\r|>\eta\r\},\,
\frac{\eta}{\lz}\r)\le 1\r\}\\
&&\hs=\inf\Bigg\{\lz\in(0,\,\fz):\ \\
&&\hs\hs\lf.\sup_{\wz \eta\in(0,\,\fz)} \fai\lf(
\lf\{x\in\rn:\ \sup_{t\in(0,\,\fz)}\lf |\lf(\psi_t\ast\lf[\sum_{j} \frac{\lz_j}{\lz}
\nabla L^{-1/2}(\az_j)\r]\r)(x)\r|>\wz\eta\r\},\,\wz\eta\r)\le 1\r\}\\
&&\hs\ls\inf\lf\{\lz\in(0,\,\fz):\ \sum_j \fai\lf(
B_j,\, \frac{|\lz_j|}{\lz\|\chi_{B_j}\|_{L^\fai(\rn)}}\r)\le 1\r\}
\sim\Lambda\lf(\lf\{\lz_j \az_j\r\}_j\r)\sim \|f\|_{H_{\fai,\,L}(\rn)},
\end{eqnarray*}
which, together with a density argument, implies that $\nabla L^{-1/2}$
is bounded from $H_{\fai,\,L}(\rn)$ to $WH_{\fai}(\rn)$. This shows
the claim.

Now, we turn to the proof of \eqref{c5.31}. Let $p_1\in[I(\fai),1)$ be a uniformly upper
type of $\fai$. Then
\begin{eqnarray*}
&&\fai\lf(\lf\{x\in 16B:\ \dsup_{t\in(0,\,\fz)} \lf|\lf(\psi_t\ast
\lf[\nabla L^{-1/2}(\lz \az)\r]\r)(x)\r|>\eta \r\},\,\eta\r)\\
&&\hs\ls \dint_{16B}\fai\lf(x,\, \dsup_{t\in(0,\,\fz)} \lf|\lf(\psi_t\ast
\lf[\nabla L^{-1/2}(\lz \az)\r]\r)(x)\r| \r)\,dx\\
&&\hs\ls \dint_{16B}\lf\{1+\lf[\dsup_{t\in(0,\,\fz)} \lf|\lf(\psi_t\ast\lf[\nabla L^{-1/2}(\az)
\r]\r)(x)\r|\|\chi_B\|_{L^\fai(\rn)}\r]^{p_1}\r\}\\
&&\hs\hs\times\fai\lf(x,\, \frac{|\lz|}{\|\chi_{B}\|_{L^\fai(\rn)}} \r)dx\\
&&\hs\ls \dint_{16B}\fai\lf(x,\,\frac{|\lz|}{\|\chi_B\|_{L^\fai(\rn)}}\r)\,dx\\
&&\hs\hs+\dint_{16B}\lf\{\dsup_{t\in(0,\,\fz)} \lf|\lf(\psi_t\ast
\lf[\nabla L^{-1/2} \az\r]\r)(x)\r|\|\chi_B\|_{L^\fai(\rn)}\r\}^{p_1}\fai\lf(x,\, \frac{|\lz|}
{\|\chi_{B}\|_{L^\fai(\rn)}} \r)dx\\
&&\hs=:\mathcal{A}+\mathcal{B}.
\end{eqnarray*}
Without loss of generality, we may only estimate $\mathcal{B}$, the estimate for $\ca$ being
similar and simpler.
Let $q\ge2$ and $q$ be as in
\eqref{rangeofq} such that \eqref{c4.11} holds true.
Let $\wz q\in (q(\fai),\fz)$. Then $\fai\in \rh_{(\frac{q}{p_1})'}(\rn)\cap
\mathbb{A}_{\wz q}(\rn)$. Moreover, from \cite[Theorem 4.2]{au07}, we know that
$\nabla L^{-1/2}$ is bounded on $L^q(\rn)$, which, together with H\"older's inequality
and the $L^q(\rn)$-boundedness of the Hardy-Littlewood maximal function $\mathcal{M}$,
implies that
\begin{eqnarray*}
\mathcal{B}&&\ls \|\chi_B\|_{L^\fai(\rn)}^{p_1} \lf\{\dint_{16B}
\lf[\mathcal{M} \lf(\nabla L^{-1/2}(\az)\r)(x)\r]^q\,dx\r\}^{\frac{p_1}{q}}\\
&&\hs\times
\lf\{\frac{1}{|16B|}\dint_{16B}\lf[\fai\lf(x,\, \frac{|\lz|}
{\|\chi_{B}\|_{L^\fai(\rn)}} \r)\r]^{(\frac{q}{p_1})'}\,dx\r\}^{\frac{1}
{(\frac{q}{p_1})'}}|B|^{\frac{1}{(\frac{q}{p_1})'}} \\
&&\ls\|\chi_B\|_{L^\fai(\rn)}^{p_1} \lf\{\dint_{16B}
\lf[\az(x)\r]^q\,dx\r\}^{\frac{p_1}{q}}\fai\lf(B,\, \frac{|\lz|}
{\|\chi_{B}\|_{L^\fai(\rn)}} \r)|B|^{-\frac{p_1}{q}}.
\end{eqnarray*}
Using Definition \ref{d4.2}, we further see that
$\mathcal{B}\ls \fai(B,\,\frac{|\lz|}{\|\chi_B\|_{L^\fai(\rn)}}).$
Thus, we have
\begin{eqnarray}\label{c5.32}
&&\fai\lf(\lf\{x\in 16B:\ \dsup_{t\in(0,\,\fz)} \lf|\lf(\psi_t\ast
\lf[\nabla L^{-1/2}(\lz \az)\r]\r)(x)\r|>\eta \r\},\,\eta\r)\\
&&\nonumber\hs\ls \lf(B,\,\frac{|\lz|}{\|\chi_B\|_{L^\fai(\rn)}}\r).
\end{eqnarray}

Now, we turn to the case that $x\in (16B)^{\complement}$. For all
$i\in\{5,\,6,\,\ldots\}$, let
\begin{eqnarray*}
\mathrm{I}_i:=\fai\lf(\lf\{x\in S_i(B):\ \dsup_{t\in(0,\,r_B)}
\lf|\psi_t\ast \lf[\nabla L^{-1/2}(\lz \az)\r](x)\r|>\frac{\eta}{2}\r\},\,
\eta\r)
\end{eqnarray*}
and
\begin{eqnarray*}
\mathrm{J}:=\fai\lf(\lf\{x\in (16B)^{\complement}:\ \dsup_{t\in[r_B,\,\fz)}
\lf|\psi_t\ast \lf(\nabla L^{-1/2}(\lz \az)\r)(x)\r|>\frac{\eta}{2}\r\},\,
\eta\r).
\end{eqnarray*}
Assume that $\wz S_i(B):=2^{i+1}B\setminus 2^{i-2}B$ and
$\wz {\wz S}_i(B):=2^{i+2}B\setminus 2^{i-3}B$. It is easy to see that,
for all $x\in S_i(B)$, $t\in(0,\,r_B)$ and $|y-x|<t$, it holds that $y\in \wz S_i(B)$.
Thus, similar to the estimate for $\mathcal{B}$, we conclude that
\begin{eqnarray*}
\mathrm{I}_i&&\ls \fai\lf( S_i(B),\,
\mathcal{M}\lf(\chi_{\wz {S}_i(B)}\nabla L^{-1/2}\lf(\lz \az\r)\r)(x)\r)\\
&&\ls\lf\{\dint_{\wz S_i(B)}\lf|\nabla L^{-1/2}(\az)(x)\r|^q\,dx
\r\}^{\frac{p_1}{q}} \lf|2^iB\r|^{-\frac{p_1}{q}} \lf\|\chi_{B}\r\|_{L^\fai(\rn)}^{p_1}
\fai\lf( 2^iB,\,\frac{|\lz|}{\|\chi_{B}\|_{L^\fai(\rn)}}\r),
\end{eqnarray*}
which, together with \eqref{c5.x29}, further implies that
\begin{eqnarray*}
\mathrm{I}_i&&\ls 2^{-i\epsilon_0 p_1}|B|^{\frac{p_1}{q}}\|\chi_B\|_{L^\fai(\rn)}^{-p_1}
\lf|2^iB\r|^{-\frac{p_1}{q}}\|\chi_B\|_{L^\fai(\rn)}^{p_1}
\fai\lf( 2^iB,\,\frac{|\lz|}{\|\chi_{B}\|_{L^\fai(\rn)}}\r)\\
&&\nonumber\ls2^{-i[\epsilon_0 p_1+\frac{p_1n}{q}-n\wz q]}\fai\lf(B,\,\frac{|\lz|}
{\|\chi_B\|_{L^\fai(\rn)}}\r),
\end{eqnarray*}
where $\epsilon_0\in(0,\fz)$ is a sufficiently large constant.
Thus, we conclude that
\begin{eqnarray}\label{c5.34}
\dsum_{i\in\nn}\mathrm{I}_i\ls \fai\lf(B,\,\frac{|\lz|}
{\|\chi_B\|_{L^\fai(\rn)}}\r).
\end{eqnarray}

We now estimate $\mathrm{J}$. For all $x\in\rn$, let
\begin{eqnarray*}
F_i(x):=\dsup_{t\in[r_B,\,\fz)}\lf|\dint_{S_i(B)} \frac{1}{t^n}
\lf[\psi\lf(\frac{x-y}{t}\r)-\psi\lf(\frac{x-x_B}{t}\r)\r]
\nabla L^{-1/2} \lf(\lz \az\r)(y)\,dy\r|.
\end{eqnarray*}
For all $j\in\nn$, $y\in S_i(B)$, $t\in [2^{j-1}r_B,\,2^jr_B)$, $|x-y|<t$ or
$|x-x_B|<t$, we have $|x-x_B|\le |x-y|+|y-x_B|<(2^{i}+2^j)r_B$. Thus,
$x\in (2^i+2^j)B$. This, together with the mean value theorem, H\"older's
inequality and \eqref{c5.x29}, implies that, there exists a sufficiently large
constant $\wz \epsilon_0$ such that
\begin{eqnarray*}
F_i(x)&&\ls \sup_{j\in\nn}\sup_{t\in [2^{j-1}r_B,\,2^{j}r_B)}
\chi_{(2^i+2^j)B}(x)\int_{S_i(B)}\frac{1}{t^n}\lf\|\nabla \psi\r\|_{L^\fz(\rn)}
\frac{|y-x_B|}{t}\\
&&\hs\times\lf|\nabla L^{-1/2} \lf(\lz \az\r)(y)\r|\,dy\\
&&\ls |\lz| \dsup_{j\in\nn}\dsup_{t\in [2^{j-1}r_B,\,2^{j}r_B)}
\chi_{(2^i+2^j)B}(x) \frac{2^ir_B}{(2^jr_B)^{n+1}}
\lf|2^iB\r|^{\frac{1}{q'}}\lf\|\nabla L^{-1/2}(\az)\r\|_{L^q(S_i(B))}\\
&&\ls |\lz| \dsup_{j\in\nn}\dsup_{t\in [2^{j-1}r_B,\,2^{j}r_B)}
\chi_{(2^i+2^j)B}(x) 2^{-j(n+1)}2^{-i\wz \epsilon_0}
\lf\|\chi_B\r\|_{L^\fai(\rn)}^{-1}\\
&&=:C_0 |\lz| \dsup_{j\in\nn}\dsup_{t\in [2^{j-1}r_B,\,2^{j}r_B)}
\chi_{(2^i+2^j)B}(x) 2^{-j(n+1)}2^{-i\wz \epsilon_0}
\lf\|\chi_B\r\|_{L^\fai(\rn)}^{-1},
\end{eqnarray*}
where $C_0$ is a positive constant independent of $i$ and $x$.
Let
$$j_0:=\max\lf\{j\in\nn:\ C_0 |\lz|2^{-j(n+1)}2^{-i\wz \epsilon_0}
\lf\|\chi_B\r\|_{L^\fai(\rn)}^{-1}>\frac{\eta}{2}\r\}.$$
We know that, for all $x\in \lf[\lf(2^i+2^{j_0}\r)B\r]^{\complement}$,
\begin{eqnarray*}
F_i(x)\le C_0|\lz|2^{-j(n+1)}2^{-i\wz \epsilon_0}
\lf\|\chi_B\r\|_{L^\fai(\rn)}^{-1}\le \frac{\eta}{2},
\end{eqnarray*}
which immediately implies that
\begin{eqnarray*}
\lf\{x\in (16B)^\complement:\ \lf|F_i(x)\r|>\frac{\eta}{2}\r\}\subset
\lf(2^i+2^{j_0}\r)B.
\end{eqnarray*}
Thus, from the assumption that $i(\fai)$ is attainable and $\fai\in\aa_1(\rn)$,
we infer that
\begin{eqnarray*}
&&\fai\lf(\lf\{x\in (16B)^\complement:\ F_i(x)>\frac{\eta}{2}
\r\},\,\eta\r)\\
&&\hs\ls \fai\lf(\lf(2^i+2^{j_0}\r)B,\,2^{-j_0(n+1)}2^{-i\wz \epsilon_0}
\frac{|\lz|}{\|\chi_B\|_{L^\fai(\rn)}}\r)\\
&&\hs\ls2^{-j_0(n+1)i(\fai)}2^{-i\wz\epsilon_0 i(\fai)} \lf(2^i+2^{j_0}\r)^{n}
\fai\lf(B,\,\frac{|\lz|}{\|\chi_B\|_{L^\fai(\rn)}}\r)\\
&&\hs\ls2^{-j_0[(n+1)i(\fai)-n]}2^{-i\wz\epsilon_0 i(\fai)-n}
\fai\lf(B,\,\frac{|\lz|}{\|\chi_B\|_{L^\fai(\rn)}}\r),
\end{eqnarray*}
which, together with $i(\fai)=\frac{n}{n+1}$, implies that
\begin{eqnarray*}
\fai\lf(\lf\{x\in (16B)^\complement:\ F_i(x)>\frac{\eta}{2}
\r\},\,\eta\r)\ls 2^{-i(\wz \epsilon_0 i(\fai)-n)}
\fai\lf(B,\,\frac{|\lz|}{\|\chi_B\|_{L^\fai(\rn)}}\r).
\end{eqnarray*}
By this, together with Proposition \ref{cp5.10} and the fact that $\int_\rn \nabla L^{-1/2}
(\lz)(x)\,dx=0$, we find that
\begin{eqnarray*}
\mathrm{J}&&\ls \fai\lf(\lf\{x\in (16B)^\complement :\ \dsum_{i=5}^\fz
\dsup_{t\in[r_B,\,\fz)}\lf|\dint_{S_i(B)}\frac{1}{t^n}\lf[\psi\lf(\frac
{x-y}{t}\r)-\psi\lf(\frac{x-x_B}{t}\r)\r]\r.\r.\r.\\
&&\hs\times\nabla L^{-1/2}\lf(\lz \az\r)(y)
\bigg|>\eta\bigg\},\,\eta\Bigg)\\
&&\ls \dsum_{i=5}^\fz2^{-i(\wz\epsilon_0 i(\fai)-n)}
\fai\lf(B,\,\frac{|\lz|}{\|\chi_B\|_{L^\fai(\rn)}}\r)\sim
\fai\lf(B,\,\frac{|\lz|}{\|\chi_B\|_{L^\fai(\rn)}}\r).
\end{eqnarray*}
This, combined with \eqref{c5.32} and \eqref{c5.34}, implies that
\eqref{c5.31} holds true, which completes the proof of Theorem \ref{ct5.12}.
\end{proof}

\section{The Musielak-Orlicz-Hardy space associated
with the \\
Schr\"odinger operator\label{s6}}

\hskip\parindent In this section, we establish several equivalent
characterizations of the Musielak-Orlicz-Hardy space
$H_{\fai,\,L}(\rn)$ associated with the Schr\"odinger operator
$L:=-\bdz+V$,  where $0\le V\in L^1_{\loc}(\rn)$, in terms  of
the Lusin-area function associated with the
Poisson semigroup of $L$, the non-tangential and the radial maximal
functions associated with the heat semigroup generated by $L$, and
the non-tangential and the radial maximal functions associated with
the Poisson semigroup generated by $L$. Moreover, we also
consider the boundedness of the associated Riesz transform
on $H_{\fai,\,L}(\rn)$.

Let
\begin{equation}\label{6.1}
L:=-\bdz+V
\end{equation}
be a Schr\"odinger operator, where $0\le V\in L^1_{\loc}(\rn)$.
Since $V$ is a nonnegative function, from the Feynman-Kac formula,
we deduce that the kernel of the semigroup $e^{-tL}$, $h_t$,
satisfies that, for all $t\in(0,\fz)$ and $x,\ y\in\rn$,
\begin{equation*}
0\le h_t(x,y)\le(4\pi t)^{-n/2}\exp\lf\{-\frac{|x-y|^2}{4t}\r\}.
\end{equation*}
Thus, $L$ satisfies Assumptions ${\rm(A)}$ and ${\rm(B)}$ with $k=1$, $p_L=1$ and $q_L=\fz$.
Moreover, $L$ satisfies Assumptions ${\rm (H_1)}$ and ${\rm (H_2)}$ as in Section \ref{cs4}.

For all $f\in L^2(\rn)$ and $x\in\rn$, define the \emph{Lusin-area
function} $S_P$ associated with the Poisson semigroup of $L$ by
\begin{equation*}
S_P(f)(x):=\lf\{\iint_{\bgz(x)}\lf|t\sqrt{L}e^{-t\sqrt{L}}f(y)\r|^2
\,\frac{dy\,dt}{t^{n+1}}\r\}^{1/2}.
\end{equation*}

Similar to Definition \ref{d4.1}, we introduce the space
$H_{\fai,\,S_P}(\rn)$ as follows.

\begin{definition}\label{d6.1}
Let $\fai$ be as in
Definition \ref{d2.2} and $L$ as in \eqref{6.1}. A function $f\in L^2(\rn)$
is said to be in $\wz{H}_{\fai,\,S_P}(\rn)$ if $S_P(f)\in L^{\fai}(\rn)$; moreover,
define
$$\|f\|_{H_{\fai,\,S_P}(\rn)}:=\|S_P(f)\|_{L^\fai(\rn)}:=
\inf\lf\{\lz\in(0,\fz):\
\int_{\rn}\fai\lf(x,\frac{S_P(f)(x)}{\lz}\r)\,dx\le1\r\}.$$
The $S_P$-\emph{adapted} \emph{Musielak-Orlicz-Hardy space $H_{\fai,\,S_P}(\rn)$}
is defined to be the completion of
$\wz{H}_{\fai,\,S_P}(\rn)$ with respect to the \emph{quasi-norm}
$\|\cdot\|_{H_{\fai,\,S_P}(\rn)}$.
\end{definition}

For any $f\in L^2(\rn)$ and $x\in\rn$, let
$$\cn_h(f)(x):=\sup_{y\in B(x,t),\,t\in(0,\fz)}
\lf|e^{-t^2L}(f)(y)\r|,\ \ \cn_P(f)(x):=\sup_{y\in B(x,
t),\,t\in(0,\fz)}\lf|e^{-t\sqrt{L}}(f)(y)\r|,
$$
$$\ccr_h(f)(x):=\sup_{t\in(0,\fz)}|e^{-t^2L}(f)(x)|\ \text{and}\
\ccr_P(f)(x):=\sup_{t\in(0,\fz)}|e^{-t\sqrt{L}}(f)(x)|.$$

\begin{definition}\label{d6.2}
Let $L$ and $\fai$ be as in \eqref{6.1} and Definition \ref{d2.2}, respectively.
A function $f\in L^2(\rn)$ is said to be in
$\wz{H}_{\fai,\,\cn_h}(\rn)$ if $\cn_h(f)\in L^\fai(\rn)$;
moreover, let
$\|f\|_{H_{\fai,\,\cn_h}(\rn)}:=\|\cn_h(f)\|_{L^\fai(\rn)}.$
The \emph{$\cn_h$-adapted Musielak-Orlicz-Hardy space $H_{\fai,\,\cn_h}(\rn)$}
is defined to be the completion of
$\wz{H}_{\fai,\,\cn_h}(\rn)$ with respect to the \emph{quasi-norm}
$\|\cdot\|_{H_{\fai,\,\cn_h}(\rn)}$.

The \emph{spaces $H_{\fai,\,\cn_P}(\rn)$, $H_{\fai,\,\ccr_h}(\rn)$}
and \emph{$H_{\fai,\,\ccr_P}(\rn)$} are respectively defined in a similar way.
\end{definition}

Now, we give the following equivalent characterizations of
$H_{\fai,\,L}(\rn)$ in terms of maximal functions associated with
$L$.

\begin{theorem}\label{t6.1}
Assume that $\fai$ and $L$ are as in Definition \ref{d6.2}. Then
$$H_{\fai,\,L}(\rn),\ H_{\fai,\,\cn_h}(\rn)\ H_{\fai,\,\cn_P}(\rn),\
H_{\fai,\,\ccr_h}(\rn),\ H_{\fai,\,\ccr_P}(\rn)\ \text{and}\ H_{\fai,\,S_P}(\rn)$$
coincide with  equivalent quasi-norms.
\end{theorem}

\begin{remark}\label{r6.1}
Theorem \ref{t6.1} generalizes \cite[Theorem 7.4]{yys4} by extending the range of
the considered weights. More precisely, The radial and non-tangential maximal
function characterizations of $H_{\fai,\,L}(\rn)$ were obtained in \cite[Theorem 7.4]{yys4}
under the assumption $\fai\in\rh_{2/[2-I(\fai)]}(\rn)$. However, in the above
Theorem \ref{t6.1}, the assumption $\fai\in\rh_{2/[2-I(\fai)]}(\rn)$ is not needed.
\end{remark}

\begin{proof}[Proof of Theorem \ref{t6.1}]
The proof of Theorem \ref{t6.1} is divided into the following six
steps.

\textbf{Step 1.} $\wz{H}_{\fai,\,L}(\rn)\subset
\wz{H}_{\fai,\,\cn_h}(\rn)$. Let $M$ and $q$ be as in Theorem \ref{ct4.3}.
By the proof of Theorem \ref{ct4.3}, we know that
$\wz{H}_{\fai,\,L}(\rn)$ and $\wz{H}^{M,\,q}_{\fai,\,L,\,\at}(\rn)$ coincide  with
equivalent quasi-norms. Thus, we only need to prove
$\wz{H}^{M,\,q}_{\fai,\,L,\,\at}(\rn)\subset\wz{H}_{\fai,\,\cn_h}(\rn)$.
To this end, similar to the
proof of \eqref{c4.10}, it suffices to show that, for any $\lz\in\cc$ and
$(\fai,\,q,\,M)_L$-atom $\az$ associated to the ball $B$,
\begin{equation*}
\int_{\rn}\fai\lf(x,\cn_h(\lz \az)(x)\r)\,dx\ls
\fai\lf(B,|\lz|\|\chi_B\|^{-1}_{L^\fai(\rn)}\r).
\end{equation*}
From the $L^q(\rn)$-boundedness of $\cn_h$, similar to the proof of
\eqref{c4.10}, it follows that the above estimate holds true. We omit the details here.

\textbf{Step 2.} $\wz{H}_{\fai,\,\cn_h}(\rn)\subset
\wz{H}_{\fai,\,\ccr_h}(\rn)$, which is deduced from the fact
that, for all $f\in L^2(\rn)$ and $x\in\rn$,
$\ccr_h(f)(x)\le\cn_h(f)(x)$.

\textbf{Step 3.} $\wz{H}_{\fai,\,\ccr_h}(\rn)\subset
\wz{H}_{\fai,\,\ccr_P}(\rn)$, whose proof is similar to that of Step 3 in the proof of
\cite[Theorem 7.4]{yys4}. We omit the details here.

\textbf{Step 4.} $\wz{H}_{\fai,\,\ccr_P}(\rn)\subset
\wz{H}_{\fai,\,\cn_P}(\rn)$, whose proof is similar to that of Step 4 in the proof of
\cite[Theorem 7.4]{yys4}, and hence we omit the details here.

\textbf{Step 5.} $\wz{H}_{\fai,\,\cn_P}(\rn)\subset
\wz{H}_{\fai,\,S_P}(\rn)$, whose proof is similar to that of
\cite[Proposition 7.6]{yys4}. We omit the details here.

\textbf{Step 6.} $\wz{H}_{\fai,\,S_P}(\rn)\subset
\wz{H}_{\fai,\,L}(\rn)$.

Let $f\in \wz{H}_{\fai,\,S_P}(\rn)$.
Then $t\sqrt{L}e^{-t\sqrt{L}}f\in T_\fai(\rr^{n+1}_+)$, which,
together with Proposition \ref{p4.1}(ii), implies that
$\pi_{L,\,M}(t\sqrt{L}e^{-t\sqrt{L}}f)\in H_{\fai,\,L}(\rn)\cap L^2(\rn)$.
Furthermore, from the $H_{\fz}$ functional calculus, we infer that
$$f=\frac{\wz{C}_{(M)}}{C_{(M)}}\pi_{L,\,M}
(t\sqrt{L}e^{-t\sqrt{L}}f)
$$
in $L^2(\rn)$, where $\wz{C}_{(M)}$ is a positive constant such
that $\wz{C}_{(M)}\int_0^\fz t^{2(M+1)}e^{-t^2}te^{-t}\,\frac{dt}{t}=1$ and
$C_{(M)}$ is as in \eqref{4.3}. This, combined with
$\pi_{L,\,M}(t\sqrt{L}e^{-t\sqrt{L}}f)\in H_{\fai,\,L}(\rn)$,
implies that $f\in H_{\fai,\,L}(\rn)$. Therefore, we know that
$\wz{H}_{\fai,\,S_P}(\rn)\subset \wz{H}_{\fai,\,L}(\rn)$.

From Steps 1 though 6, we deduce that
\begin{eqnarray*}
\wz{H}_{\fai,\,L}(\rn),\ \wz{H}_{\fai,\,\cn_h}(\rn),\ \wz{H}_{\fai,\,\ccr_h}(\rn),\
\wz{H}_{\fai,\,\ccr_P}(\rn),\ \wz{H}_{\fai,\,\cn_P}(\rn)\ \text{and}\
\wz{H}_{\fai,\,S_P}(\rn)
\end{eqnarray*}
coincide with equivalent quasi-norms, which, together with the fact that
$$\wz{H}_{\fai,\,L}(\rn),\ \wz{H}_{\fai,\,\cn_h}(\rn),\
\wz{H}_{\fai,\,\ccr_h}(\rn),\  \wz{H}_{\fai,\,\ccr_P}(\rn),\
\wz{H}_{\fai,\,\cn_P}(\rn)\ \text{and}\ \wz{H}_{\fai,\,S_P}(\rn)$$
are, respectively, dense in $H_{\fai,\,L}(\rn)$,
$H_{\fai,\,\cn_h}(\rn)$,
$H_{\fai,\,\ccr_h}(\rn)$, $H_{\fai,\,\ccr_P}(\rn)$,
$H_{\fai,\,\cn_P}(\rn)$ and $H_{\fai,\,S_P}(\rn)$, and a density
argument, then implies that the spaces $H_{\fai,\,L}(\rn)$,
$H_{\fai,\,\cn_h}(\rn)$, $H_{\fai,\,\ccr_h}(\rn)$,
$H_{\fai,\,\ccr_P}(\rn)$, $H_{\fai,\,\cn_P}(\rn)$
and $H_{\fai,\,S_P}(\rn)$ coincide with equivalent quasi-norms, which
completes the proof of Theorem \ref{t6.1}.
\end{proof}

From now on, we study the boundedness of $\nabla L^{-1/2}$ on $H_{\fai,\,L}(\rn)$.
Similar to Theorem \ref{ct5.7}, we have the following conclusions.

\begin{theorem}\label{t6.2}
Let $\fai$ and $L$ be as in Definition \ref{d2.2} and
\eqref{6.1}, respectively. Assume
that $\nabla L^{-1/2}$ is bounded on $L^r(\rn)$ for all $r\in(1,p_0)$ with
some $p_0\in(2,\fz)$. Let $i(\fai)$, $q(\fai)$ and
$r(\fai)$ be, respectively, as in \eqref{2.5}, \eqref{2.7} and \eqref{2.8}.

{\rm(i)} If $r(\fai)>(p_0/I(\fai))'$,
then $\nabla L^{-1/2}$ is bounded from $H_{\fai,\,L}(\rn)$ into $L^\fai(\rn)$.

{\rm(ii)} If $i(\fai)\in(\frac{n}{n+1},1]$,
$\frac{q(\fai)}{i(\fai)}\in(1,\frac{n+1}{n})$ and
$r(\fai)>(p_0/q(\fai))'$, then $\nabla L^{-1/2}$ is bounded from
$H_{\fai,\,L}(\rn)$ into $H_\fai(\rn)$.
\end{theorem}

\begin{proof}
The proof of Theorem \ref{t6.2}(i) is similar to that of Theorem \ref{ct5.7}(i).
We omit the details here. Now we prove (ii). Let $f\in H_{\fai,\,L}(\rn)\cap L^2(\rn)$
and $M\in\nn$ with $M>\frac{n}{2}(\frac{q(\fai)}{i(\fai)})$. Then there
exist $p_2\in(0,i(\fai))$ and $q_0\in(q(\fai),\fz)$ such that
$M>\frac{n}{2}(\frac{q_0}{p_2})$, $\fai$ is of uniformly
lower type $p_2$ and $\fai\in\aa_{q_0}(\rn)$. Moreover, by
Proposition \ref{ct4.3}, we know that there exist
$\{\lz_j\}_j\subset\cc$ and a sequence $\{a_j\}_j$ of
$(\fai,\,q,\,M)_L$-atoms such that $f=\sum_j\lz_ja_j$ in $L^2(\rn)$
and $\|f\|_{H_{\fai,\,L}(\rn)}\sim\|f\|_{H^{M,\,q}_{\fai,\,\at}(\rn)}$.
Moreover, we know that $\nabla L^{-1/2}(f)=\sum_j\lz_j\nabla L^{-1/2}(a_j)$ in
 $L^2(\rn)$.

Let $a$ be a $(\fai,\,q,\,M)_L$-atom associated with the ball $B$.
For $i\in\zz_+$, let $\chi_i:=\chi_{S_i(B)}$,
$\wz{\chi}_i := |S_i(B)|^{-1}\chi_i$, $m_i:=\int_{S_i (B)}\nabla
L^{-1/2}(a)(x)\,dx$ and $M_i:=\nabla L^{-1/2}(a)\chi_i-m_i
\wz{\chi}_i$. Then we have
\begin{equation}\label{6.10}
\nabla L^{-1/2}(a)=\sum_{i=0}^{\fz}M_i +\sum_{i=0}^{\fz}m_i
\wz{\chi}_i.
\end{equation}
For $j\in\zz_+$, let $N_j :=\sum_{i=j}^{\fz}m_i$. By an
argument similar to that used in the proof of \cite[Theorem 6.3]{jy11},
we know that $\int_{\rn}\nabla L^{-1/2}(a)(x)\,dx=0$, which, together with \eqref{6.10},
yields
\begin{equation*}
\nabla L^{-1/2}(a)=\sum_{i=0}^{\fz}M_i
+\sum_{i=0}^{\fz}N_{i+1}\lf(\wz{\chi}_{i+1} -\wz{\chi}_i\r).
\end{equation*}
Obviously, for all $i\in\zz_+$,
\begin{equation}\label{6.12}
\supp M_i\subset 2^{i+1}B\ \ \text{and} \ \ \int_{\rn}M_i
(x)\,dx=0.
\end{equation}

When $i\in\{0,\,\ldots,\,4\}$, by H\"older's inequality and the
$L^q(\rn)$-boundedness of $\nabla L^{-1/2}$, we conclude that
\begin{eqnarray}\label{6.13}
\hs\|M_i\|_{L^q(\rn)}&&\le\lf\{\int_{S_i(B)}|\nabla
L^{-1/2}a(x)|^q\,dx\r\}^{1/q}+\lf\{\int_{S_i(B)}|m_i
\wz{\chi}_i(x)|^q\,dx\r\}^{1/q}\\ \nonumber
&&\ls\|a\|_{L^q(\rn)}\ls|B|^{1/q}\|\chi_B\|^{-1}_{L^\fai(\rn)}.
\end{eqnarray}

Moreover, similar to \eqref{c5.x29}, we know that,
for all $i\in\nn$ with $i\ge5$,
\begin{eqnarray}\label{6.14}
\hs\|M_i\|_{L^q(\rn)}\ls\lf\|\nabla
L^{-1/2}a\r\|_{L^q(S_i(B))}\ls2^{-2Mi}
|B|^{1/q}\|\chi_B\|^{-1}_{L^\fai(\rn)}.
\end{eqnarray}

Furthermore, by $\fai\in\rh_{(p_0/q(\fai))'}(\rn)$, we
see that there exist $q\in(2,p_0)$ and $\widetilde{q}\in(q(\fai),\fz)$ such that
$\fai\in\aa_{\widetilde{q}}(\rn)\cap\rh_{(q/\widetilde{q})'}(\rn)$. From this,
H\"older's inequality, \eqref{6.13} and \eqref{6.14}, it follows
that, for all $i\in\zz_+$ and  $t\in(0,\fz)$,
\begin{eqnarray*}
&&\lf[\fai(2^{i+1}B,t)\r]^{-1}\int_{2^{i+1}B}|M_i(x)|^
{\widetilde{q}}\fai(x,t)\,dx\\ \nonumber
&&\hs\le\lf[\fai(2^{i+1}B,t)\r]^{-1}
\lf\{\int_{2^{i+1}B}|M_i(x)|^q\,dx\r\}^{\frac{\widetilde{q}}q}
\lf\{\int_{2^{i+1}B}[\fai(x,t)]^{(\frac{q}{\widetilde{q}})'}\,dx\r\}^{\frac{1}
{(q/\widetilde{q})'}}\\ \nonumber
&&\hs\ls2^{-2\widetilde{q}iM}|B|^{\frac{\widetilde{q}}{q}}
\|\chi_B\|^{-\widetilde{q}}_{L^\fai(\rn)}|2^{k+1}B|^{-\frac{\widetilde{q}}{q}},
\end{eqnarray*}
which implies that
\begin{eqnarray}\label{6.16}
\|M_i\|_{L^{\widetilde{q}}_\fai(2^{i+1}B)}\ls2^{-(2M+\frac nq)i}
\|\chi_B\|^{-1}_{L^\fai(B)}.
\end{eqnarray}
Then by \eqref{6.12} and \eqref{6.16}, we conclude that, for each
$i\in\zz_+$, $M_i$ is a constant multiple of a $(\fai,\,\widetilde{q},\,\,0)$-atom.
Moreover, from \eqref{6.14}, it follows that $\sum_{i=0}^{\fz}M_i$ converges in $L^q(\rn)$.

Now we estimate
$\|N_{i+1}(\wz{\chi}_{i+1}-\wz{\chi}_i)\|_{L^q(\rn)}$ with
$i\in\zz_+$. By H\"older's inequality and \eqref{6.13}, we see that
\begin{eqnarray}\label{6.17}
\lf\|N_{i+1}(\wz{\chi}_{i+1}-\wz{\chi}_i)\r\|_{L^q(\rn)}&&\ls|N_{i+1}||2^i
B|^{-\frac1{q'}}\ls\sum_{j=i+1}^\fz|m_{j+1}||2^i B|^{-\frac1{q'}}\\ \nonumber
&&\ls\sum_{j=i+1}^\fz|2^i B|^{-\frac1{q'}}|2^j B|^{\frac1{q'}}\|\nabla
L^{-1/2}a\|_{L^q(S_j(B))}\\ \nonumber
&&\ls2^{-2kM}|B|^{\frac1q}\|\chi_B\|^{-1}_{L^\fai(\rn)}.
\end{eqnarray}
From this and H\"older's inequality, similar to the proof of \eqref{6.16},
we deduce that, for all $i\in\zz_+$,
\begin{eqnarray}\label{6.18}
\lf\|N_{i+1}(\wz{\chi}_{i+1}-\wz{\chi}_i)\r\|_{L^{\widetilde{q}}_\fai(2^{i+1}B)}
\ls2^{-(2M+\frac nq)i} \|\chi_B\|^{-1}_{L^\fai(B)},
\end{eqnarray}
which, together with $\int_\rn
(\wz{\chi}_{i+1}(x)-\wz{\chi}_i(x))\,dx=0$ and
$\supp(\wz{\chi}_{i+1}-\wz{\chi}_i)\subset 2^{i+1}B$, implies that,
for each $i\in\zz_+$, $N_{i+1}(\wz{\chi}_{i+1}-\wz{\chi}_i)$ is a
constant multiple of a $(\fai,\,\widetilde{q},\,\,0)$-atom. Moreover,
by \eqref{6.17}, we see that
$\sum_{i=0}^\fz N_{i+1}(\wz{\chi}_{i+1}-\wz{\chi}_i)$ converges in $L^q(\rn)$.

Thus, \eqref{6.10} is an atomic decomposition of $\nabla
L^{-1/2}a$ and, furthermore, by \eqref{6.16}, \eqref{6.18}, the
uniformly lower type $p_2$ property of $\fai$ and
$M>\frac{n}{2}(\frac{q_0}{p_2}-\frac{1}{2})$, we know that
\begin{eqnarray}\label{6.19}
\qquad&&\sum_{i\in\zz_+}\fai\lf(2^{i+1}B,\|M_i\|_{L^{\widetilde{q}}_\fai(2^{i+1}B)}\r)
+\sum_{i\in\zz_+}\fai\lf(2^{i+1}B,
\|N_{i+1}(\wz{\chi}_{i+1}-\wz{\chi}_i)\|_{L^{\widetilde{q}}_\fai(2^{i+1}B)}\r)\\
\nonumber &&\hs\ls
\sum_{i\in\zz_+}\fai\lf(2^{i+1}B,2^{-(2M+\frac nq)i}
\|\chi_B\|^{-1}_{L^\fai(\rn)}\r)\ls\sum_{i\in\zz_+}
2^{-(2M+\frac nq)p_2}2^{knq_0}\ls1.
\end{eqnarray}

Replace $a$ by $a_j$ and, consequently, we then denote $M_i$,
$N_i$ and $\wz{\chi}_i$ in \eqref{6.10}, respectively, by
$M_{j,\,i}$, $N_{j,\,i}$ and $\wz{\chi}_{j,\,i}$. Similar to \eqref{6.10}, we know that
\begin{eqnarray*}
\nabla L^{-1/2}(f)&=&\sum_j\sum_{i=0}^{\fz}\lz_j
M_{j,\,i}+\sum_j\sum_{i=0}^{\fz} \lz_j N_{j,\,i+1}
(\wz{\chi}_{j,\,i+1}- \wz{\chi}_{j,\,i}),
\end{eqnarray*}
where, for each $j$ and $i$, $M_{j,\,i}$ and $N_{j,\,i+1}
(\wz{\chi}_{j,\,i+1}- \wz{\chi}_{j,\,i})$ are constant multiples of
$(\fai,\,\widetilde{q},\,\,0)$-atoms and the both summations hold
true in $L^q (\rn)$ and hence in $\cs'(\rn)$. Moreover, from \eqref{6.19} with $B$,
$M_{i}$, $N_{i+1} (\wz{\chi}_{i+1}- \wz{\chi}_{i})$ replaced, respectively, by
$B_j$, $M_{j,\,i}$, $N_{j,\,i+1} (\wz{\chi}_{j,\,i+1}-
\wz{\chi}_{j,\,i})$, we deduce that
$$\blz\lf(\{M_{j,\,i}\}_{j,\,i}\r)+
\blz\lf(\{N_{j,\,i+1} (\wz{\chi}_{j,\,i+1}-
\wz{\chi}_{j,\,i})\}_{j,\,i}\r) \ls\blz\lf(\{\lz_j a_j\}_j\r)\ls
\lf\|f\r\|_{H_{\fai,\,L}(\rn)}.
$$

By this, we conclude that $\|\nabla
L^{-1/2}(f)\|_{H_{\fai}(\rn)}\ls\|f\|_{H_{\fai,\,L}(\rn)}$,
which, together with the fact that $H_{\fai,\,L}(\rn)\cap
L^2(\rn)$ is dense in $H_{\fai,\,L}(\rn)$ and a density argument,
implies that $\nabla L^{-1/2}$ is bounded from $H_{\fai,\,L}(\rn)$
to $H_\fai(\rn)$. This finishes the proof of Theorem \ref{t6.3}.
\end{proof}

Moreover, similar to Theorem \ref{ct5.12}, for the Riesz transform $\nabla L^{-1/2}$ associated
with the Schr\"odinger operator $L$, we also have the following endpoint boundedness.

\begin{theorem}\label{t6.3}
Let $\fai$ and $L$ be respectively as in Definition \ref{d2.2} and \eqref{6.1}, and
$i(\fai)$, $I(\fai)$, $q(\fai)$ and $r(\fai)$ be respectively as in \eqref{2.5},
\eqref{2.4}, \eqref{2.7} and \eqref{2.8}. Assume
that $\nabla L^{-1/2}$ is bounded on $L^r(\rn)$ for all $r\in(1,p_0)$ with
some $p_0\in(2,\fz)$, and $\fai\in\aa_1(\rn)\cap\rh_{(p_0/q(\fai))'}(\rn)$.
If $i(\fai)=\frac{n}{n+1}$ is attainable and
$I(\fai)\in(0,\,1)$, then $\nabla L^{-1/2}$ is bounded from $H_{\fai,\,L}(\rn)$
to $WH_\fai(\rn)$.
\end{theorem}

\begin{remark}\label{r6.2}
(i) Theorem \ref{t6.2} improves \cite[Theorems 7.11 and 7.15]{yys4} by widening the range
of weights. More precisely, it was proved in \cite[Theorems 7.11 and 7.15]{yys4},
respectively, that $\nabla L^{-1/2}$ is bounded from $H_{\fai,\,L}(\rn)$ to $L^\fai(\rn)$
when $\fai\in\rh_{2/[2-I(\fai)]}(\rn)$, and from $H_{\fai,\,L}(\rn)$ to $H_\fai(\rn)$
when $\fai\in\rh_{2/[2-q(\fai)]}(\rn)$. From the assumption
$p_0\in(2,\fz)$, it follows that $(p_0/I(\fai))'<(2/I(\fai))'=2/[2-I(\fai)]$ and
$(p_0/q(\fai))'<(2/q(\fai))'=2/[2-q(\fai)]$,
which, together with Lemma \ref{l2.4}(iv), implies that
$$\rh_{(2/q(\fai))'}(\rn)\subset RH_{(p_0/q(\fai))'}(\rn).$$
Thus, Theorem \ref{t6.2} essentially improves \cite[Theorems 7.11 and 7.15]{yys4}.

(ii) Theorem \ref{t6.3} completely covers \cite[Corollary 1.1]{cyy} by taking
$\fai(x,t):=t^{n/(n+1)}$ for all $x\in\rn$ and $t\in[0,\fz)$.
\end{remark}

The proof of Theorem \ref{t6.3} is similar to that of Theorem \ref{ct5.12}. We omit the
details here.

\medskip

{\bf Acknowledgements.} The authors would like to thank the referees
for their careful reading and several valuable remarks which
improve the presentation of this article.

\bigskip

\noindent The Anh Bui

\medskip

\noindent Department of Mathematics, Macquarie University, NSW 2109, Australia and
Department of Mathematics, University of Pedagogy, Ho chi Minh city, Vietnam

\smallskip

\noindent{\it E-mails:} \texttt{the.bui@mq.ed.au} \ \ and \ \ \texttt{bt\_anh80@yahoo.com}

\bigskip

\noindent Jun Cao, Dachun Yang (Corresponding author) and Sibei Yang

\medskip

\noindent School of Mathematical Sciences, Beijing Normal
University, Laboratory of Mathematics and Complex Systems, Ministry
of Education, Beijing 100875, People's Republic of China

\smallskip

\noindent{\it E-mails:} \texttt{caojun1860@mail.bnu.edu.cn} (J. Cao)

\hspace{1.2cm}\texttt{dcyang@bnu.edu.cn} (D. Yang)

\hspace{1.2cm}\texttt{yangsibei@mail.bnu.edu.cn} (S. Yang)

\bigskip

\noindent Luong Dang Ky

\medskip

\noindent MAPMO-UMR 6628,\! D\'epartement de Math\'ematiques,
Universit\'e d'Orleans,\! 45067
Orl\'eans Cedex 2, France

\smallskip

\noindent{\it E-mails:} \texttt{dangky@math.cnrs.fr}

\end{document}